\documentclass[reqno,twoside]{amsart}
\usepackage{amsfonts}
\usepackage{amsmath,amssymb}
\usepackage{epic}
\usepackage{pstricks}
\usepackage{graphicx}
\usepackage{xcolor}
\usepackage{upgreek}
\usepackage[ruled,vlined,linesnumbered]{algorithm2e}
\usepackage[hidelinks]{hyperref}

\usepackage{a4wide,amsmath,amssymb,latexsym,amsthm}
\usepackage{eucal}
\usepackage{graphicx}

\newtheorem{theorem}{Theorem}[section]
\newtheorem{proposition}[theorem]{Proposition}
\newtheorem{remark}[theorem]{Remark}
\newtheorem{lemma}[theorem]{Lemma}
 
\newtheorem{definition}[theorem]{Definition}

\numberwithin{equation}{section}
\newtheorem{teo}{Theorem}[section]
\newtheorem{Obs}{Remark}[section]
\newtheorem{problem}{Problem}

\graphicspath{{./Imgs/}}

\newcommand{\be}{\begin{equation}}
\newcommand{\ee}{\end{equation}}
\newcommand{\ba}{\begin{eqnarray}}
\newcommand{\ea}{\end{eqnarray}}
\newcommand{\beq}{\begin{equation}}
\newcommand{\eeq}{\end{equation}}

\usepackage{color}

\usepackage[textsize=small]{todonotes}

\numberwithin{equation}{section}

\usepackage{color}
\usepackage{stmaryrd}

\subjclass[2010]{}

\begin{document}
\title[Robust Stackelberg controllability for the KS equation]{Robust Stackelberg controllability for the Kuramoto--Sivashinsky Equation}

\author{Cristhian  Montoya}
\address{C. Montoya,  Instituto de Ingenier\'{i}a Matem\'atica y Computacional, Pontificia Universidad Cat\'olica de Chile, Chile}
\email{cdmontoya85@gmail.com -- cdmontoy@mat.uc.cl}

\author{Louis Breton}
\address{Facultad de Ciencias,Universidad Nacional Aut\'onoma de M\'exico, M\'exico}
 \email{louis.breton@ciencias.unam.mx}

\thanks{The first author is supported by the Fondecyt Postdoctoral Grant N 3180100.}

\begin{abstract}
In this article the robust Stackelberg controllability (RSC) problem is studied for 
	 a nonlinear fourth--order parabolic equation, namely, the Kuramoto--Sivashinsky equation.
	When three external sources are acting into the system, the RSC problem consists essentially in combining two subproblems: 
	the first one is
	a saddle point problem among two sources. Such an sources are called the ``follower control'' and  its 
	associated ``disturbance signal''. This procedure corresponds to a robust control problem.
	The second one is a hierarchic control problem (Stackelberg strategy), which involves the third force, 
	so--called leader control.  
	The RSC problem establishes a simultaneous game for these forces in the sense that, the leader control has 
	as objective to verify a controllability property, while the follower control and perturbation solve a robust 
	control problem. In this paper the leader control obeys to the exact controllability to the trajectories.
	Additionally, iterative algorithms to approximate  the robust control problem 
	as well as the robust Stackelberg strategy for the nonlinear Kuramoto--Sivashinsky equation are developed
	and implemented.
	
\end{abstract}

\maketitle
\section{Main problems. Robust Stackelberg controllability} 
	The Stackelberg strategy is a concept from game theory which appears with the publication by 
	Heinrich Von Stackelberg in 1934 `` Market structure and equilibrium''.  It is a non--cooperative
	competition game with applications to economic processes that involves two--player with a hierarchic structure, 
	namely, the first player (called the leader) enforce its strategy on the other player, and then the second
	player (called the follower) reacts trying to win or optimize the answer to the leader 
	movement, see \cite{stackelberg1952theory,von2010market}. The previous sentences correspond to a general 
	notion on a Stackelberg strategy, which is applied in the context of hierarchic control for some models
	described by partial differential equations (PDEs).  
	
	On the other hand, the robustness in a control system is the sensitivity to the effects that are not considered in the 
	analysis and design such as disturbance signals and noise measurements. In other words, a system 
	is said to be robust when it is hardy, durable and resilient,  and also stable over the range of 
	parameter variations.  In this sense, one could think in the worst--case disturbance of the system, 
	and design a controller which is suited to handle even this extreme situation. Thus, the problem 
	of finding a robust control involves the problem of finding the worst-case disturbance in the spirit 
	of a non--cooperative game (when there is no cooperation between the controller and disturbance function), 
	that means from a mathematical point of view to reach a saddle point for the pair disturbance--controller.
	In the literature there are many works concerning robust control problems, see for instance the books 
	\cite{book2012Greenlinear,bookdullerud2013course,bookdragan2006mathematical,bookchristofides2002nonlinear} and its 
	references therein for a complete description on this subject.
		
	From a theoretical perspective, recent works have mixed the concept of robust control with a  Stackelberg strategy, and 
	applied it to semilinear and linear heat equations \cite{2018-robust-heat,2020victorlilinana}, and to 
	the Navier--Stokes system 
	\cite{2018-robust-ns}. This new idea in control theory is being abridged and called ``Robust Stackelberg controllability''
	(RSC), see Problem \ref{p3.robustandstackelberg} below. In the case of a semilinear heat equation 
	 \cite{2018-robust-heat}, the RSC problem used external forces acting into the system, where the leader control has as
	 constrain the controllability to trajectories. On the other hand, \cite{2020victorlilinana} solves  a RSC problem for the linear
	 heat equation by considering that the either the leader or follower control acts on a small part of the boundary.  In 
	 \cite{2020victorlilinana} the leader control satisfies the null controllability property. In the RSC problem for 
	 the Navier--Stokes system \cite{2018-robust-ns} all controls are external forces acting on the systems, the leader control 
	 has a local null controllability objective, while the perturbation and the follower control solve a robust control problem.  
	 However, these three works have three things in common: 1) they deal with systems whose main
	operator is a second--order operator (Laplace operator, Stokes operator), 2) independent of the configuration or 
	localization of forces (either interior or bounded), the property of the exact controllability to  the trajectories for the leader
	control remains open for nonlinear systems, and 3) as it can see, they do not present any numerical framework.
	
	In what follows we describe the main contributions of this work.
	\begin{enumerate}
	\item [1.] We solve the robust internal control problem for the nonlinear KS equation posed  on a bonded domain. 
		Our approach use central ideas from robust boundary control problem for the same equation \cite{2001TemamChangbing2}.
		To do that, several points related to regularity of solutions and to the existence of a saddle point are modified 
		and adapted.
		
		\noindent In the numerical context, to our knowledge, this paper contains the first numerical description concerning the 
		robustness process for the KS equation. Due to the high--order in space (i.e., fourth--order derivates), an 
		appropriate change of variable  will be used to implement low--order finite elements, more precisely, 
		$\mathbb{P}_1$--type  Lagrange elements, meanwhile, a $\theta$--scheme/Adams--Bashforth method is  created for the
		time discretization.  Thus, our method does not require a higher--order approach to the KS equation. Although this 
		paper  does not present an exhaustive numerical analysis of our method, since it 
		is far way of the main goals, several configurations to the  time--space discretization display good results for 
		the error (among the exact and numeric solution) in the $L^2$--norm and $L^\infty$--norm. Besides, from 
		the algorithms presented in  
		\cite{2000-bewleytemamziane,2002iterative-tachim} for the Navier--Stokes system, we propose new iterative schemes 
		of constructing the ascent and descent directions, and whose basis is the preconditioned nonlinear gradient conjugate
		method.
		
	\item [2.] Once we have obtained the robust pair, the robust stackelberg controllability (RSC) problem for the KS equation
		is studied.  The second theoretical contribution of our article is that, as far as we know, we use for the first time  
		the exact controllability to the trajectories for the leader control subject to a nonlinear system. 
		The main novelties  are new Carleman inequalities and its relationship with the 
		robustness parameters. Additionally, since the leader control obeys to the exact controllability to the
		trajectories
		and its formulation includes a coupled system of fourth--order equations, new algorithms based in regularization 
		techniques are introduced and implemented. Finally, we want to highlight the sensitivity in the robustness 
		parameters, the initial data, and also on the different subdomains for obtaining good results. Indeed, 
		numerical experiments show that non--cooperative relation among the leader control and follower might be 
		removed in some sense. 	  	
	\end{enumerate}

	\subsection{\normalsize{Main problems}}		
	In an abstract setting, the main problems to treat can be formulated as follows: let $(X,\langle\cdot,\cdot\rangle)$ 
	be an Hilbert space and let $(\mathcal{A},D(\mathcal{A}))$ be an unbounded operator in 
	$X$ such that $-\mathcal{A}$ generates an analytic semigroup in $X$. Let $(U,[\cdot,\cdot])$ be another Hilbert space and 
	for $i=1,2$,\, let $\mathcal{B}_i$ be bounded operators from $U$ into $D(\mathcal{A}^*)'$. Moreover, 
	let $\Omega$  be a nonempty bounded connected open subset of  $\mathbb{R}^d$ of class 
	$C^{\infty}$, $d\in \mathbb{N}$, and let $\omega$ be a (small) nonempty open subset of $\Omega$. Let $T>0$ be given. 
	We use the notation $Q:=\Omega\times(0,T)$,\, $\Sigma:=\partial\Omega\times(0,T)$.
	  
	Let us consider the non--homogeneous evolution problem 
	
	\begin{equation}\label{1.1.main_system.abstract}
    \left\{
    \begin{array}{lll}
    \begin{array}{llll}
    	u_{t}+\mathcal{A}u+\mathcal{N}u=h1_{\omega}+\mathcal{B}_1v+\mathcal{B}_2\psi  & \text{ in }& Q,\\
   		u(\cdot,0)=u_0(\cdot)& \text{ in }&\Omega,
    \end{array}
    \end{array}\right.
	\end{equation}
	where $\mathcal{N}$ is associated to the nonlinear part, and the functions $h,v,\,\psi$ belong to appropriate spaces. 
	Here, $1_{\omega}$ is the characteristic function of the set $\omega$. In \eqref{1.1.main_system.abstract} the interior
	forcing has been decomposed into a function $\psi$, called disturbance signal,  and two functions,
	$h$ and $v$. In our framework  $h=h(x,t)$ will be called the ``leader control'', 
    meanwhile $v=v(x,t)$  will be called the ``follower control''. To be precise, the interaction between such functions 
    and the problems that arise from them as well as the operators $\mathcal{B}_1,\, \mathcal{B}_2$ will be defined below 
    for every problem.
    In this abstract framework, the cost functional is given by
    \begin{equation}\label{eq.functional.intro}
    J_{r}(v,\psi;h):=\frac{1}{2}\iint\limits_{\mathcal{O}_d\times(0,T)}|u-u_{d}|^2         
    dxdt+\frac{1}{2}\Biggl(\ell^2\int\limits_0^T\|\mathcal{B}_1v\|_{X}^2 dt
    -\gamma^2\int\limits_0^T\|\mathcal{B}_2\psi\|_Y^2 dt \Biggr),
     \end{equation}
	where $X, Y$ are suitable Sobolev spaces, $\mathcal{O}_d$ is a nonempty open subset of $\Omega$,\,
	$u_d$ is a given function and $\ell,\gamma$ are positive constants. 
	The parameter $\ell$ can be interpreted as a measure of the ``cost'' of the control to the engineer. Thus, 
	when $\ell\to+\infty$, it corresponds to the ``expensive'' control, and results in $v\to 0$ in the minimization with respect
	to $v$ for the present problem. On the other hand, reduced values of $\ell$, corresponding to cheap control,
	reduce the increase in the cost functional upon the application of a control $v$. Similarly, the parameter 
	$\gamma$ can be interpreted as a measure of the price of the disturbance. The limit as
	$\gamma\to+\infty$ results in $\psi\to 0$ in the maximization with respect to $\psi$, and 
	reduced values of $\gamma$ decrease the cost functional upon the application of a disturbance $\psi$.
	
	\begin{problem}\label{p1.saddlepoint}
		\textit{Robust control.}  In \eqref{1.1.main_system.abstract}, $h\equiv 0$ and $\mathcal{B}_1,\, \mathcal{B}_2$ are 
		mapping from $L^2(\Omega)$ into itself. The robust internal control problem consists in finding a unique pair 
		$(\overline v,\overline\psi)\in L^2(Q)^2$ such that 
		\begin{equation}\label{saddlepoint.abstract}
    		J_{r}(\overline{v},\psi;0)\leq J_{r}(\overline{v},\overline{\psi};0)
    		\leq  J_{r}(v,\overline{\psi};0),\quad \forall (v,\psi)\in L^2(Q)^{2},
    	\end{equation}
		subject to the system \eqref{1.1.main_system.abstract}. 
	\end{problem}
	Before mentioning the other two problems that we deal in this paper, let  $\overline{u}$ be a solution of the 
	homogeneous equation:
	\begin{equation}\label{eq.target.u}
	\begin{cases}
		\overline{u}_{t}+\mathcal{A}\overline{u}+\mathcal{N}\overline{u}=0 & \text{in }Q,\\
		\overline{u}(\cdot,0)=\overline{u}_0& \text{in }\Omega.
	\end{cases}
	\end{equation}
	
 	\begin{problem}\label{p2.stackelberg}
		\textit{Stackelberg strategy.} In \eqref{1.1.main_system.abstract}, $\mathcal{B}_2\equiv 0$ and 
		$\mathcal{B}_1=1_{\mathcal{O}}$, where $\mathcal{O}$ is a small 
		open subset of $\Omega$  with $\mathcal{O}\cap\omega=\emptyset$. The hierarchic control problem consists in finding 
		a leader control $h\in L^2(0,T;L^2(\omega))$ and a unique follower control $v\in L^2(0,T;L^2(\mathcal{O}))$ minimizing 
		\eqref{eq.functional.intro}, and an associated solution $u$ to \eqref{1.1.main_system.abstract} verifying 
		$u(\cdot, T)=\overline{u}(\cdot, T)$ in $\Omega$, where $\overline{u}$ is solution of \eqref{eq.target.u}.
	\end{problem}
	
 	\begin{problem}\label{p3.robustandstackelberg}
		 \textit{Robust Stackelberg controllability.} For every fixed leader control $h$, solve 
   		 the saddle point problem for the system \eqref{1.1.main_system.abstract}, that is, to find the best control $v$ 
   		 in the presence of the disturbance $\psi$ which maximally spoils the  follower control for the
   		 system \eqref{1.1.main_system.abstract}. Once the saddle point has been identified for each leader control $h$, we 
   		 deal with the problem of finding  the control $h$ of minimal norm
   		 satisfying  constraints of exact controllability to the trajectories. More precisely, we look for a control  
   		 $\overline{h}$ such that 
    	\begin{equation}\label{cp_second_subproblem}
    		J(\overline{h})=\min\limits_{h}\frac{1}{2}\iint\limits_{\omega\times(0,T)}
    		|h|^2 dxdt,\quad \mbox{subject to the restriction}\quad u(\cdot, T)=\overline{u}(\cdot, T)\,\,\,\mbox{in}\,\,\Omega.
    	\end{equation}
	\end{problem}
	\begin{Obs}
		Note that Problem \ref{p2.stackelberg} is a particular case of Problem \ref{p3.robustandstackelberg} by  
		considering $\psi=0$ in \eqref{eq.functional.intro}. Thus, a Stackelberg strategy is a direct consequence of the 
		robust Stackelberg controllability, and therefore, in this article we will only treat Problem 
		\ref{p1.saddlepoint} and Problem \ref{p3.robustandstackelberg}. 
	\end{Obs}

\subsection{Main results}	
	A particular case of \eqref{1.1.main_system.abstract} corresponds to the Kuramoto--Sivashinsky (KS) equation, it
	is a fourth-order parabolic equation that serves as a model for phase turbulence in reaction-diffusion systems
	\cite{kuramoto1975formation, kuramoto1976persistent} 
	and also for modeling  the diffusive instabilities in a laminar flame 
	\cite{Sivashinsky1977,michelson1977nonlinear,2010purwinsdissipative,yanez2016analysis}. This equation obeys to an 
	one dimensional model, which 
	for our propose is given by
\begin{equation}\label{eq.ks}
	\begin{cases}
		u_{t}+u_{xxxx}+u_{xx}+uu_{x}=h1_{\omega}+v1_{\mathcal{O}}+\psi & \text{in }(0,1)\times(0,T)=:Q,\\
		u(0,t)=u(1,t)=u_{x}(0,t)=u_{x}(1,t)=0 & \text{on }(0,T),\\
		u(\cdot,0)=u_{0}(\cdot)& \text{in }(0,1),
	\end{cases}
\end{equation}	
	where $\omega$ and $\mathcal{O}$ are nonempty open subsets of $(0,1)$ such that  $\omega\cap\mathcal{O}=\emptyset$. 
	
	From a physical point of view, the term $u_{xx}$ is responsible for an instability at large scales; 
	the dissipative term $u_{xxxx}$ provides damping at small scales; and the non--linear term $uu_{x}$
	(which has the same form as that in the Burgers equation) 
	stabilizes by transferring energy between large and small scales. As mentioned, the terms on the right--hand side of 
	\eqref{eq.ks} are representing the leader control, the follower control and the disturbance signal, respectively.
	
	To our knowledge there is no results on  robust internal control problem for the KS system \eqref{eq.ks}.
	Thus, our paper fills this gap by using the functional 
	\eqref{eq.functional.intro} with $\mathcal{B}_1=1_{\mathcal{O}}$ into $L^2((0,1))$ and $\mathcal{B}_2=I$ onto $L^2((0,1))$.
	More precisely, the Problem \ref{p1.saddlepoint} is proved throughout the functional 
	 \begin{equation}\label{eq.functional.ks}
    	J_{r}(v,\psi;h):=\frac{1}{2}\iint\limits_{\mathcal{O}_d\times(0,T)}|u-u_{d}|^2         
    	dxdt+\frac{1}{2}\Biggl(\ell^2\iint\limits_{\mathcal{O}\times(0,T)}|v|^2 dx dt
   		 -\gamma^2\iint\limits_{Q}|\psi|^2 dxdt \Biggr).
    \end{equation}
	In the context of the robust control, the works
	\cite{2001TemamChangbing} and \cite{2001TemamChangbing2} proven robust boundary control problems for the KS equation.
	In these articles the cost functional is clearly different to 
	the presented for us in \eqref{eq.functional.intro}. On the other hand, the techniques of spatially dependent scaling 
	and static output feedback control are used in \cite{2007Sakthivel} and \cite{2003louoptimal}
	for obtaining  a robust controller design and  an optimal sensor placement for the KS equation, respectively. 
	
	Our first main result concerns the robust internal control problem for the KS equation. This is given in the 
	following theorem. 
	\begin{teo}\label{teo1.saddlepoint}
		Let $u_0\in H_0^2(0,1)$ and $h\in L^2(0,T;L^2(\omega))$ be fixed. Then, for $\gamma$ and $\ell$ sufficiently
		large, there exists  a unique saddle point $(\bar{v},\bar{\psi})\in L^2(0,T;L^2(\mathcal{O}))\times L^2(Q)$
		and $u=u(h,\bar{\psi},\bar{v})$ solution of \eqref{eq.ks} such that
		\[J_{r}(\psi,\overline{v};h)\leq J_{r}(\overline{v},\overline{\psi};h)
    		\leq  J_{r}(v,\overline{\psi};h),\quad \forall (v,\psi)\in L^2(0,T;L^2(\mathcal{O}))\times L^2(Q) .\]
	\end{teo}
	
	As mentioned, the second problem we aim to solve is to find the minimal norm control satisfying a controllability 
	to trajectories constrain. More precisely, let us fix a uncontrolled trajectory of system \eqref{eq.ks}, namely,
	a sufficiently regular solution to
	\begin{equation}\label{intro.target.u}
	\begin{cases}
		\overline{u}_{t}+\overline{u}_{xxxx}+\overline{u}_{xx}+\overline{u}\overline{u}_x
		=0 & \text{in }Q,\\
		\overline{u}(0,t)=\overline{u}(1,t)=\overline{u}_{x}(0,t)=\overline{u}_{x}(1,t)=0  & \text{on }(0,T),\\
		\overline{u}(\cdot,0)=\overline{u}_0& \text{in }(0,1).
	\end{cases}
	\end{equation}
	Thus, according to Problem \ref{p3.robustandstackelberg}, we look for a control $\overline{h}\in L^2(0,T;L^2(\omega))$ 
	satisfying  \eqref{cp_second_subproblem}. 
	
	In the case where $v=\psi=0$, system \eqref{eq.ks} is controllable to trajectories \cite{2011cerpamercado}. 
	Recently, for the case where the disturbance disappears, that is,  in \eqref{eq.functional.ks} $\psi\equiv 0$, 
	it is possible to deduce that the system \eqref{eq.ks} satisfies a Stackelberg strategy to trajectories 
	\cite{2018carreno-santos-stackelberg}.
	In contrast to \cite{2018carreno-santos-stackelberg}, this paper shows a different role among the forces $h, v$ and 
	$\psi$ in system \eqref{eq.ks}, and therefore, other optimization problems are carried out. In other works, this paper
	can be seen as an alternative development based in other Carleman estimates for solving Problem \ref{p2.stackelberg}.
	Actually, in our framework,  the theoretical solution to Problem \ref{p2.stackelberg}  is a consequence of 
	the simultaneous robust control and hierarchic control, see below Theorem \ref{teo2.RobustStackelbergContKS}.

	In order to present our second main result, let us define 	
	\begin{equation}\label{def.spaceZ}
		\mathcal{Z}:=C([0,T];H_0^2(0,1))\cap L^2(0,T;H^4(0,1))\cap L^\infty(0,T;W^{1,\infty}(0,1)).
	\end{equation}
	\begin{teo}\label{teo2.RobustStackelbergContKS}
		Assume that $\overline{u}\in \mathcal{Z}$ is the solution of \eqref{eq.target.u} 
		and $\omega\cap \mathcal{O}_d\neq \emptyset$. Then, for every $T>0$ and 
    	$\mathcal{O}\subset (0,1)$ open subset such that $\mathcal{O}\cap \omega=\emptyset$, 
		there exist $\gamma_0,\ell_0,\delta>0$ and a positive function $\rho=\rho(t)$ blowing up $t=T$ such that, 
		for any $\gamma\geq \gamma_0,\,\ell\geq\ell_0,\, u_0\in L^2(0,1)$ and 
		$u_d\in L^2(0,T;L^2(\mathcal{O}_d))$
    	satisfying  
		\begin{equation}\label{hypothesis.weight.mainresult}
			\|u_0-\overline{u}_0\|_{L^2(0,1)}+
			\displaystyle\iint\limits_{\mathcal{O}_d\times(0,T)}\rho^2(t)|\overline{u}-u_d|^2 dxdt \leq \delta,
		\end{equation}	    
    	there exist a leader control $h\in  L^2(0,T;L^2(\omega))$
 		and a unique saddle point $(\overline{v},\overline{\psi})\in L^2(0,T;L^2(\mathcal{O}))\times L^2(Q)$ 
		for the functional given by \eqref{eq.functional.ks}, and an 
    	associated solution $u$ to \eqref{eq.ks} verifying $u(\cdot,T)=\overline{u}(\cdot,T) \mbox{ in }\, (0,1).$
	\end{teo}

	It is worth mentioning again that the theoretical results known up to now 
	on robust Stackelberg controllability (Problem \ref{p3.robustandstackelberg}) are \cite{2018-robust-heat}, 
	\cite{2018-robust-ns} and \cite{2020victorlilinana}, and there is no evidence on both numerical algorithms and 
	a controllability to trajectories constrain for the leader control for nonlinear systems. 
	Therefore, this paper we pretend to show theoretical results and carry out  
	numerical schemes jointly with its implementation to Problems \ref{p1.saddlepoint},\ref{p3.robustandstackelberg} 
	for the KS equation \eqref{eq.ks}. 
	
	\vskip 0.3cm
	The rest of the paper is divided as follows: Section \ref{section.robust.control} contains all theoretical and 
	numerical answers to the robust control problem (see Problem \ref{p1.saddlepoint}) for the system \eqref{eq.ks}. 
	First, we present the existence, uniqueness and characterization of the robust control throughout optimal 
	control tools. Afterwards, a discrete scheme for the KS equation \eqref{eq.ks} as well as the procedure
	to the robust internal control problem are presented. 
	We devote  Section \ref{section.controllability} to prove the robust Stackelberg strategy 
	for the KS equation, see Theorem \ref{teo2.RobustStackelbergContKS}. That means, we prove the exact controllability 
	to the trajectories for the coupled KS system that arises as characterization of the robust control problem. 
	In the theoretical framework, the main tools will be 
	new Carleman estimates and fixed point arguments for coupled fourth--order parabolic systems. 
	Meanwhile, the implemented numerical scheme in the previous section will be adapted and complemented for coupled 
	and discretized KS equations.

\section{The robust control problem}\label{section.robust.control}	
	The main objetive in robust interior control is to determine the best control function
	$v\in L^2(0,T;L^2(\mathcal{O}))$ in the presence of the disturbance $\psi\in L^2(Q)$ which maximally spoils
	the control. In this section we prove the existence, uniqueness and characterization of a solution to the robust 
	internal control problem established in Problem \ref{p1.saddlepoint} and Theorem \ref{teo1.saddlepoint}. 
	In what follows, we assume that the leader $h$ has made a choice, so will keep it fixed along this section.

\subsection{\normalsize{Existence of the saddle point}}\label{subsection.saddlepoint}	
	This subsection is devoted to solve the minimization problem concerning the robust control problem. First,
	we prove the existence of a saddle point for the functional defined in \eqref{eq.functional.ks}. 
	The proof of existence of a saddle point $(\bar{v},\bar\psi)$  (Problem \ref{p1.saddlepoint}) is based on the 
	following proposition. Its proof can be found in \cite{bookekelandTemam}.   
	\begin{proposition}\label{prop.onsaddlepoint}
		Let $J$ be a functional defined on $X\times Y$ , where $X$ and $Y$ are convex, closed, non--empty, unbounded sets. If
		\begin{enumerate}
		\item [a)]$\forall v\in X,\,\psi\longmapsto J(v,\psi)$ is concave and upper semicontinuous.
		\item [b)]$\forall\psi\in Y,\,v\longmapsto J(v,\psi)$ is convex and lower semicontinuous.
		\end{enumerate}
		Then the functional $J$ possesses at least one saddle point $(\bar{v},\bar{\psi})$ on $X\times Y$ and 
		\begin{equation}\label{saddlepoint.cond}
		J(\bar{v},\bar{\psi})=\min_{v\in X}\sup_{\psi\in Y}\,J({v},{\psi})=\max_{\psi\in Y}\inf_{v\in X}
		\,J({v},{\psi}).
		\end{equation}
		Moreover, if $J$ is strictly concave with respect to $\psi$ and strictly convex with respect to $v$, $(\bar{v},\bar{\psi})$
		is unique.
	\end{proposition}

	In order to guarantee the existence of the saddle point $(\bar{v},\bar{\psi})$, we prove the following lemma.
	\begin{lemma}\label{lema.existence}
		Let $u_0\in H_0^2(0,1)$ be given. Then, there exists positive constants $\gamma_0$ and  $\ell_0$ such that, 
		for any $\gamma\geq \gamma_0$ and $\ell\geq\ell_0$ we have
		\begin{enumerate}
			\item [a)]$\forall v\in L^2(0,T;L^2(\mathcal{O})),\,\psi\longmapsto J_r(v,\psi)$ is concave and upper semicontinuous.
			\item [b)]$\forall\psi\in L^2(Q),\,v\longmapsto J_r(v,\psi)$ is convex and lower semicontinuous.
		\end{enumerate}
	\end{lemma}
	\begin{proof}[Proof of Lemma \ref{lema.existence}]\smallskip\
	First, since the norm is continuous, we only need to check the continuity of the first term in $J_r$ with respect to $v,\psi$. 
	To do this, let $u^{i}=u^{i}(v^{i},\psi^{i})\in C([0,T];H_0^2(0,1))\cap L^2(0,T;H^4(0,1))$, $i=1,2$ be the solutions of 
	equation \eqref{eq.ks} associated with the corresponding 
	external sources in $L^2(Q)$ (see Lemma \ref{lema.nonlinearregularity} and remark \ref{obs.final}). 
	Let $\delta u=u^1-u^2$, $\delta v=v^1-v^2$ and $\delta\psi=\psi^1-\psi^2$. 
	Using \eqref{eq.ks}, it is easy to verify that $\delta u$ satisfies the following system
	\begin{equation}\label{eq.ks.delta}
	\begin{cases}
		(\delta u)_{t}+(\delta u)_{xxxx}+(\delta u)_{xx}+u^1(\delta u)_x+u^2_{x}(\delta u)
		=\delta v+\delta\psi & \text{in }(0,1)\times(0,T)=:Q,\\
		(\delta u)(0,t)=(\delta u)(1,t)=(\delta u)_{x}(0,t)=(\delta u)_{x}(1,t)=0  & \text{on }(0,T),\\
		(\delta u)(\cdot,0)=0& \text{in }(0,1).
	\end{cases}
	\end{equation}	
	Due to the $u^1,u^2_x\in L^\infty(Q)$,  lemma \ref{apendix.lema.strong.linear}
	allows us to guarantee the existence of a positive constant $C$ depending on 
	$\|u^1\|_{L^\infty(Q)}$ and $\|u^2_x\|_{L^\infty(Q)}$ such that
	\[\|\delta u\|^2_{L^2(0,T;H^2(0,1))}\leq C(\|\delta v\|^2_{L^2(Q)}+\|\delta\psi\|^2_{L^2(Q)}).\]
	This complete the continuity of $J_r$ with respect to $(v,\psi)$. 
	\begin{enumerate}
	\item [a)] Since the norm is lower semicontinuous, the map $\psi\longmapsto J_r(v,\psi)$ is upper semicontinuous. In order to 
	prove the concavity, it is enough to show that 
	\[g(\rho)=J_{r}(v,\psi+\rho\psi' )\]
	is concave with respect to $\rho$ near $\rho=0$, that is, $g''(0)\leq 0$. 
	
	Let $u'=u'(0,\psi')=\frac{Du}{D\psi}\cdot\psi'$. Then $u'$ is the solution of 
	\begin{equation}\label{eq.ks.uprima}
	\begin{cases}
		u'_{t}+u'_{xxxx}+u'_{xx}+uu'_x++u_{x}u'
		=\psi' & \text{in }Q,\\
		u'(0,t)=u'(1,t)=u'_{x}(0,t)=u'_{x}(1,t)=0  & \text{on }(0,T),\\
		u'(\cdot,0)=0& \text{in }(0,1).
	\end{cases}
	\end{equation}	
	By computing, we have
	\begin{equation}\label{eq.der1}	
		g'(\rho)=\frac{DJ_r}{D\psi}(0,\psi+\rho\psi')\cdot\psi'=\int\limits_0^T(u-u_d,u')_{L^2(\mathcal{O}_d)}dt
		-\gamma^2\int\limits_0^T(\psi+\rho\psi',\psi')_{L^2(0,1)}dt.
	\end{equation}
	Similarly, let $\widehat{\psi'}\in L^2(Q)$ be another disturbance direction, and consider 
	$u''=\frac{D^2u}{D\psi^2}\cdot\psi'\cdot\widehat{\psi'}$, which solves the following system:
	\begin{equation}\label{eq.ks.udosprima}
	\begin{cases}
		u''_{t}+u''_{xxxx}+u''_{xx}+uu''_x+u_{x}u''
		=-w^2w^1_{x}-w^1w^2_{x}& \text{in }Q,\\
		u''(0,t)=u''(1,t)=u''_{x}(0,t)=u''_{x}(1,t)=0  & \text{on }(0,T),\\
		u''(\cdot,0)=0& \text{in }(0,1),
	\end{cases}
	\end{equation}		
	where $w^1=u'=\frac{Du}{D\psi}\cdot\psi'$ and $w^2=u'=\frac{Du}{D\psi}\cdot\widehat{\psi'}$ are solutions 
	of \eqref{eq.ks.uprima}. 
	By taking $\widehat{\psi'}=\psi'$ and thus $w^1=w^2$, we really have on the right--hand side of the equation
	\eqref{eq.ks.udosprima} the term $-2u'u'_{x}$.
	
	On the other hand, from \eqref{eq.der1} we get
	\begin{equation}\label{eq.der2}	
		g''(\rho)=\iint\limits_{\mathcal{O}_d\times(0,T)}(u-u_d)u''\,dxdt+\iint\limits_{\mathcal{O}_d\times(0,T)}|u'|^2\,dxdt
		-\gamma^2\iint\limits_{Q}|\psi'|^2 dxdt.
	\end{equation}
	Now, we will see that for $\gamma$ sufficiently large, the last term in the above identity dominates in the expression \eqref{eq.der2},
	and therefore $g''(0)\leq 0$, for $(v,\psi)\in L^2(0,T;L^2(\mathcal{O}))\times L^2(Q)$. We begin by estimating the 
	second term. Thanks to the assumptions that $u\in\mathcal{Z}$ (see \eqref{def.spaceZ} and lemma 
	\ref{lema.nonlinearregularity}),  lemma 
	\ref{apendix.lema.strong.linear} can be applied to the linearized system 
	\eqref{eq.ks.uprima}. Thus, for any $\psi'\in L^2(Q)$, there exists a unique solution 
	$u'\in C([0,T];L^2(0,1))\cap L^2(0,T;H^2(0,1))\equiv W(0,T)$ to \eqref{eq.ks.uprima}
	such that 
	\begin{equation}\label{estimate.c1}
		\iint\limits_{\mathcal{O}_d\times(0,T)}|u'|^2\,dxdt\leq C_1\iint\limits_{Q}|\psi'|^2 dxdt,
	\end{equation}
	where $C_1$ is a positive constant.
	 	
	To estimate the first term, we need an upper bound for $u''$. Using the fact that  $u'\in W(0,T)$, it follows that 
	$|u'|^2\in L^1(0,T;W^{0,1}(0,1))$. Then, we have  that $(|u'|^2)_x$ belongs to $L^1(0,T;W^{-1,1})$. 
	Applying lemma \ref{apendix.lemma.weak} with $y=u'',\overline{y}=u$ and $f=(|u'|^2)_x$, the linearized system 
	\eqref{eq.ks.udosprima} has a unique solution $u''\in C([0,T];H^{-2}(0,1))\cap L^2(0,T;L^2(0,1))$. In addition, 
	from definition \ref{def.weakweaksol} we obtain
	 \begin{equation}
	 	\iint\limits_{Q}(u-u_d)u''\,dxdt=\langle -2u'u'_x,w\rangle_{L^1(0,T;W^{-1,1}), L^\infty(0,T;W^{1,\infty}(0,1))},
	 \end{equation}	
	 where $w\in\mathcal{Z}$ is the solution of \eqref{apendix.linearized.w}. 
	 
 	Thus, there exists a positive constant $C_2$ only depending on $\|w\|_{L^\infty(0,T;W^{1,\infty}(0,1))}$ such that
	 \begin{equation}
	 	\iint\limits_{Q}(u-u_d)u''\,dxdt\leq C_2\|u'u'_x\|_{L^1(0,T;W^{-1,1}(0,1))}\leq \frac{C_2}{2}\|u'\|^2_{L^2(Q)}.
	 \end{equation}	
	Using again that $u'\in W(0,T)$ is solution of \eqref{eq.ks.uprima} and the previous inequality,  we deduce
	\begin{equation}\label{estimate.c2}
	 	\iint\limits_{\mathcal{O}_d\times(0,T)}(u-u_d)u''\,dxdt\leq  \frac{C_2}{2}\|u'\|^2_{L^2(Q)}\leq C_1\frac{C_2}{2}\|\psi'\|^2_{L^2(Q)}.
	 \end{equation}	
	 Putting together \eqref{eq.der1}, \eqref{estimate.c1} and \eqref{estimate.c2} yields
	 \[g''(\rho)\leq \Bigl(C_1+C_1\frac{C_2}{2}-\gamma^2\Bigr)\|\psi'\|^2_{L^2(Q)},\quad \forall\psi'\in L^2(Q),\,\psi'\neq 0.\]
	 Therefore, under the assumption that $\gamma^2\geq \gamma_0^2:=C_1+C_1\displaystyle\frac{C_2}{2}$, we have 
	 $g''(\rho)\leq 0$ for all $\rho\in\mathbb{R}$. 
	 Thus, the function $g$ is concave and the strictly concavity of $\psi\longmapsto J_r(v,\psi)$ follows for  $\gamma$ 
	 large enough.
	 \item [b)] Under the same scheme of the above proof, in order to show convexity of the map 
	 $v\longmapsto J_r(v,\psi)$, it
	 is sufficient to prove that 
	\[g(\rho)=J_{r}(v+\rho v',\psi)\]
	is convex with respect to $\rho$ near $\rho=0$, that is, $g''(0)> 0$. Arguing as above, we obtain
	\begin{equation}\label{eq.der2.convex}	
		g''(\rho)=\iint\limits_{\mathcal{O}_d\times(0,T)}(u-u_d)u''\,dxdt
		+\iint\limits_{\mathcal{O}_d\times(0,T)}|u'|^2\,dxdt
		+\ell^2\iint\limits_{\mathcal{O}\times(0,T)}|v'|^2 dxdt,
	\end{equation}
	where we have denoted  $u'=u'(v',0)=\frac{Du}{D v}\cdot v'$ and 
	$u''=\frac{D^2u}{D v^2}\cdot v'\cdot\widehat{v'}$. Observe that estimates for $u'$ and $u''$ can be obtained
	in the same way of Condition a) by replacing $\psi'$ by $v'$ in \eqref{eq.ks.uprima} and \eqref{eq.ks.udosprima}.
	Thus, it follows that
	\[g''(\rho)\geq \Bigl(\ell^2-C_1-C_1\frac{C_2}{2}\Bigr)\|v'\|^2_{L^2(0,T;L^2(\mathcal{O}))},\quad \forall
	v'\in  L^2(0,T;L^2(\mathcal{O})),\,v'\neq 0.\]
	Therefore, under the assumption that $\ell^2\geq \ell^2_0:=C_1+C_1\displaystyle\frac{C_2}{2}$ we have 
	$g''(\rho)\geq 0$ for all $\rho\in\mathbb{R}$. 
	 Thus, the function $g$ is convex and the strictly convex of $v\longmapsto J_r(v,\psi)$ follows for  $\ell$ 
	 large enough.
	\end{enumerate}
	This complete the proof of Lemma \ref{lema.existence}. 
	\end{proof}
	
	Next, we carry out the proof of the main result of this section, i.e., Theorem \ref{teo1.saddlepoint}. 
	\begin{proof}[Proof of Theorem \ref{teo1.saddlepoint}]
	It is a direct consequence of lemma \ref{lema.existence} and proposition \ref{prop.onsaddlepoint}. Therefore, there
	exists  a  pair 
	$(\bar{v},\bar{\psi})$ on $L^2(0,T;L^2(\mathcal{O}))\times L^2(Q)$ and an associated solution to \eqref{eq.ks} 
	$u=u(h,\bar{\psi},\bar{v})$ satisfying \eqref{saddlepoint.cond}. 
	\end{proof}
	A useful characterization of saddle 	point, in the case where $J_r$ is a differentiable function is the 
	following proposition (see \cite{2001TemamChangbing2} and references therein).
	\begin{proposition}\label{prop.useful}
		In addition to the hypotheses of Proposition	 \ref{prop.onsaddlepoint}, assume
		\begin{enumerate}
			\item [c)] $\forall v\in X,\,\psi\longmapsto J(v,\psi)$ is Gateaux differentiable.
			\item [d)] $\forall \psi\in Y,\,v\longmapsto J(v,\psi)$ is Gateaux differentiable.			
		\end{enumerate}
		Then $(\bar{v},\bar{\psi})\in X\times Y$ is saddle point of $J$ if and only if 
		\begin{equation}\label{ine.charac.saddlepoint}
			\left\{
			\begin{array}{ll}
				\langle \frac{\partial J}{\partial v}(\bar{v},\bar{\psi}), v-\bar{v}\rangle\geq 0,& 
				\quad\forall v\in X,\vspace{0.2cm} \\
				 \langle \frac{\partial J}{\partial\psi}(\bar{v},\bar{\psi}), \psi-\bar{\psi}\rangle\leq 0,& 
				\quad\forall \psi\in Y.
			\end{array}\right.
		\end{equation}
	\end{proposition}
	Observe that proposition \ref{prop.useful} is also applicable to our case, so we have the characterization 
	\eqref{ine.charac.saddlepoint} for the saddle point in theorem \ref{teo1.saddlepoint}. It will be studied in the 
	next subsection.

\subsection{\normalsize{Characterization of the robust control}}	
	In this subsection we will identify the gradient of the cost functional $J_r$ (see \eqref{eq.functional.ks})
	with respect to the control $v$ and the disturbance $\psi$, which turn out to be useful for the numerical
	framework for determining the robust control solution, and whose analysis is given later on. As proved in the 
	above subsection, the existence of a saddle point $(\bar v,\bar\psi)$ of the functional $J_r$ implies 
	\eqref{ine.charac.saddlepoint}. As consequence,  for the functional $J_r$ follows that  
	for any $\psi\in L^2(Q)$ and $v\in L^2(0,T;L^2(\mathcal{O}))$  
	\[\frac{D J_r}{D\psi}(\bar{v},\bar{\psi})=0,\quad 
	\frac{D J_r}{D v}(\bar{v},\bar{\psi})=0.\]
	Following the arguments by \cite{2000-bewleytemamziane} and \cite{2001TemamChangbing2}, we can deduce 
	 \[\frac{D J_r}{D v}(\bar{v},\bar{\psi})= (\ell^2\bar{v}-z)1_{\mathcal{O}}\quad
	\mbox{and}\quad \frac{D J_r}{D \psi}(\bar{v},\bar{\psi})=-\gamma^2\bar{\psi}-z,\]
	where $z$ is the solution to the problem 
	\begin{equation*}
		\begin{cases}
		-z_{t}+z_{xxxx}+z_{xx}-uz_{x}=(u-u_{d})1_{\mathcal{O}_{d}} & \text{in }Q,\\
		z(0,t)=z(1,t)=z_x(0,t)=z_x(1,t)=0 & \text{on }(0,T),\\
		z(\cdot,T)=0& \text{in }(0,1).
		\end{cases}
	\end{equation*}	  

	In summary, the robust internal control problem is characterized 
	by the following Lemma. 
	\begin{lemma}\label{lema.ks.characterization}
		Let $h\in L^2(0,T;L^2(\omega))$ and $u_0\in H_0^2(0,L)$ be given. 
		Suppose that $(\bar{v},\bar{\psi})$ is the solution to the  robust control 
		problem established in Theorem \ref{teo1.saddlepoint}. Then
		\[\bar\psi=\frac{1}{\gamma^2}z\quad\mbox{and}\quad\bar{v}=-\frac{1}{\ell^2}z1_{\mathcal{O}},\]
		where $z$ is the second component of $(u,z)$ solution to the 
		following coupled system
		\begin{equation}\label{eq.ks.characterization}
		\begin{cases}
		u_{t}+u_{xxxx}+u_{xx}+uu_x=h1_{\omega}+ (-\ell^{-2}1_{\mathcal{O}}+\gamma^{-2})z & \text{in }Q,\\
		-z_{t}+z_{xxxx}+z_{xx}-uz_{x}=(u-u_{d})1_{\mathcal{O}_{d}} & \text{in }Q,\\
		u(0,t)=u(1,t)=z(0,t)=z(1,t)=0 & \text{on }(0,T),\\
		u_{x}(0,t)=u_{x}(1,t)=z_{x}(0,t)=z_{x}(1,t)=0 & \text{on }(0,T),\\
		u(\cdot,0)=u_{0}(\cdot),\,\, z(\cdot,T)=0& \text{in }(0,1).
		\end{cases}
		\end{equation}	
	\end{lemma}
	
	\subsection{\normalsize{Numerical method}}\label{section.numerical.method}	
	Finite element solutions for the KS equation are not common because the primal variational formulation of fourth--order
	operators requires finite element basis functions which are piecewise smooth and globally at least $C^1$--continuous.
	Although the KS equation has been studied numerically by several schemes such as local discontinuous Galerkin methods 
	\cite{2006Xu}, finite elements \cite{2019Doss,1994Akrivis},
	variable mesh finite difference methods \cite{2017Mohanty}, B--spline finite difference--collocation method \cite{2012Lakestani},
	the inverse scattering method \cite{bookDrazin}, a higher--order finite element approach \cite{2012Andersetal}, finite difference 
	\cite{1992Akrivis,2016singhnote,michelson1977nonlinear,2019Mohanty}, 
	spectral method \cite{2013Barker}. In this paper, a new numeric
	solution for the KS equation is obtained by introducing a $\theta$--scheme/Adams--Bashforth algorithm for the time discretization
	and $\mathbb{P}_1$--type Lagrange polynomials for the spatial approximation. This setting simplifies the treatment of the
	nonlinearity in a semi--implicit form and also decompose the fourth--order equation to a coupled system of two second--order
	equations, which allows to use $C^0$--basis functions instead of $C^0$--basis functions.

	In this subsection, we develop a finite element method for the solution of the
	nonlinear robust control problem associated to the KS equation \eqref{eq.ks}. As mentioned, this problem is equivalent 
	to find a saddle point for the functional $J_r$, which is characterized by the coupled system 
	\eqref{eq.ks.characterization}.
	In order to obtain better illustrations of our results, we consider a symmetric domain $(-L,L)$ ($L>0$) instead of $(0,1)$.
	Let us first consider the KS equation
	\begin{equation}\label{eq.ks.numeric1}
	\begin{cases}
		u_{t}+u_{xxxx}+u_{xx}+uu_{x}=f & \text{in }(-L,L)\times(0,T),\\
		u(-L,t)=u(L,t)=u_{x}(-L,t)=u_{x}(L,t)=0 & \text{in }(0,T),\\
		u(\cdot,0)=u_{0}(\cdot)& \text{in }(-L,L),
	\end{cases}
	\end{equation}
	By defining a new variable $w$ as $w=u_{xx}$, the above problem may be considered in a coupled manner as:
	\begin{eqnarray}
		&&u_{xx}=w,\label{eq.ks.numeric2} \\
		&&u_{t}+w_{xx}+w+uu_{x}=f,\label{eq.ks.numeric3}
	\end{eqnarray}
	and subject to the following initial and boundary conditions:
	\begin{eqnarray}
		&&u(x,0)=u_{0}(x),\quad  -L\leq x\leq L, \label{ic1}\\
		&&u(-L,t)=u(L,t)=u_{x}(-L,t)=u_{x}(L,t)=0,\quad t>0,\label{ic2}\\
		&&u_{xx}(-L,t)=w(-L,t)=w_1(t),\quad u_{xx}(L,t)=w(L,t)=w_2(t),\quad t>0,\label{ic3}
	\end{eqnarray}
	where $u_0,w_1$ and $w_2$ are given smoothness functions. Thanks to the initial condition \eqref{ic1}, the values 
	of all successive partial derivatives of $u$ can be determined at $t=0$. Thus, the value of $w$ is also known at $t=0$. 
	
	To obtain the numerical solution of the problem \eqref{eq.ks.numeric2}--\eqref{eq.ks.numeric3} subject to 
	\eqref{ic1}--\eqref{ic3}, the time domain is split into $N_T$ intervals, i.e., $0<t_1<t_2<\cdots t_{N_T}=T$, where the 
	steps are of equal length $\Delta t$. Besides, we will use  $\mathbb{P}_1$--type finite elements for the spatial
	discretization (see for instance \cite[Section 6.2]{book.allaire1}) and a $\theta$--scheme/Adams--Bashforth 
	(TAB2) for the  time advancing. More precisely,
	by letting $u^n(x)=u(x,n\Delta t)$ and $w^n(x)=w(x,n\Delta t)$ for some small $\Delta t$. TAB2 approximations
	to \eqref{eq.ks.numeric2}--\eqref{eq.ks.numeric3} are given by 
	\begin{eqnarray}
		&&\frac{u^{n+1}-u^n}{\Delta t}+\theta \mathcal{A}(w^{n+1})+(1-\theta)\mathcal{A}(w^{n})-\frac{3}{2}N(u^n)+\frac{1}{2}\mathcal{N}(u^{n-1})
		=f^{n+1},\label{TAB2.eq1}\\
		&& w^{n+1}-u_{xx}^{n+1}=0,\label{TAB2.eq11}
	\end{eqnarray}
	where $\mathcal{A}w=w_{xx}+w$ corresponds to the linear part and $\mathcal{N}(u)=uu_x$ is the nonlinear term.  
	For the spatial discretization, we consider the discrete space 
	\[V_h=\{u\in C([-L,L]):  u|_{[x_j,x_{j+1}]}\in \mathbb{P}_1\,\,\mbox{for all}\,\,0\leq j\leq N \}\]
	and its subspace
	\[V_{0h}=\{u\in V_h: u(-L)=u(L)=0\}.\]
	Thus, after integrating by parts, the 
	discrete variational problem of the internal approximation becomes: to find 
	$(u^{n+1}_h,w^{n+1}_h)\in V_{0h}\times  V_{0h}$ such that
	\begin{eqnarray}\label{TAB2.eq2}
		& (u^{n+1}_h,u_1)+\Delta t\theta((w^{n+1}_h,u_1)-(\partial_xw_h^{n+1}, \partial_x u_1))=
		F(u^n_h,u^{n-1}_h,u_1),\quad  \forall u_1\in V_{0h},\label{TAB2.eq1}\\
		& (w^{n+1}_h,u_2)+(\partial_{x} u_{h}^{n+1},\partial_{x}u_2)=0, \hspace{5.4cm}  \forall u_2\in V_{0h},
	\end{eqnarray} 
	where
	\begin{equation}\label{TAB2.eq3}
	\begin{array}{lll}
		F(u^n_h,u^{n-1}_h,u_1)=&\Delta t(\theta-1)((w^{n}_h,u_1)-(\partial_{x}w_h^{n}, \partial_x u_1))+\frac{3}{2}\Delta t(\mathcal{N}(u^n_h),u_1)
		-\frac{1}{2}\Delta t(\mathcal{N}(u^{n-1}_h,u_1)\\
	  	&+\Delta t(f^{n+1},u_1)
	\end{array}
	\end{equation}
	and $(\cdot,\cdot)$ denotes the inner product of $L^2((-L,L))$.

	First, we test numerically the accuracy of our method for the resolution of the nonlinear KS equation 
	\eqref{eq.ks.numeric1} by taking the following function
	\[u(x,t)=(t+1)\sin^2\Bigl(\frac{\pi x}{30}\Bigr),\quad x\in (-30,30),\]
	 as the solution of \eqref{eq.ks.numeric1}, where the right--hand side term is 
	\[f(x,t)=-\frac{\pi^2(-225+\pi^2)(1+t)\cos(\frac{\pi x}{15})}{101250}+\frac{1}{30}
	\Bigl(30+\pi(1+t)^2\sin^2\Bigl(\frac{\pi x}{15}\Bigr)\Bigr)\sin^2\Bigl(\frac{\pi x}{30}\Bigr).\]
	
	\begin{figure}[ht]
	\begin{center}
		\includegraphics[scale=0.22]{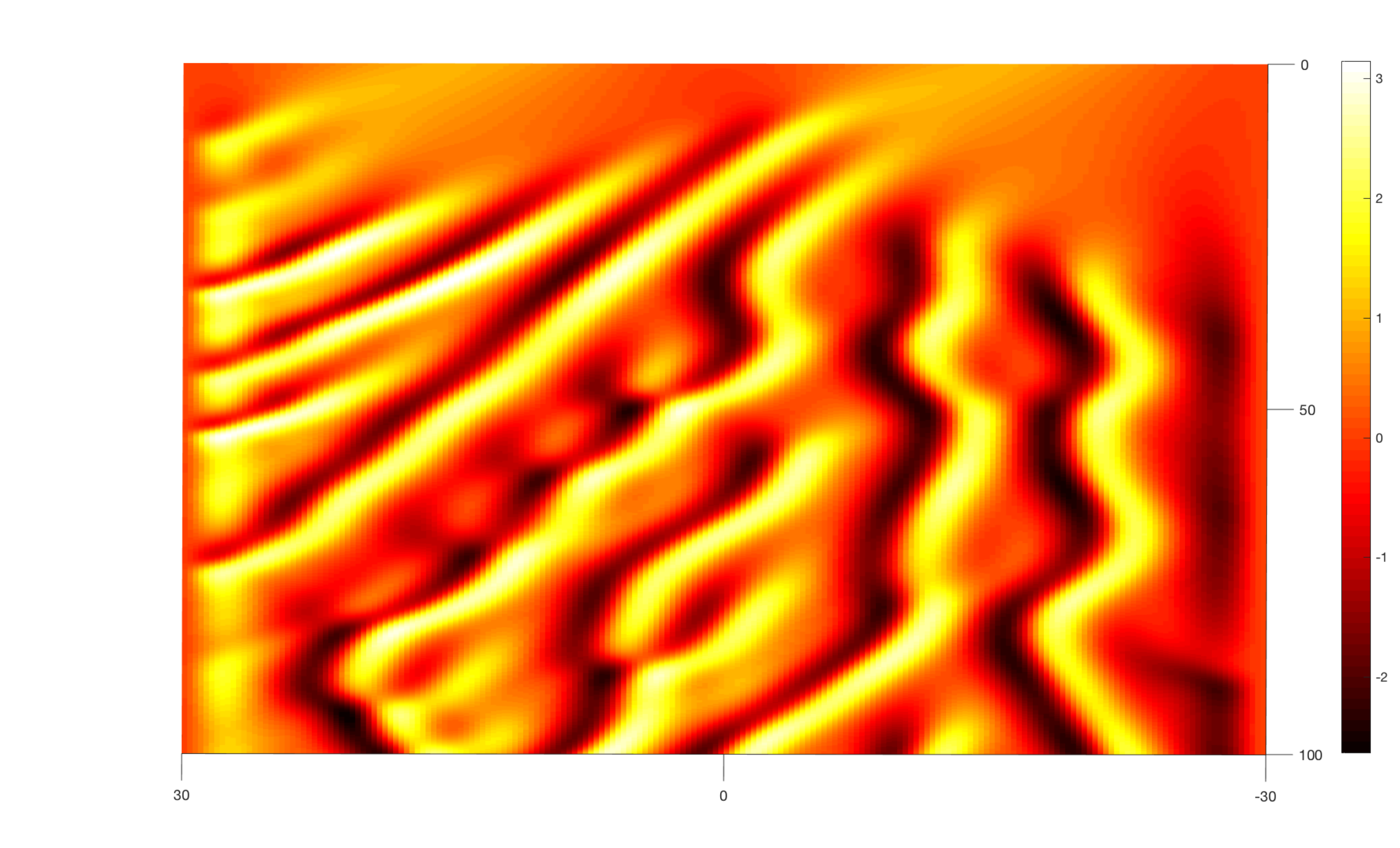} 
		\hspace{0.01cm}
		\includegraphics[scale=0.22]{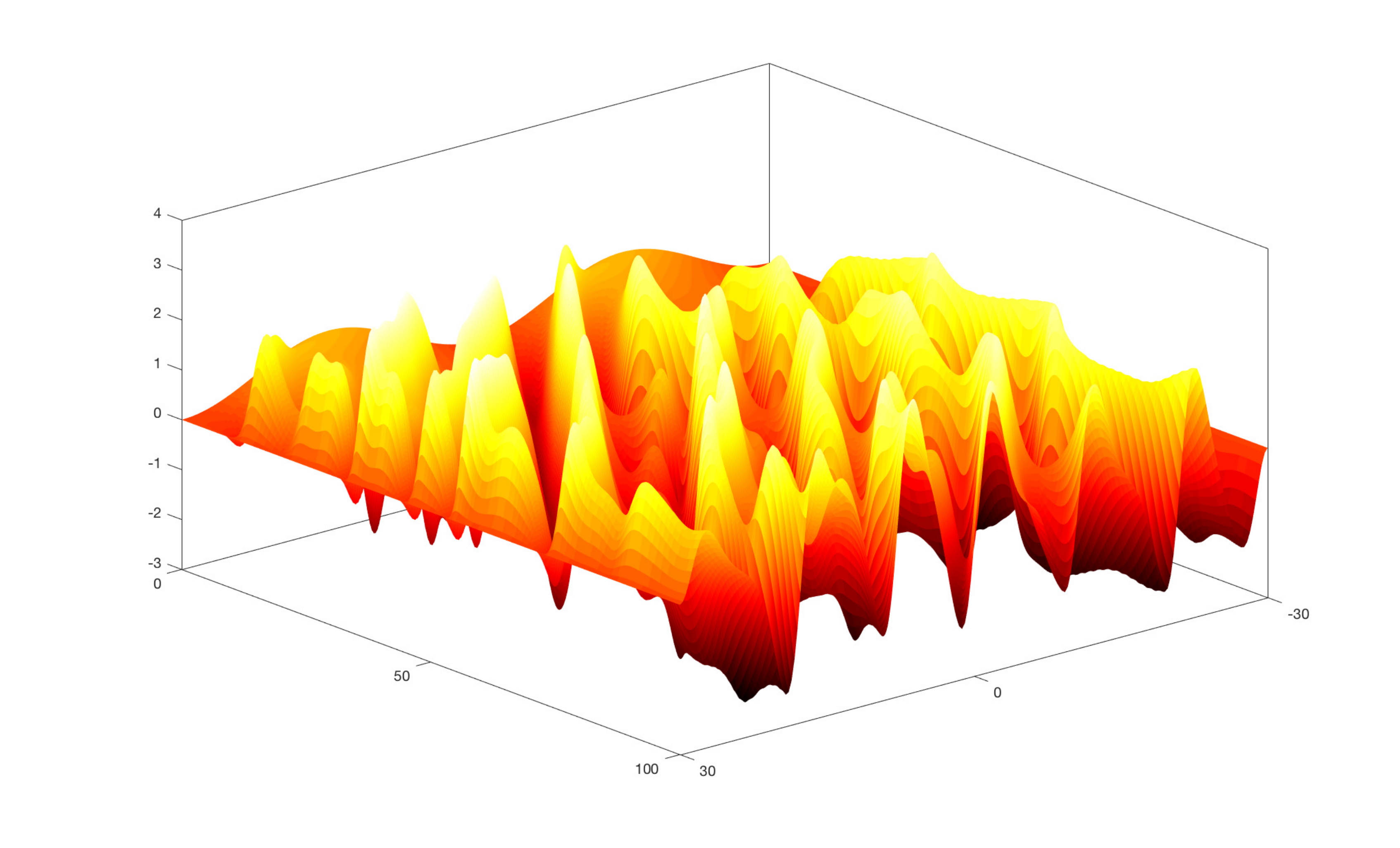}   		
	\end{center}
	\par
	\caption{The graph of numerical solution for $N=200$ (spatial nodes) with temporal step $\Delta t=10^{-3}$ and $T=100$. 
	The numerical approximation by the 
	$\theta$/Adams--Bashforth method with $\theta=\frac{3}{4}$ and 
	Lagrange finite elements of order 1.}\label{fig.solutions.ks}
	\end{figure}
	To see the order of the accuracy between our numerical approximation and the exact solution given above, 
	we let $\Delta t$ decrease from $10^{-1}$ to $10^{-6}$ for large $N=200$, and let 
	$N$ increase from $25$ to $100$ for small $\Delta t=10^{-6}$. The results are given in Table \ref{table.errorKS1}.
	The mathematical study of stability, convergence and accuracy for the above method will be developed in a forthcoming 
	paper.  
	\begin{table}[ht!]
	\begin{centering}
		\scalebox{1.00}{\begin{tabular}{|c|c|c|c|}
	\hline 
	$\Delta t$ &   $N$ & $ L^{\infty}-\mbox{error}$ &  $ L^2-\mbox{error}$
	\tabularnewline
	\hline 
	 $1e{-01}$ &   & $1.32e-02$ & $5.54e-06$

	\tabularnewline
	 $1e{-02}$ &    &$ 1.13e-03$ & $ 3.46e-08$

	\tabularnewline
	$1e{-03}$ &   200 & $8.79e-05$ & $2.71e-10$
	\tabularnewline
	$1e{-04}$ &    &$ 5.55e-05$ & $ 5.66e-11$

	\tabularnewline
	$1e{-05}$ &   &$ 5.46e-05$ & $ 6.58e-11$

	\tabularnewline

	$1e{-06}$ &    &$ 5.45e-05$ & $ 6.70e-11$

	\tabularnewline
	\hline \hline
	         &   25 &$ 3.31e-03$ & $ 1.86e-06$
	\tabularnewline
	$1e{-06}$   & 50 &$ 8.83e-04$ & $ 6.58e-08$
    \tabularnewline
             &      100 &$ 2.19e-04$ & $ 2.13e-09$
    \tabularnewline
    \hline 
	\end{tabular}}
	\par\end{centering}
	\vskip 0.5cm
	\caption{Errors in $L^\infty$ and $L^2$ norms at $T=1s$ using the  $\theta$--scheme/Adams--Bashforth method with 
		$\mathbb{P}_1$--type Lagrange polynomials \eqref{TAB2.eq2}--\eqref{TAB2.eq3} for the KS equation 
		\eqref{eq.ks.numeric1}.}\label{table.errorKS1}  
	\end{table}
	
	Now, in order to approximate the solution of the robust control problem, Problem \ref{p1.saddlepoint}, we use as starting
	point the above discretization schemes as well as the characterization given by Lemma \ref{lema.ks.characterization}.
	Secondly, based in the numerical algorithm 
	proposed in \cite{2000-bewleytemamziane,2002iterative-tachim} for the Navier--Stokes equations, we propose a similar 
	algorithm for the KS system. Our main novelty relies in the form of constructing the ascent and descent directions, 
	and whose basis is the preconditioned nonlinear gradient conjugate method \cite{1994Quarteroni}. The algorithm
	reads as follows. 
	\begin{algorithm}[ht]\label{algorithm1}
	\SetAlgoLined
	\KwIn {Initialize $k=0$ and  $(v^0,\psi^0)=(0,0)$ on $t\in [0,T]$, where $k$ is the iteration index and 
		$(v^k,\psi^k)$ is the numerical approximation of the control and the disturbance during the $k$th iteration of 
		the algorithm.}
		Determine the state $u^{k+1}$ on $[0,T]$ from the KS equation with initial datum $u^0$ and the forcing 
		$(v^k1_{\mathcal{O}},\psi^k)$, where $\mathcal{O}\subset (-L,L)$.
			
		Determine the adjoint state $z^{k+1}$ on $[0,T]$ from the adjoint equation based on the state $u^{k+1}$.
		
		Determine the local expression of the gradients
		\[\frac{DJ_r}{Dv}(v^k,\psi^k)\quad\mbox{and}\quad \frac{DJ_r}{D\psi}(v^k,\psi^k).\]	
	
		Determine the updated disturbance $\psi^{k+1}$ using
		\[\psi^{k+1}=\psi^k+\alpha^k\frac{DJ_r}{D\psi}(v^k,\psi^k),\]	
		where $\alpha^k\in (0,1)$ is determined by an iterative procedure described in Remark \ref{obs.directions.alpha.beta}. 
		
		Determine the updated control $v^{k+1}$ using 
		\[v^{k+1}=v^k-\beta^k\frac{DJ_r}{Dv}(v^k,\psi^k),\]	
		where $\beta^k\in (0,1)$ is determined by an iterative procedure described in Remark \ref{obs.directions.alpha.beta}. 
		
		Increment index $k=k+1$. Repeat from step 3 until convergence.
				
 		\caption{Robust control algorithm}
	\end{algorithm}
		
 	\begin{Obs}\label{obs.directions.alpha.beta}
 		To find appropriate $\alpha^k$ and $\beta^k$ directions, we should be able to minimize 
		the following nonlinear functions: 
		\[f_{k}(\alpha)=J_r\Bigl(v^{k},\psi^{k}+\alpha\frac{DJ_r}{D\psi}(v^{k},\psi^{k})\Bigr)\quad\mbox{and}\quad 
		g_{k}(\beta)=J_r\Bigl(v^{k},v^{k}+\beta\frac{DJ_r}{Dv}(v^{k},\psi^{k})\Bigr).\] 
	Thus, for $k\in \mathbb{N}$, we have 
	\begin{equation*}
	\begin{aligned}
		f'_k(\alpha)=&\lim\limits_{\varepsilon\to 0}\frac{f_k(\alpha+\varepsilon)-f_k(\alpha)}{\varepsilon}\\
	  	=&\lim\limits_{\varepsilon\to 0} \frac{J_r(v^{k},\psi^{k}+\alpha\frac{DJ_r}{D\psi}(v^{k},\psi^{k})
		+\varepsilon\frac{DJ_r}{D\psi}(v^{k},\psi^{k}))-J_r(v^{k},\psi^{k}+\alpha\frac{DJ_r}{D\psi}(v^{k},\psi^{k}))}{\varepsilon}\\
		=&\Bigl\langle\frac{DJ_r}{D\psi}(v^{k},\psi^{k}+\alpha\frac{DJ_r}{D\psi}(v^{k},\psi^{k})),\frac{DJ_r}{D\psi}(v^{k},\psi^{k})
		\Bigl\rangle_{L^2(Q)}.
	\end{aligned}
	\end{equation*}
	Observe that if $\alpha<<1$, then 
	\[|f'(\alpha)|\simeq\Bigl\|\frac{DJ_r}{D\psi}(v^{k},\psi^{k})\Bigr\|_{L^2(Q)}.\]
	If $\|\frac{DJ_r}{D\psi}(v^{k},\psi^{k})\|_{L^2(Q)}>>1$, then we take
	\[\alpha^{k+1}=\alpha^{k}-\frac{f'_k(\alpha^{k})}{\|\frac{DJ_r}{D\psi}(v^{k},\psi^{k})\|_{L^2(Q)}},\]
	otherwise (that is, if $\|\frac{DJ_r}{D\psi}(v^{k},\psi^{k})\|_{L^2(Q)}\leq 1$),  then we take
	\[\alpha^{k+1}=\alpha^{k}-f'_k(\alpha^{k}).\]
	In other words, we have consider a preconditioner 
	\footnote{\url{https://doc.freefem.org/documentation/algorithms-and-optimization.html}} 
	$P_{k}(x)$ defined by (see \cite{1998Lucquin}):
	\[P_{k}(x):=
	\begin{cases}
		\displaystyle\frac{x}{\|\frac{DJ_r}{D\psi}(v^{k},\psi^{k})\|_{L^2(Q)}}, & if\, 
		\quad \Bigl\|\frac{DJ_r}{D\psi}(v^{k},\psi^{k})\Bigr\|_{L^2(Q)}>1,\\
		x, & if\, \quad\Bigl|\frac{DJ_r}{D\psi}(v^{k},\psi^{k})\Bigr\|_{L^2(Q)}\leq1.
	\end{cases}
	\]
	An analogous procedure is realized for obtaining $\beta^k$. 
	
	\end{Obs}
	The criterion for the termination of the algorithm is given by 
	\[\Bigl\|\frac{DJ_r}{D\psi}(v^{k},\psi^{k})\Bigr\|_{L^2(Q)}+\Bigl\|\frac{DJ_r}{Dv^k}(v^{k},\psi^{k})\Bigr\|_{L^2(Q)}<tol,\]
	which is analogous to the presented in \cite{2002iterative-tachim}.
	
	Figures \ref{fig.control1disturbance1.ks}--\ref{fig.state3target3.ks} display 
	numerical results on the robust internal control problem by considering different 
	parameters $\ell$ and $\gamma$ as well as for some functions $u_d$. In all experiments,  
	$tol=10^{-6}$, the initial datum is
	$u^0(x)=\sin^2\Bigl(\frac{\pi x}{30}\Bigr)$,\,$x\in(-30,30)$. 
	\begin{figure}[h!]
	\begin{center}
		\includegraphics[scale=0.21]{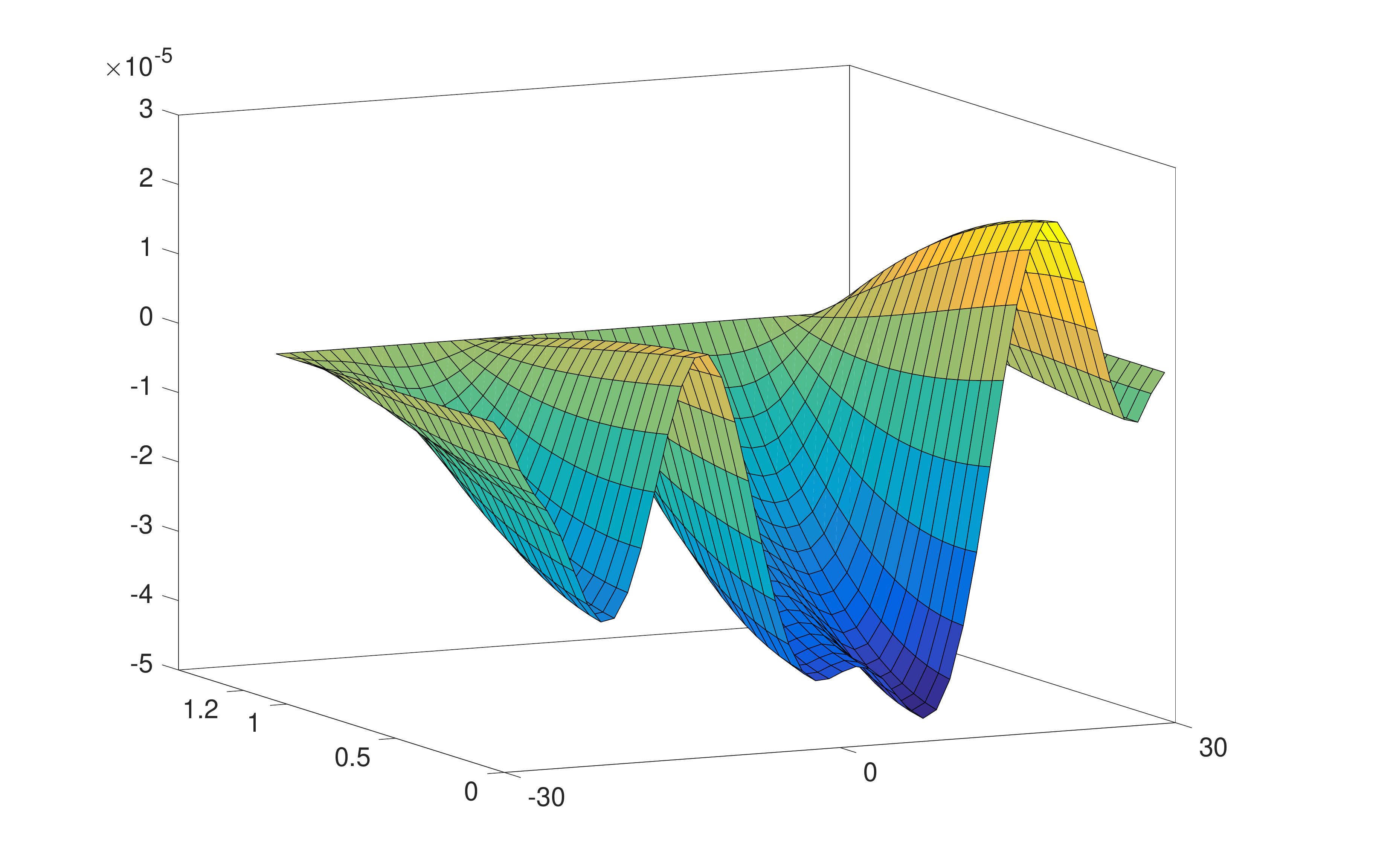}   		
		\includegraphics[scale=0.21]{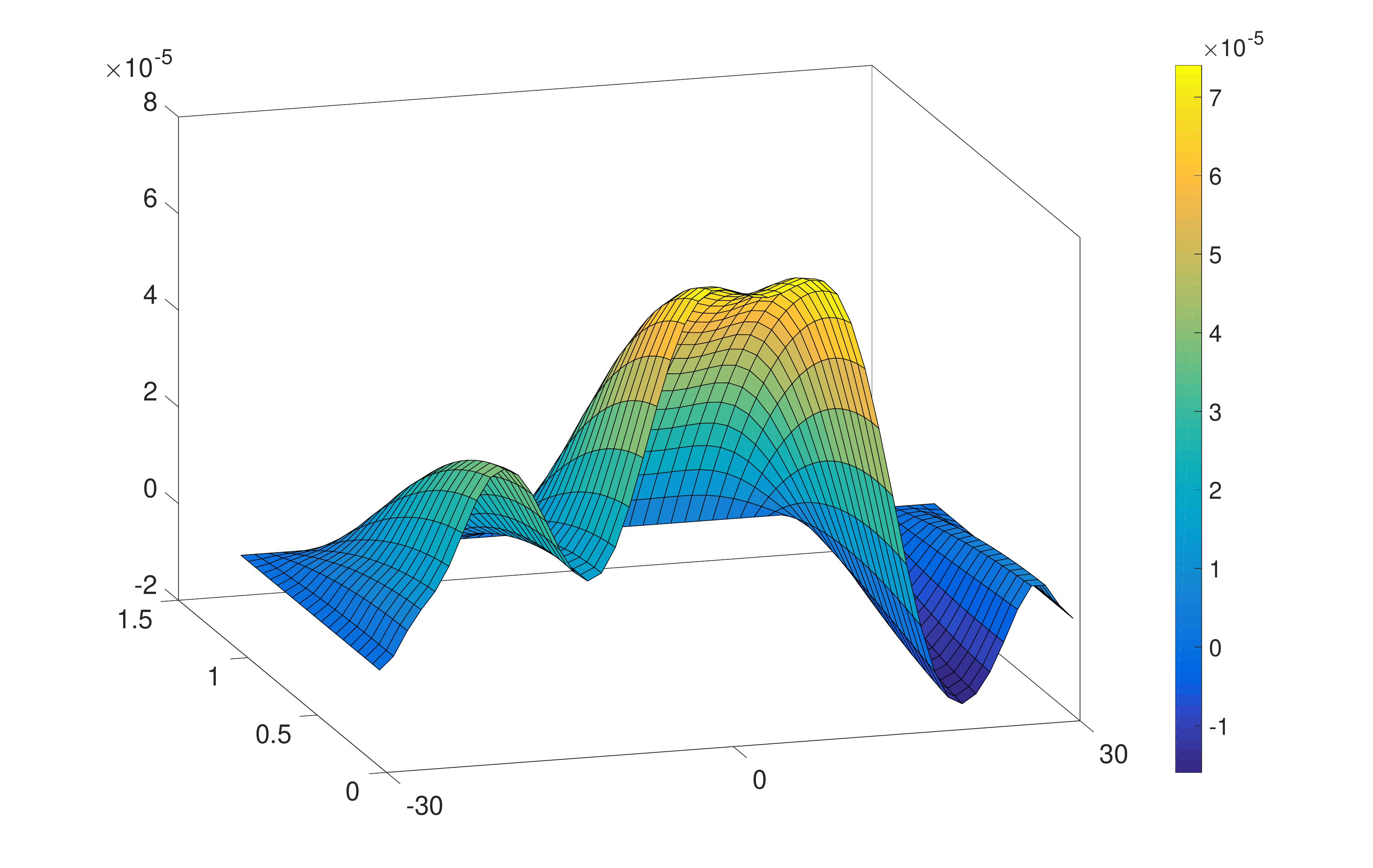} 	
	\end{center}
	\par
	\caption{Disturbance signal $\psi$ (left) and control function $v$ (right) on the spatial domain $(-30,30)$. 
	$T=1s,$\, $N=50, \Delta t=2\times10^{-2}$ and $ 
	\ell=40, \gamma=40$.}\label{fig.control1disturbance1.ks}
	\end{figure}	
	
	\begin{figure}[ht!]
	\begin{center}
		\includegraphics[scale=0.21]{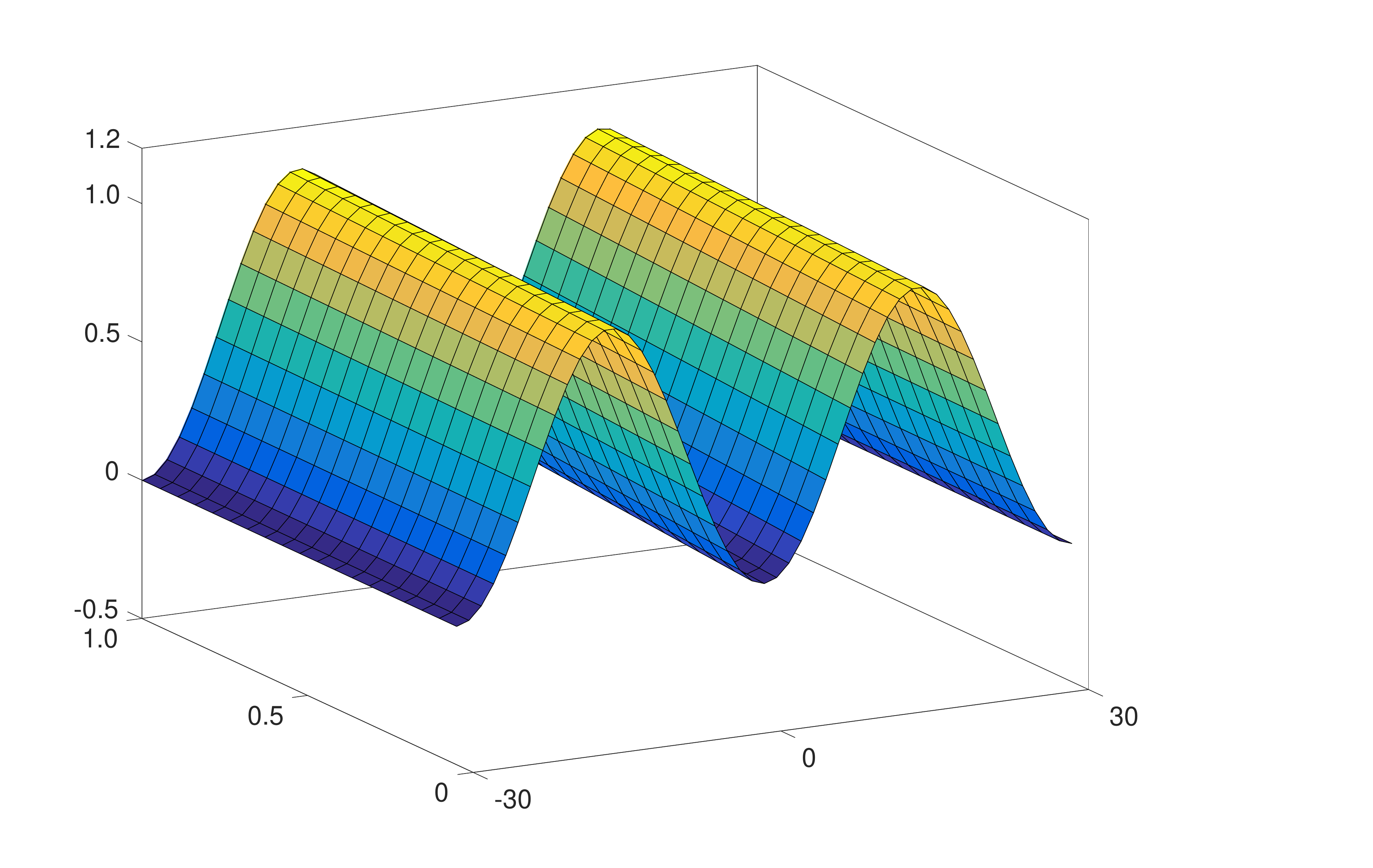}   		
		\includegraphics[scale=0.21]{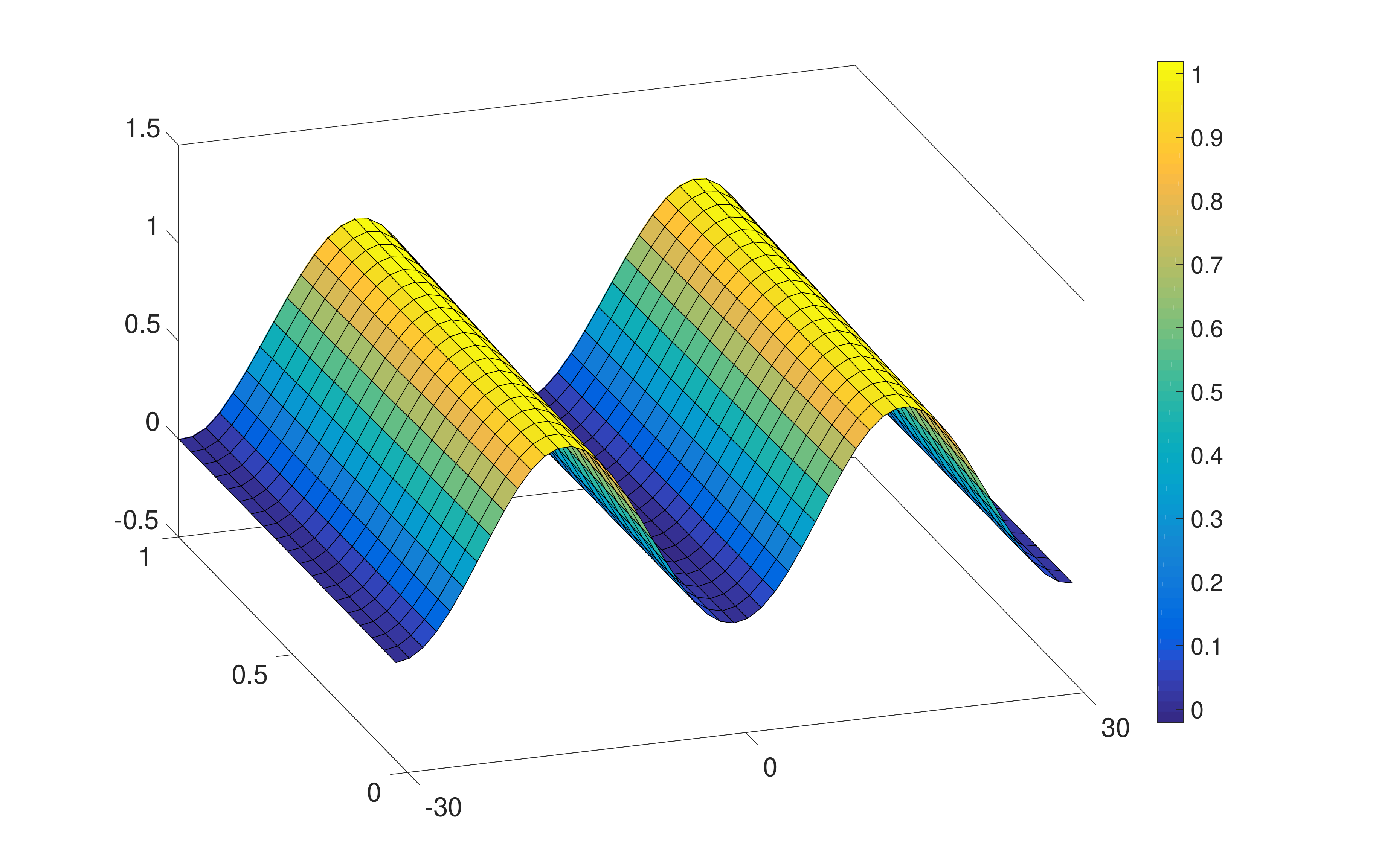}   		
	\end{center}
	\par
	\caption{Function $u_d(x,t)=\sin^2(\frac{\pi x}{30})+10^{-1}t(\cos(\frac{\pi x}{30})+1)$ (left) and state function $u(x,t)$ (right). 	
	$T=1s,$\, $N=50, \Delta t=2\times10^{-2}$ and $\ell=40, \gamma=40$.}\label{fig.state1target1.ks}
	\end{figure}	

	\begin{figure}[ht!]
	\begin{center}
		\includegraphics[scale=0.19]{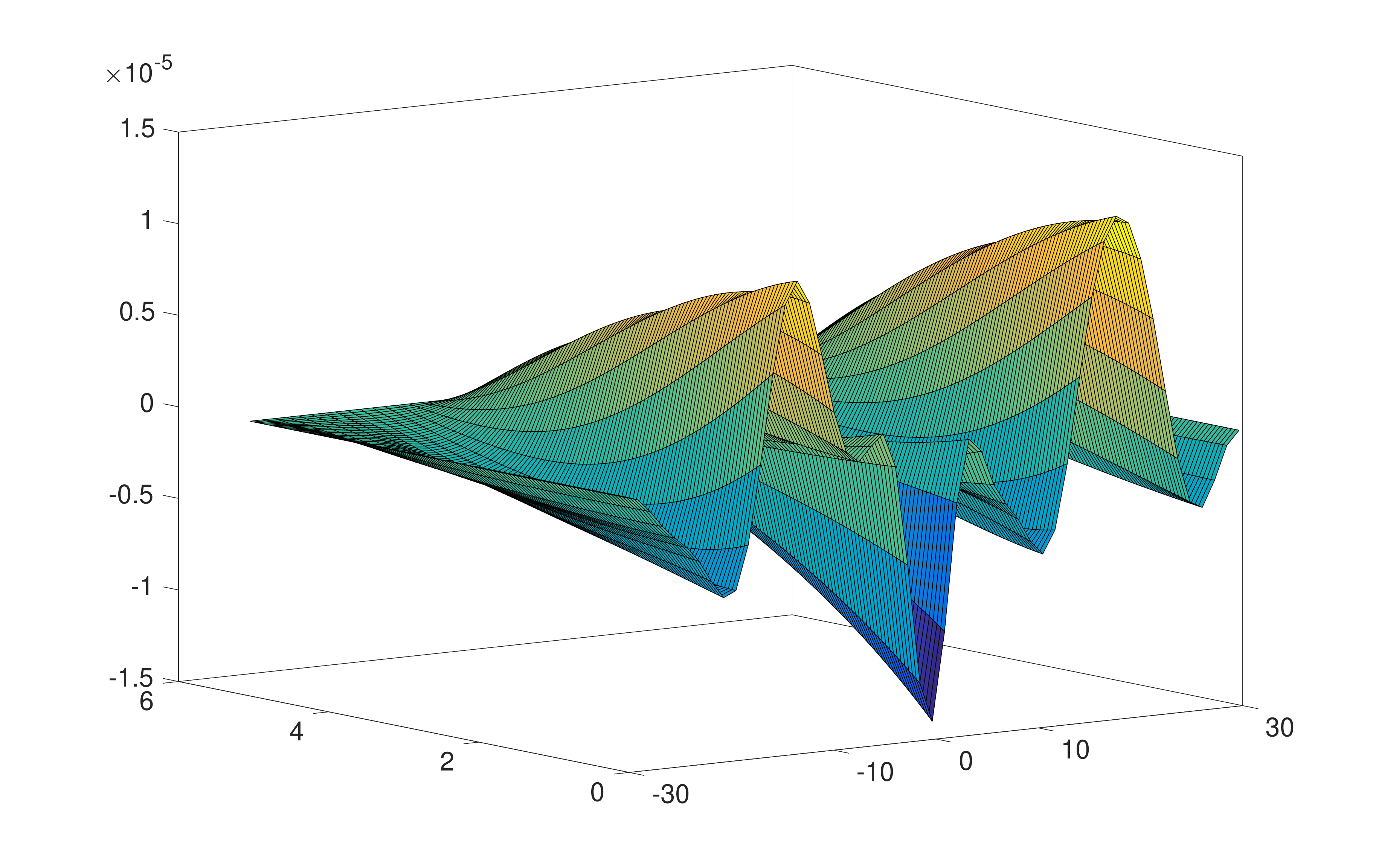} 
		\includegraphics[scale=0.19]{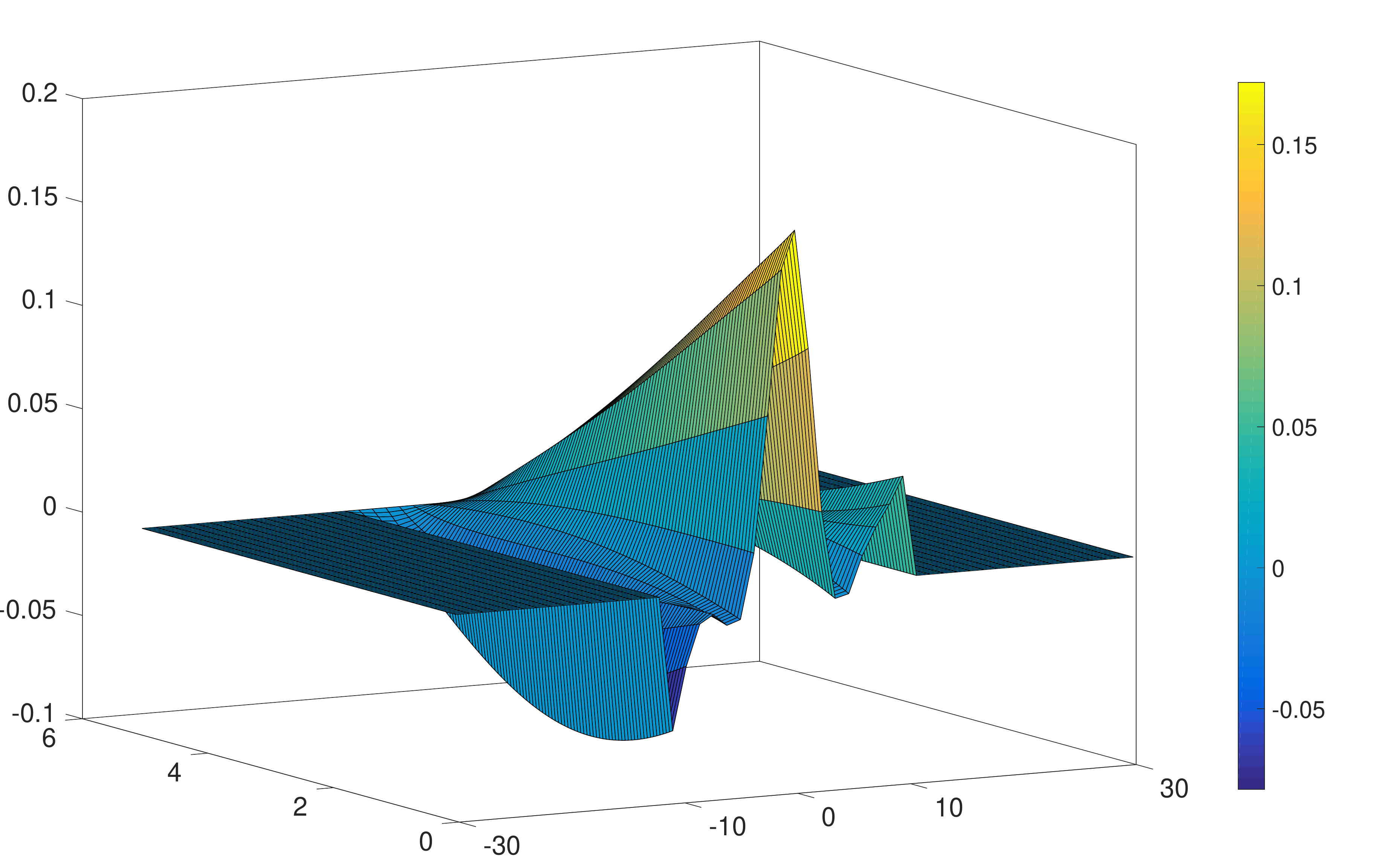}   		
	\end{center}
	\par
	\caption{Disturbance signal $\psi$  (left) in the interval $(-30,30)$ and control function $v$ (right) with support in $\mathcal{O}=(-10,10)$. 
	$T=5s, N=50, \Delta t=2\times 10^{-2}$ and 
	 $\ell=4, \gamma=400$.}\label{fig.control2disturbance2.ks}
	\end{figure}
	
	\begin{figure}[h!]
	\begin{center}
		\includegraphics[scale=0.21]{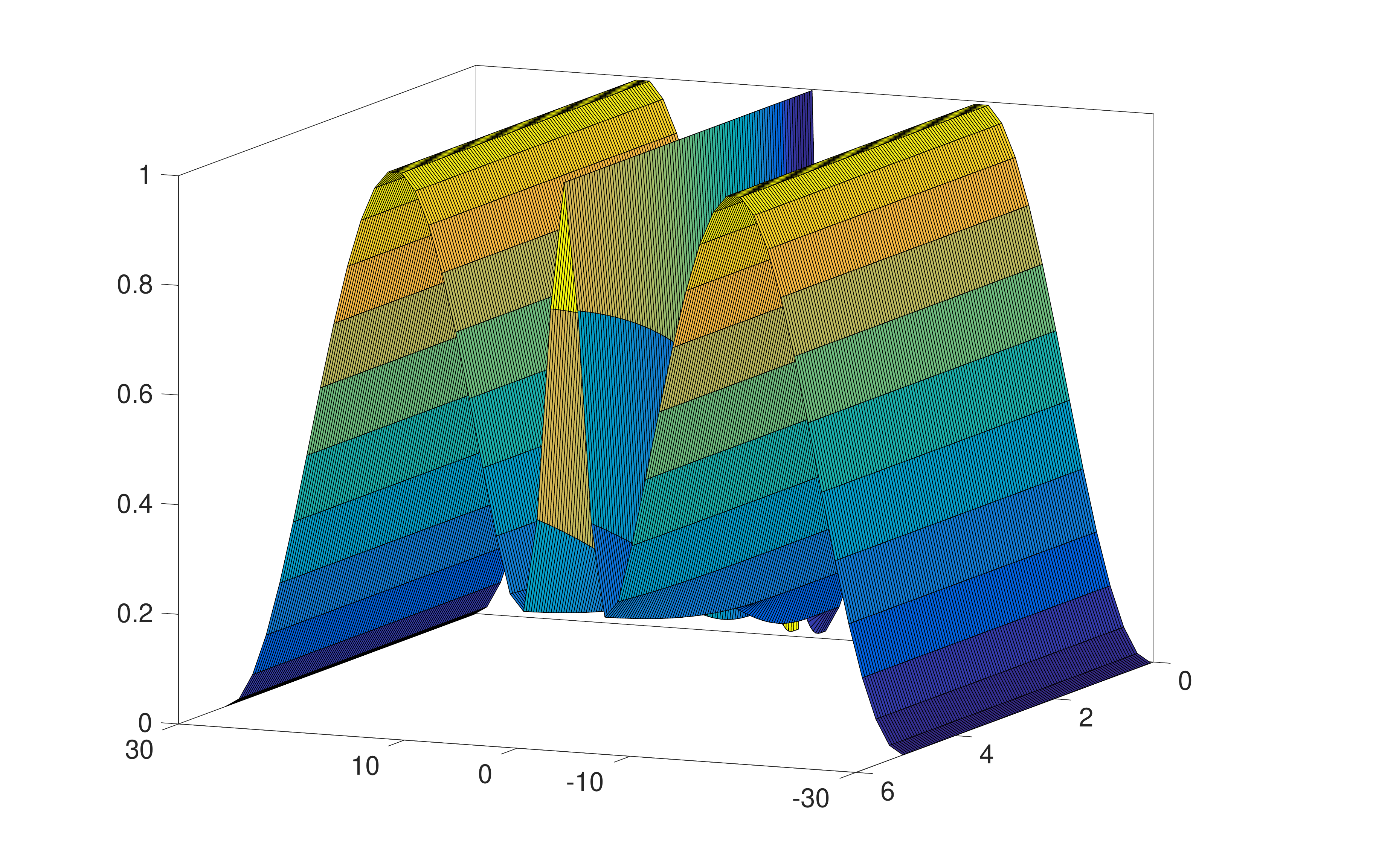} 
		\includegraphics[scale=0.21]{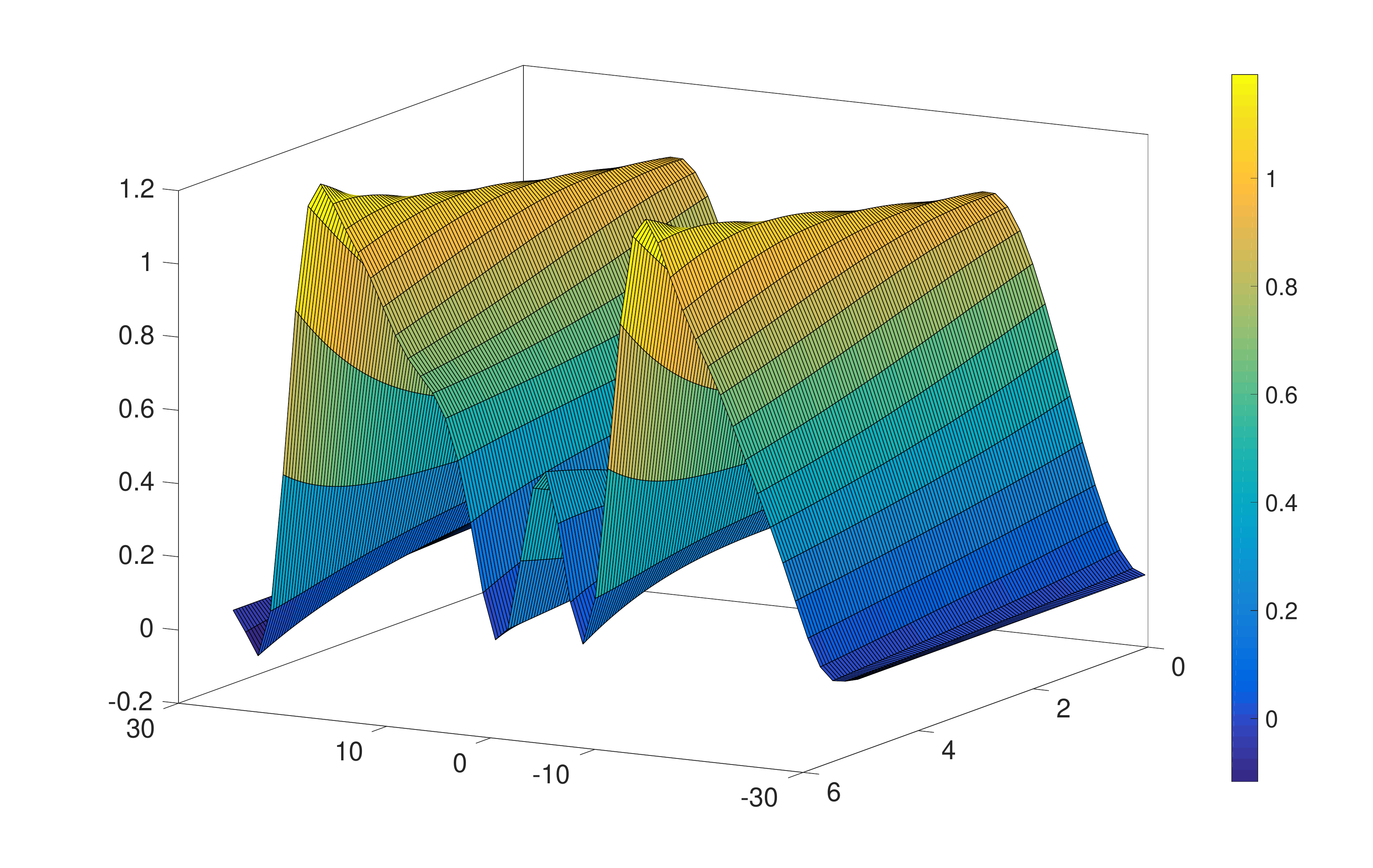}   			
	\end{center}
	\par
	\caption{Function $u_d(x,t)=\exp(-x^2)+\sin^2(\frac{\pi x}{30})$ (left) and state function $u(x,t)$ (right). $T=5s, N=50, \Delta t=2\times 10^{-2}$ and 
	 $\ell=4, \gamma=400$.}\label{fig.state2target2.ks}
	\end{figure}		
	
	\begin{figure}[h!]
	\begin{center}
		\includegraphics[scale=0.21]{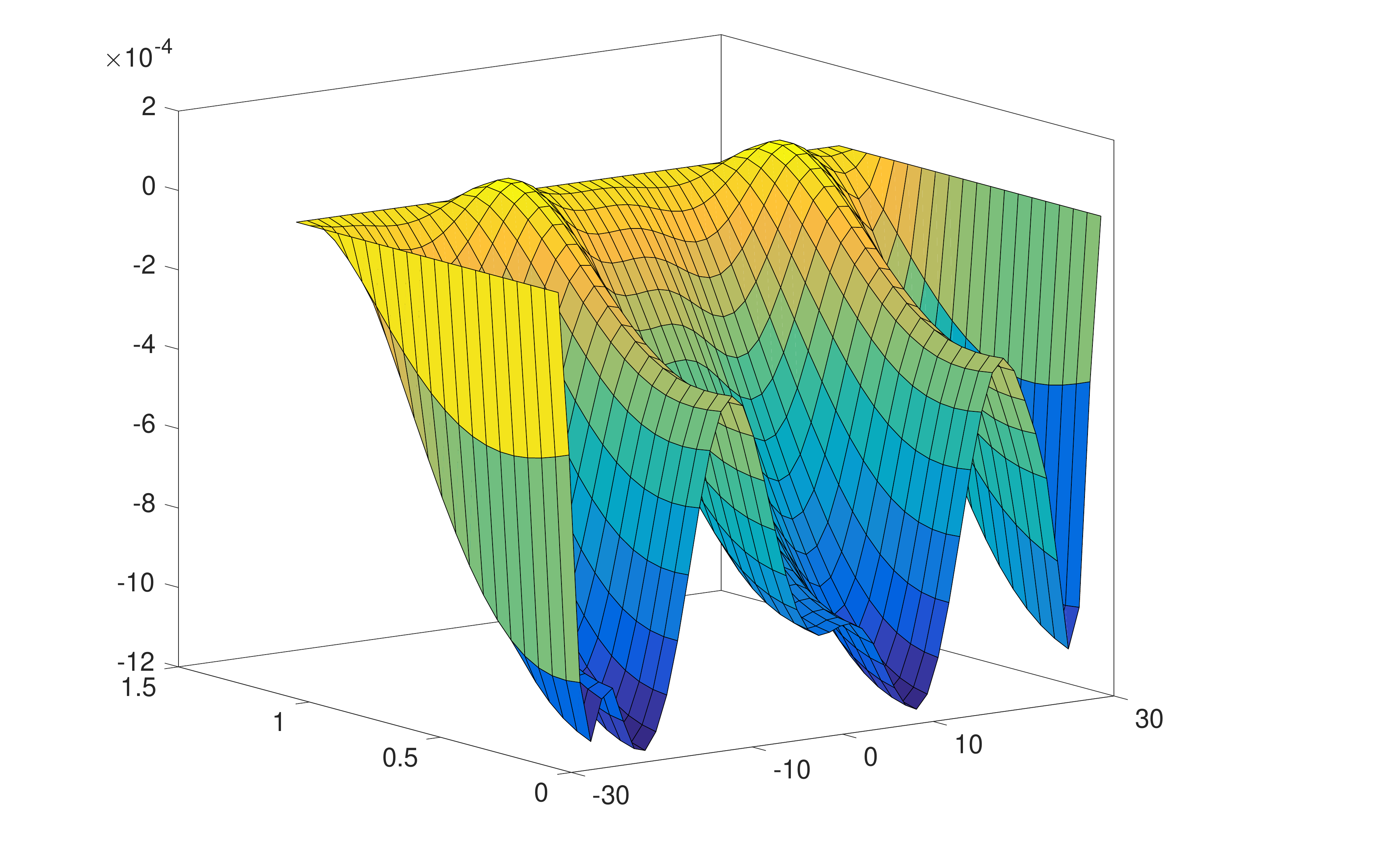} 
		\includegraphics[scale=0.21]{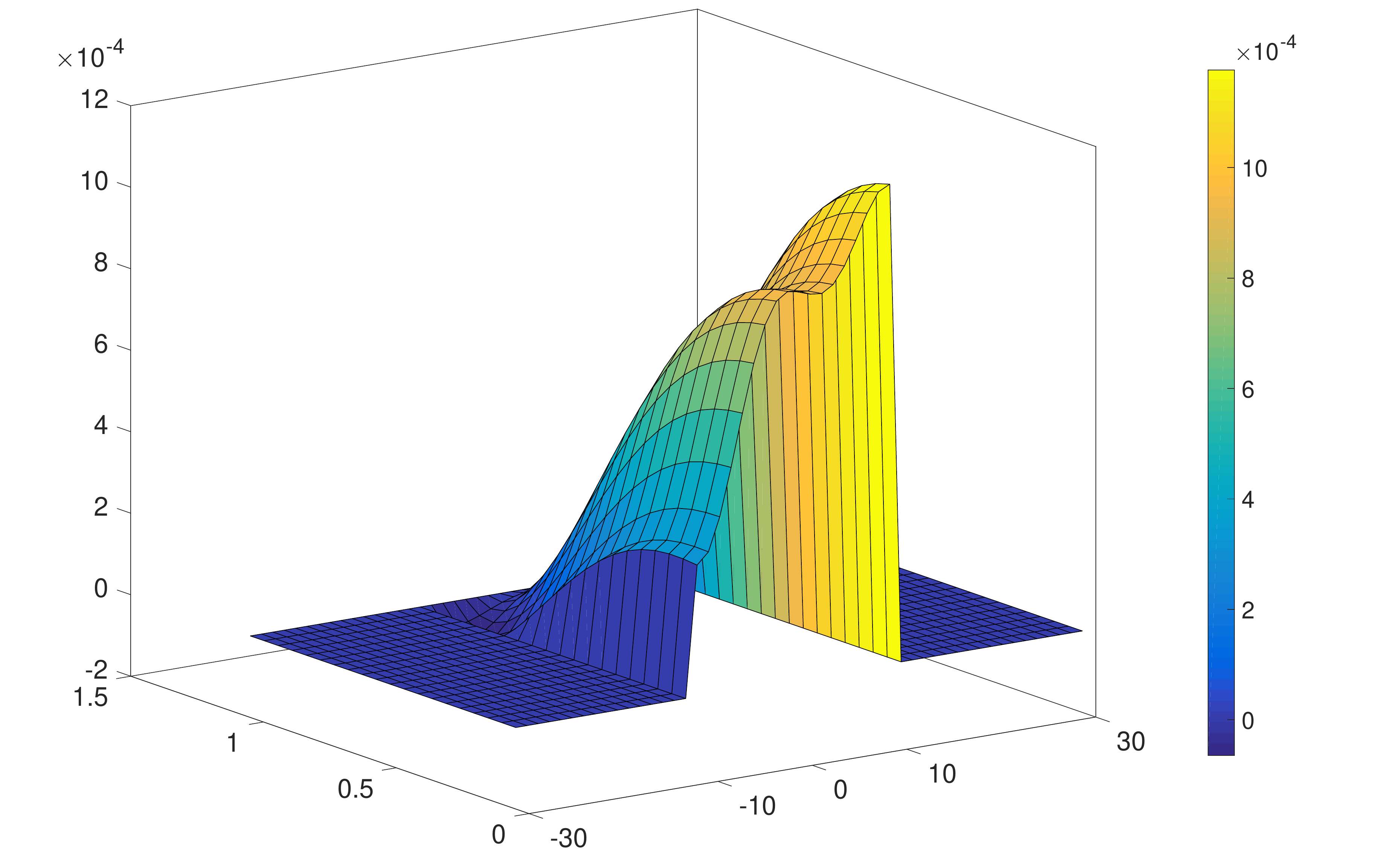}   		
	\end{center}
	\par
	\caption{Disturbance signal $\psi$  (left) in the interval $(-30,30)$ and control $v$ (right) with support in  $\mathcal{O}=(-10,10)$.
	$T=1s, N=50, \Delta t=2\times 10^{-2}$ and $\ell=\gamma=10$.}
	\label{fig.control3disturbance3.ks}
	\end{figure}	
	
	\begin{figure}[h!]
	\begin{center}
		\includegraphics[scale=0.19]{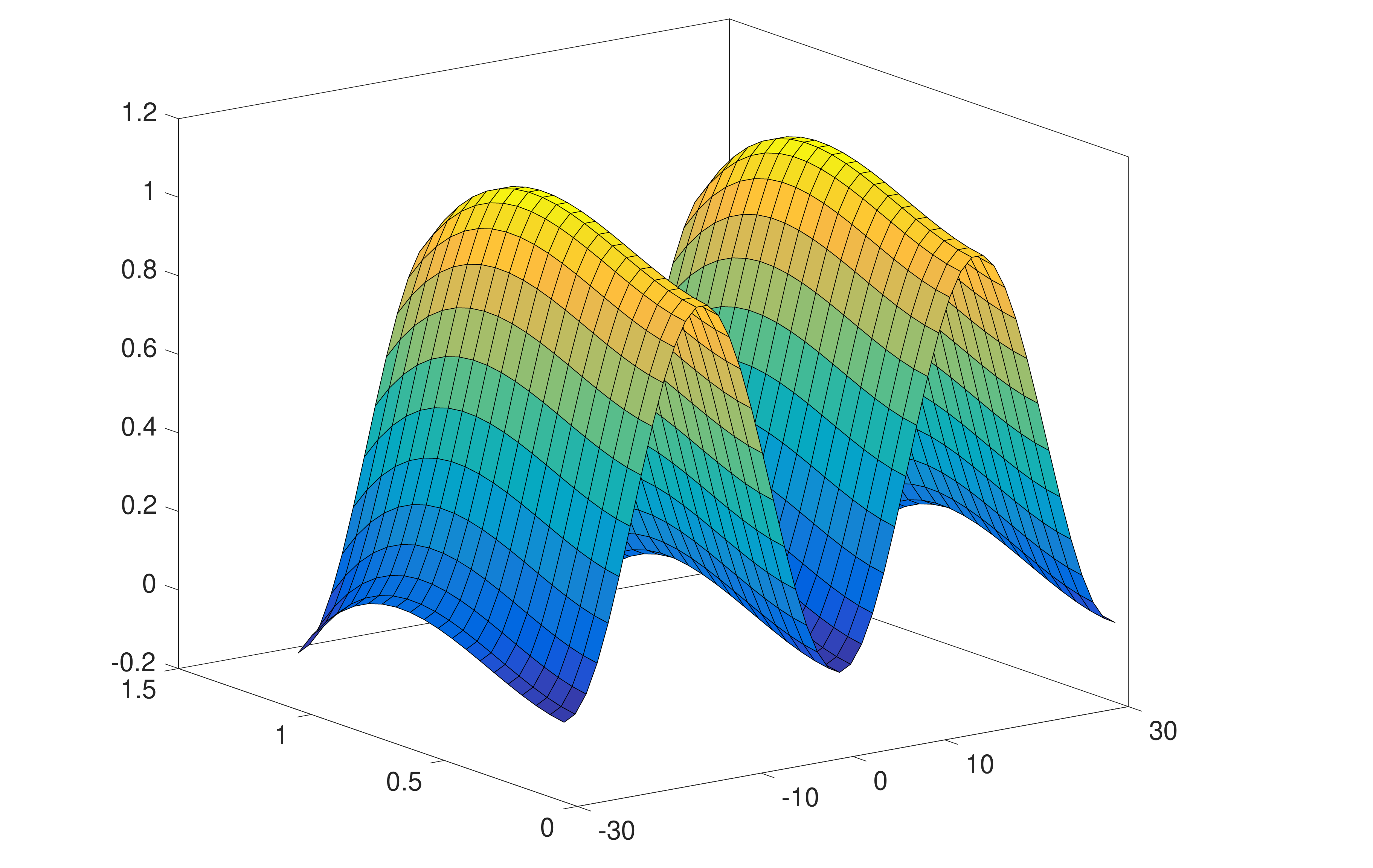} 
		\includegraphics[scale=0.19]{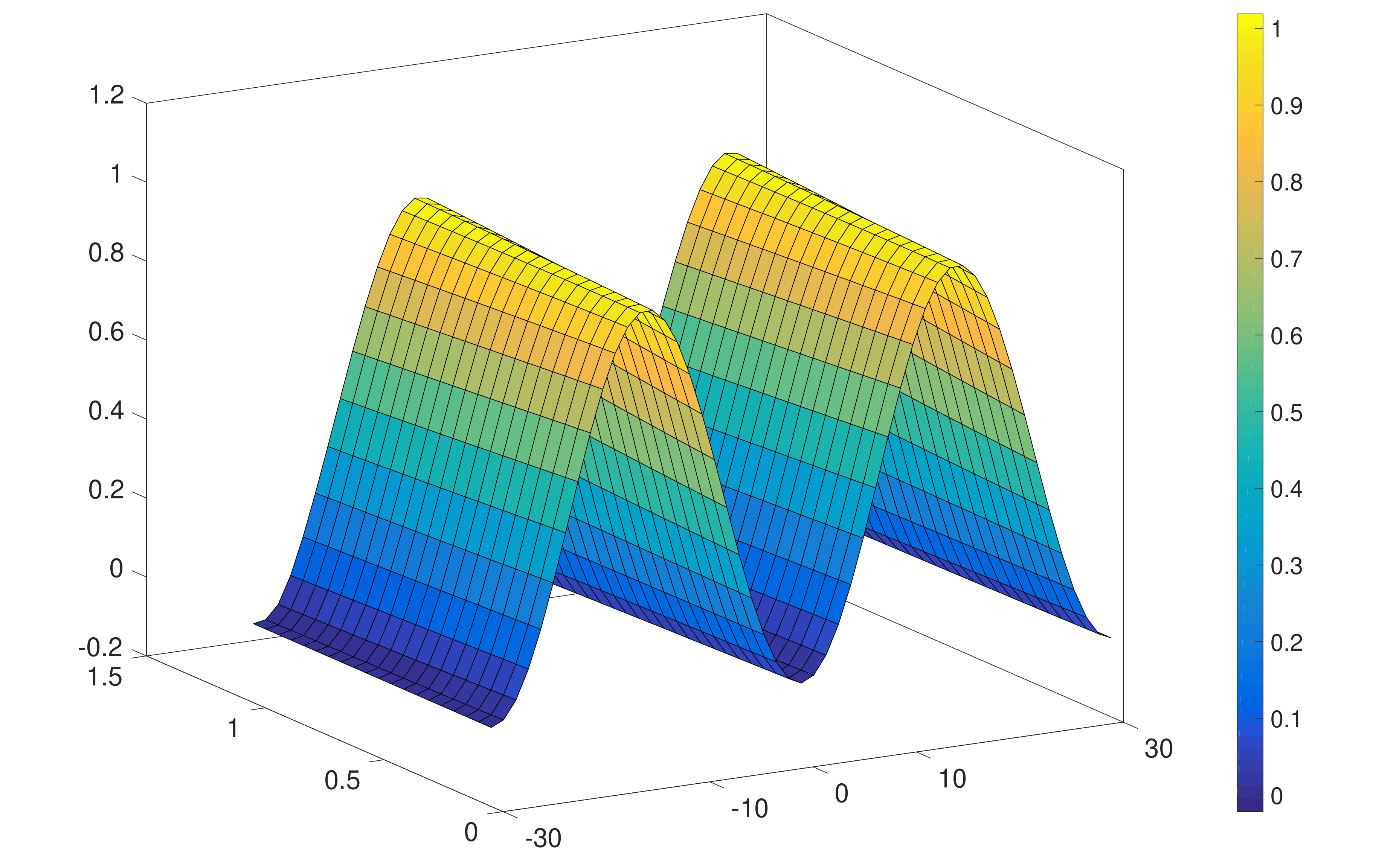}   			
	\end{center}
	\par
	\caption{Function $u_d(x,t)=(-t^{3}+t^{2})+\sin^2(\frac{\pi x}{30})$  (left) and state function $u(x,t)$ (right). 
	$T=1s, N=50, \Delta t=2\times 10^{-2}$ and $\ell=\gamma=10$.}\label{fig.state3target3.ks}
	\end{figure}			
\section{\normalsize{Controllability}}\label{section.controllability}
	In the previous section the robust control problem was characterized by a coupled system which needs to be solved. 
	In order to establish a Stackelberg strategy for the case in which the leader control leads the state to the trajectory
	in a finite time, we must find
	a function $h\in L^2(0,T;L^2(\omega))$ such that the corresponding $u$ solution to \eqref{eq.ks} satisfies 
	$u(T)=\overline{u}(T)$, with $\overline{u}$ solution to \eqref{intro.target.u}. To be precise, to 
	prove the exact controllability to the trajectories, we consider two relevant control systems,
	namely, the linearized system of \eqref{eq.ks.characterization} around $\overline{u}$ which is
	\begin{equation}\label{sys.linearized.coupled}
	\begin{cases}
		u_{t}+u_{xxxx}+u_{xx}+\overline{u}u_x+\overline{u}_xu= f_1+
		h1_{\omega}+ (-\ell^{-2}1_{\mathcal{O}}+\gamma^{-2})z & \text{in }Q,\\
		-z_{t}+z_{xxxx}+z_{xx}-(u+\overline{u})z_{x}=f_2+(u-u_{d})1_{\mathcal{O}_{d}} & \text{in }Q,\\
		u(0,t)=u(1,t)=z(0,t)=z(1,t)=0 & \text{on }(0,T),\\
		u_{x}(0,t)=u_{x}(1,t)=z_{x}(0,t)=z_{x}(1,t)=0 & \text{on }(0,T),\\
		u(\cdot,0)=u_{0}(\cdot),\,\, z(\cdot,T)=0& \text{in }(0,1).
	\end{cases}
	\end{equation}	
	and the adjoint system associated to \eqref{sys.linearized.coupled}
	 \begin{equation}\label{sys.adjoint.coupled}
    \left\{
    \begin{array}{lll}
    \begin{array}{llll}
        -\varphi_{t}+\varphi_{xxxx}+\varphi_{xx}-\overline{u}\varphi_x
        =g_1+\theta1_{\mathcal{O}_d} & \text{in } Q,\\
        \theta_{t}+\theta_{xxxx}+\theta_{xx}+(\overline{u}\theta)_x
        =g_2 -\ell^{-2}\varphi 1_{\mathcal{O}}+\gamma^{-2}\varphi  & \text{in } Q,\\        
       \varphi(0,t)=\varphi(1,t)=\theta(0,t)=\theta(1,t)=0 & \text{on }(0,T),\\
		\varphi_{x}(0,t)=\varphi_{x}(1,t)=\theta_{x}(0,t)=\theta_{x}(1,t)=0 & \text{on }(0,T),\\
        \varphi(\cdot,T)= \varphi_T(\cdot),\, \theta(\cdot,0)=0& \text{in }(0,1).
    \end{array}
    \end{array}
    \right.
	\end{equation}
	where $f_1,f_2,g_1,g_2$ and $u_0,\varphi_T$ are in appropriate spaces.

	Our strategy is as follows:
	\begin{enumerate}
	\item [i)] Establish first a global Carleman inequality for the system \eqref{sys.adjoint.coupled}. Those inequality allows us to prove a null controllability result for the linearized system \eqref{sys.linearized.coupled} with
		right--hand side satisfying suitable decreasing properties near $t=T$. 
	\item [ii)] Afterwards, to establish	 the local exact controllability to the trajectories for the KS system. Here, fixed 
		point arguments will be used. 
	\end{enumerate}

\subsection{\normalsize{Carleman inequalities}}
	We first define some weight functions which will be useful in the sequel. Let $\omega$ and $\omega_0$ be non empty 
	subsets of $(0,1)$ such that $\omega_0\subset\subset\omega\cap\mathcal{O}_d$ and $\eta\in C^4([0,1])$ such that 
	\begin{equation*}
		|\nabla \eta|>0 \mbox{ in }[0,1]\setminus\overline\omega_0,\,\,\,\, \eta>0
     	\mbox{ in }(0,1)\,\,\, \mbox{ and }\,\, \eta(0)=\eta(1)=0.
     \end{equation*}
     The existence of such a function is proved in \cite{1996FursikovImanuvilov}. For some positive real number 
     $\lambda$, we consider the weight functions:
	\begin{equation}\label{carleman_weights}
    \begin{array}{lll}
    	&\alpha(x,t)=\dfrac{e^{7\lambda\|\eta\|_{\infty}}-e^{\lambda(2\|\eta\|_{\infty}+\eta(x))}} 
    	{(t(T-t))^{2/5}},\quad \xi(x,t)=\dfrac{e^{\lambda(2\|\eta\|_{\infty}+\eta(x))}}{(t(T-t))^{2/5}},\\
    	&\widehat\alpha(t)=\max\limits_{x\in [0,1]} \alpha(x,t),\quad \quad
    	\widehat\xi(t)=\max\limits_{x\in[0,1]} \xi(x,t),\\
    	&\breve\alpha(t)=\min\limits_{x\in[0,1]} \alpha(x,t),\quad 
    	\quad \,\,\,\,\breve\xi(t)=\min\limits_{x\in[0,1]} \xi(x,t).
    \end{array}
	\end{equation}

    Henceforth, the constants $a_0$ and $m_0$ are fixed, and satisfy  
	\begin{equation}\label{constantes}
		\frac{5}{4}\leq a_0<a_0+1<m_0<2a_0,\quad m_0<2+a_0. 
	\end{equation}
	Moreover, we will use the following notation for the weighted energy:
	\[I_{0}(\rho, u)=\int\limits_{0}^T\int\limits_0^1\rho(s^{-1}\xi^{-1}(|u_t|^2+|u_{xxxx}|^2)dxdt+I_1(\rho, u),\]
	\[I_1(\rho, u)= \int\limits_{0}^T\int\limits_0^1\rho(s\lambda^2\xi|u_{xxx}|^2+s^3\lambda^4\xi^3|u_{xx}|^2
		+s^5\lambda^6\xi^5|u_x|^2+s^7\lambda^8\xi^7|u|^2)dxdt,\]
	and we also recall the space 		
	\[\mathcal{Z}:=C([0,T];H_0^2(0,1))\cap L^2(0,T;H^4(0,1))\cap L^\infty(0,T;W^{1,\infty}(0,1)).\]

	Our Carleman estimate is given in the the following proposition.

	\begin{proposition}\label{teo.carleman_carleman1}
    Let $\overline{u}\in \mathcal{Z}$ and assume that $\omega\cap \mathcal{O}_d\neq \emptyset$ and that 
    $\ell$ and $\gamma$ are large enough.
    Then, there exist a constant  $\overline{\lambda}$ such that for any $\lambda\geq\overline{\lambda}$
    exist two constants $\overline{s}(\lambda)>0$ and $C=C(\lambda)>0$ depending only on 
    $\omega$  such that  for any $g_1,g_2\in L^2(Q)$ and any $\varphi_T\in L^2((0,1))$,  the solution of
    \eqref{sys.linearized.coupled} satisfies
    \begin{equation}\label{carleman_obs_ineq1}
    \begin{aligned}
        I_1(e^{-2s\alpha-2a_0s\widehat\alpha},\theta)+&I_0(e^{-2sm_0\alpha},\varphi)
        \leq C\Biggl(s^{15}\lambda^{16}\displaystyle\iint\limits_{\omega\times(0,T)}
        e^{-2s\breve\alpha-2a_0s\widehat\alpha}(\widehat\xi)^{29}|\varphi|^2 dxdt\\
        &+s^{7}\lambda^{8}\displaystyle\iint\limits_{Q}
        e^{-2s\widehat\alpha-2a_0s\widehat\alpha}(\widehat\xi)^{7}|g_1|^2dxdt
        +s^{7}\lambda^{8}\displaystyle\iint\limits_{Q}e^{-2a_0s\widehat\alpha}|g_2|^2dxdt\Biggr),
    \end{aligned}
    \end{equation}
    for any $s\geq \overline{s}$.
	\end{proposition}
	Before giving the proof of Proposition \ref{teo.carleman_carleman1}, we recall some technical results. Let us introduce
	the system
	\begin{equation}\label{eq.ks.aux1}
	\begin{cases}
		u_{t}+u_{xxxx}+u_{xx}+\overline{u}u_x+\overline{u}_xu=f & \text{in }Q,\\
		u(0,t)=u(1,t)=u_{x}(0,t)=u_{x}(1,t)=0 & \text{on }(0,T),\\
		u(\cdot,0)=u_{0}(\cdot)& \text{in }(0,1),
	\end{cases}
	\end{equation}	
	where $f\in L^2(Q)$ and $\overline{u}\in \mathcal{Z}$.
	
	\begin{lemma}\label{lema.cerpaetal2015}
    	Assume  $f\in L^2(Q)$ and  $\omega\subset (0,1)$. Then, there exist positive constants $C(\omega),\,s_1$ and 
    	$\lambda_1$ such that 
    	\begin{equation}\label{carleman.cerpaetal2015}
        I_1(e^{-2s\alpha},u) \leq I_0(e^{-2s\alpha},u)\leq C\Bigl(\int\limits_0^T\int\limits_{0}^1 e^{-2s\alpha}|f|^2 dxdt
        	+s^7\lambda^8\int\limits_0^T\int\limits_{\omega}e^{-2s\alpha}\xi^7|u|^2 dxdt \Bigr),
    	\end{equation}
    	for every $s\geq s_1,\,\lambda\geq \lambda_1$, and $u$ solution to \eqref{eq.ks.aux1} with 
    	$\overline{u}\in \mathcal{Z}$.
	\end{lemma}
	\begin{Obs}
		Carleman inequality of Lemma \ref{lema.cerpaetal2015} was proven in \cite{2015-cerpaatal} with $\overline{u}=0$. 
		However, thanks to the Carleman weight functions and the fact that $\overline{u}\in \mathcal{Z}$ its extension 
		to \eqref{eq.ks.aux1} is direct. Besides, in \cite{2015-cerpaatal} slightly different weight functions are used 
		to prove Lemma \ref{lema.cerpaetal2015}. Nevertheless, 
    	the inequality remains valid since the key point of the proof is that $\alpha$ goes to $+\infty$ when $t$ tends to $0$ 
    	and $T$. In addition, there exists another Carleman estimate for the system
		\eqref{eq.ks.aux1} (with $\overline{u}=0$) \cite{2012Zhou}. To our propose is convenient to use \cite{2015-cerpaatal}
		instead of \cite{2012Zhou}.
	\end{Obs}

	\begin{Obs}\label{obs.carleman.cerpaetal2015obs}
		A direct consequence of the weight functions \eqref{carleman_weights} shows that, the first term
		in the right--hand side of \eqref{carleman.cerpaetal2015} can be upper bounded by the term 
		$\|e^{-2s\breve\alpha+s\widehat\alpha}f\|^2_{L^2(Q)}$. Therefore, \eqref{carleman.cerpaetal2015} is transformed
		in 
		\begin{equation}\label{carleman.cerpaetal2015obs}
        			I_1(e^{-2s\alpha},u) \leq I_0(e^{-2s\alpha},u)\leq C\Bigl(\int\limits_0^T\int\limits_{0}^1
			e^{-4s\breve\alpha+2s\widehat\alpha}|f|^2 dxdt
        			+s^7\lambda^8\int\limits_0^T\int\limits_{\omega}e^{-2s\alpha}\xi^7|u|^2 dxdt \Bigr),
    		\end{equation}
		for every $s\geq s_1,\,\lambda\geq \lambda_1$, and $u$ solution to \eqref{eq.ks.aux1}.
	\end{Obs}

	Another result holds from the relation between the weight function $\widehat\alpha$ and $\breve\alpha$. The interested
	reader can see \cite{2018-robust-ns} for more details.
	
    \begin{lemma}\label{lema.relationship.weights}
		For any $\varepsilon>0$, any $M_1,M_2\in \mathbb{R}$, there exists $\lambda_0>0$ and $C=C(\varepsilon,M_1, M_2)>0$ 
		such that 
		\begin{equation}\label{c.intro.inequality.special}
    		e^{s\widehat\alpha}\leq Cs^{M_1}\lambda^{M_2}(\breve\xi)^{M_1}e^{s(1+\varepsilon)\breve\alpha},
		\end{equation}
    	for every $\lambda>\lambda_0$. 
	\end{lemma}
	\begin{Obs}
    	In relation to Lemma \ref{lema.relationship.weights}, it was proven in \cite{2018-robust-ns} for $\widehat\xi$ 
    	instead of $\breve\xi$, nevertheless, it is easy to verify that the same arguments holds. 
    \end{Obs}
   
	Now, in order to give the proof of Proposition \ref{teo.carleman_carleman1}, we adapt the structure made 
	by Montoya and deTeresa in \cite{2018-robust-ns}. More precisely, we must first make a Carleman estimate for $\theta$ 
	with appropriate weight functions. Afterwards, another Carleman inequality for $\varphi$ will be established. 
	The weight functions should be such that all terms respecto to $\theta$ in the right--hand side can be absorbed by the left--hand side. Finally, to estimate
	local terms of $\theta$, we will use the geometric condition $\mathcal{O}_d\cap\omega_0\neq \emptyset$.
	\begin{proof}
	{\bf{Carleman estimate for $\theta$.}}
		Let define $\theta^*:=\rho^*\theta$, where $\rho^*=\rho^*(t)=e^{-a_{0}s\widehat\alpha}$ 
    	and $a_0$ fixed satisfying \eqref{constantes}.
    	From \eqref{sys.adjoint.coupled}, $\theta^*$ is the solution of the following system
		\begin{equation*}
    	\left\{
        \begin{array}{lll}
        \begin{array}{llll}
        	\theta^*_{t}+\theta^*_{xxxx}+\theta^*_{xx}+(\overline{u}\theta^*)_x=\rho^*g_2+
        	\rho^*(-\ell^{-2}\varphi1_{\mathcal{O}}+\gamma^{-2}\varphi)+\rho^*_{t}\theta & \text{in }Q,\\
			\theta^*(0,t)=\theta^*(1,t)=\theta^*_{x}(0,t)=\theta^*_{x}(1,t)=0 & \text{on }(0,T),\\
       		\theta^*(\cdot,0)=0& \text{in }(0,1).
        \end{array}
    	\end{array}\right.
		\end{equation*} 
   		Now, we decompose $\theta^*$ as follows: 
    	\begin{equation}\label{carleman_decompose}
        	\theta^*=\widehat{\theta}+\tilde{\theta},
    	\end{equation}    
   		where $\widehat{\theta}$ and $\tilde{\theta}$ solve respectively
		\begin{equation}\label{carleman_sys_thetatilde}
    	\left\{
        \begin{array}{lll}
        \begin{array}{llll}
        	\tilde{\theta}_{t}+\tilde{\theta}_{xxxx}+\tilde{\theta}_{xx}+(\overline{u}\tilde{\theta})_x=\rho^*g_2+
        	\rho^*(-\ell^{-2}\varphi1_{\mathcal{O}}+\gamma^{-2}\varphi) & \text{in }Q,\\
			\tilde{\theta}(0,t)=\tilde{\theta}(1,t)=\tilde{\theta}_{x}(0,t)=\tilde{\theta}_{x}(1,t)=0 & \text{on }(0,T),\\
       		\tilde{\theta}(\cdot,0)=0& \text{in }(0,1).
        \end{array}
        \end{array}
    	\right.
		\end{equation} 
    	and
		\begin{equation}\label{carleman_sys_thetahat}
    	\left\{
        \begin{array}{lll}
        \begin{array}{llll}
        	\widehat{\theta}_{t}+\widehat{\theta}_{xxxx}+\widehat{\theta}_{xx}+(\overline{u}\widehat{\theta})_x=\rho^*_{t}\theta & \text{in }Q,\\
			\widehat{\theta}(0,t)=\widehat{\theta}(1,t)=\widehat{\theta}_{x}(0,t)=\widehat{\theta}_{x}(1,t)=0 & \text{on }(0,T),\\
       		\widehat{\theta}(\cdot,0)=0& \text{in }(0,1).        \end{array}
        \end{array}
    	\right.
		\end{equation}    
		For system \eqref{carleman_sys_thetatilde} we will use 
		Lemma \ref{apendix.lema.strong.linear} (see Appendix \ref{appendix1}) with the higher regularity, meanwhile for 
		the system \eqref{carleman_sys_thetahat} we will use some ideas of \cite{2018-robust-ns}. 
    	
    	Using Lemma \ref{lema.cerpaetal2015} with $f=\rho^*_{t}\theta$ and $u=\widehat\theta$, there exists a positive constant 
    	$C=C(\omega_0)$ such that 
    	\begin{equation}\label{carleman_ine_aux0}
        	I_1(e^{-2s\alpha},\widehat\theta)\leq C\Bigl(\int\limits_0^T\int\limits_{0}^1 e^{-2s\alpha}|\rho^*_{t}\theta|^2 dxdt
        	+s^7\lambda^8\int\limits_0^T\int\limits_{\omega_0}e^{-2s\alpha}\xi^7|\widehat\theta|^2 dxdt \Bigr),
    	\end{equation}
		for any $\lambda_1:=\lambda\geq C$ and $s\geq s_1$.
		
		Now, using the inequality $\frac{a^2}{2}-b^2\leq (a-b)^2$, for every $a,b\in \mathbb{R}$, with $a=\theta^*$
		 and $b=\tilde{\theta}$, we get (recall that $\widehat{\theta}=\theta^*-\tilde{\theta}$):
		 \begin{equation}\label{carleman_ine_aux1}
		 	\frac{1}{2}I_1(e^{-2s\alpha},\theta^*)-I_1(e^{-2s\alpha},\tilde\theta)\leq I_1(e^{-2s\alpha},\widehat\theta).
		 \end{equation}
    	 Since $s^7\lambda^8\xi^7e^{-2s\alpha}$ is upper bounded, it allows to estimate the terms involved in 
    	 $I_1(e^{-2s\alpha},\tilde\theta)$ using the regularity inequality \eqref{reg.high.linear} of Lemma
    	 \ref{apendix.lema.strong.linear}. In fact,  we have:
    	 \begin{equation}\label{carleman_aux2}
    		\begin{array}{ll}
   				I_1(e^{-2s\alpha},\tilde\theta)&\leq \displaystyle C_{s,\lambda}\|\tilde\theta\|^2_{L^2(0,T;H^4(0,1)\cap H_0^2(0,1))}\\
    			&\leq C_{s,\lambda}\|\rho^*g_2\|^2_{L^2(Q)}
    			+C_{s,\lambda}\|\rho^*(-\ell^{-2}\varphi1_{\mathcal{O}}+\gamma^{-2}\varphi)\|^2_{L^2(Q)}, 
    	\end{array}
		\end{equation}    
    	where $C_{s,\lambda}$ is a positive constant depending on $s$ and $\lambda$, i.e., $C_{s,\lambda}=Cs^7\lambda^8$.
    	
    	On the other hand, taking into account that $|\rho^*_{t}|\leq Cs\rho^*(\xi^*)^{7/2}$ for every $s\geq C$,
    	it follows that
    	\[\int\limits_0^T\int\limits_{0}^1 e^{-2s\alpha}|\rho^*_{t}\theta|^2 dxdt
    	\leq Cs^2\int\limits_0^T\int\limits_{0}^1 e^{-2s\alpha-2a_0s\widehat\alpha}(\widehat\xi)^{7}|\theta|^2 dxdt,\]
        which can be absorbed by the first term in the left--hand side of \eqref{carleman_ine_aux1}, for every 
    	$\lambda\geq 1,\, s\geq C$.
    	
    	Now, to estimate the local term that appear in the right--hand side of \eqref{carleman_ine_aux0}, we use the 
    	identity $\theta^*=\widehat\theta+\tilde\theta$ (recall \eqref{carleman_decompose}). Thus, we have 
    	\begin{equation}\label{carleman_aux4}
    	\begin{aligned}
        	s^7\lambda^8\int\limits_0^T\int\limits_{\omega_0}e^{-2s\alpha}\xi^7|\widehat\theta|^2 dxdt
        	&\leq Cs^7\lambda^8\int\limits_0^T\int\limits_{\omega_0}e^{-2s\alpha}\xi^7(|\tilde\theta|^2+|\theta^*|^2) dxdt\\
        	&\leq Cs^7\lambda^8\int\limits_0^T\int\limits_{\omega_0}e^{-2s\alpha}\xi^7|\theta^*|^2dxdt
        	+C_{s,\lambda}\|\rho^*g_2\|^2_{L^2(Q)}\\
        	&\hspace{1cm} +C_{s,\lambda}\|\rho^*(-\ell^{-2}\varphi1_{\mathcal{O}}+\gamma^{-2}\varphi)\|^2_{L^2(Q)}.    
    	\end{aligned}
		\end{equation}
		Putting together \eqref{carleman_ine_aux0}--\eqref{carleman_aux4}, we have for the moment
		\begin{equation}\label{carleman_end_step1}
    	\begin{array}{ll}
        	I_1(e^{-2s\alpha-2a_0s\widehat\alpha},\theta)&\leq Cs^7\lambda^8\displaystyle\int\limits_0^T\int\limits_{\omega_0}
        	e^{-2s\alpha-2a_0s\widehat\alpha}\xi^7|\theta|^2dxdt
        	+C_{s,\lambda}\|\rho^*g_2\|^2_{L^2(Q)}\\
        	&\hspace{1cm} +C_{s,\lambda}\|\rho^*(-\ell^{-2}\varphi1_{\mathcal{O}}+\gamma^{-2}\varphi)\|^2_{L^2(Q)},   	
        \end{array}
		\end{equation}
    	for every $s\geq C$ and $\lambda_1:=\lambda\geq C$.

	\noindent{\bf{Carleman estimate for $\varphi$.}}
	First, assuming that $\theta$ is given, we look at $\varphi$ as the solution of
	\begin{equation}\label{carleman_adjoint_system}
    	\left\{
    	\begin{array}{lll}
    	\begin{array}{llll}
        	-\varphi_{t}+\varphi_{xxxx}+\varphi_{xx}-\overline{u}\varphi_x
        	=g_1+\theta1_{\mathcal{O}_d} & \text{in }Q,\\
       		\varphi(0,t)=\varphi(1,t)=\varphi_{x}(1,t)=\varphi_{x}(0,t)=0 & \text{on }(0,T),\\
        	\varphi(\cdot,T)= \varphi_T(\cdot)& \text{in }(0,1).
    	\end{array}
    	\end{array}
    	\right.
	\end{equation}
	Applying Lemma \ref{lema.cerpaetal2015} jointly with its remark \ref{obs.carleman.cerpaetal2015obs}
	for $f=g_1+\theta1_{\mathcal{O}_d}$ and the weight function
	$m_0\alpha$ (instead of $\alpha$), where $a_0+1<m_0\leq 2a_0$ and $m_0\leq 2+a_0$, we obtain
	\begin{equation}\label{carleman1step2}
	\begin{array}{lll}
		I_0(e^{-2m_0s\alpha},\varphi)&\leq C\Biggl(\displaystyle\int\limits_0^T\int\limits_{0}^1 e^{-4m_0s\breve\alpha
		+2m_0s\widehat\alpha}|g_1|^2 dxdt
        +\displaystyle\int\limits_0^T\int\limits_{\mathcal{O}_d}e^{-4m_0s\breve\alpha+2m_0s\widehat\alpha}|\theta|^2 dxdt\\
        &\hspace{1cm}+s^7\lambda^8\displaystyle\int\limits_0^T\int\limits_{\omega_0}e^{-2m_0s\alpha}\xi^7|\varphi|^2 dxdt 
        \Biggr),
	\end{array}
	\end{equation}
	for any $\lambda_2:=\lambda\geq C$ and $s\geq C$.
		
		By considering Lemma \ref{lema.relationship.weights} with  $\varepsilon=\frac{m_0-a_0-1}{m_0+a_0+1}$, 
    	$M_1=\frac{7}{2(m_0+a_0+1)}$ and $M_2=\frac{4}{(m_0+a_0+1)}$, the second term in the right--hand side 
    	of \eqref{carleman1step2} can be estimated by $I_1(e^{-2s\widehat\alpha-2a_0s\widehat\alpha},\theta)$
    	and therefore it can be absorbed by the left--hand side of \eqref{carleman_end_step1}.
		
	From \eqref{carleman_end_step1} and \eqref{carleman1step2}	we have
	\begin{equation}\label{c.theta.and.varphi}
    	\begin{aligned}
        	&I_1(e^{-2s\alpha-2a_0s\widehat\alpha},\theta)+I_0(e^{-2m_0s\alpha},\varphi)\\
        	&\leq Cs^7\lambda^8\int\limits_0^T\int\limits_{\omega_0}e^{-2m_0s\alpha}\xi^7|\varphi|^2 dxdt
        	+Cs^7\lambda^8\int\limits_0^T\int\limits_{\omega_0}e^{-2s\alpha-2a_0s\widehat\alpha}\xi^7|\theta|^2dxdt
        	+C_{s,\lambda}\|\rho^*g_2\|^2_{L^2(Q)}\\
        	&\hspace{1cm}+C\int\limits_0^T\int\limits_{0}^Le^{-4m_0s\breve\alpha+2m_0s\widehat\alpha}
        	|g_1|^2 dxdt +C_{s,\lambda}\|\rho^*(-\ell^{-2}\varphi1_{\mathcal{O}}+\gamma^{-2}\varphi)\|^2_{L^2(Q)},   	
        \end{aligned}
	\end{equation}
    	for any $\lambda_3:=\max\{\lambda_1,\lambda_2\}\geq C,\,s\geq C$ and $C_{s,\lambda}$ depending on $s,\lambda$.
    	
    	Taking $\ell$ and $\gamma$ large enough, i.e., $\ell,\gamma> C_1T^{14/10}e^{C_2/T^{4/5}}$, where $C_1,\,C_2$
    	are positive constants depending on $a_0,m_0, s$, we can absorb the last term in the right--hand side of
    	\eqref{c.theta.and.varphi} by the left--hand side.
    	
    	Finally, we should estimate the local term concerning $\theta$ in terms of $\varphi$. The idea is to use
    	the first equation of \eqref{carleman_adjoint_system} and the hypothesis $\omega\cap\mathcal{O}_d\neq \emptyset$,
    	where $\omega_0\subset\omega\subset \mathcal{O}_d$. Thus, we introduce an open set $\omega_1\subset\omega$ such that 
    	$\omega_0\subset\omega_1$ and a positive function  $\zeta\in C^4_c(\omega_1)$ such that $\zeta\equiv 1$ in $\omega_0$. 
    	Then, by using \eqref{carleman_adjoint_system} and after several integration by parts in time and space we get:
        \begin{equation*}
    	\begin{aligned}
        	J&=s^7\lambda^8\int\limits_0^T\int\limits_{\omega_0}e^{-2s\alpha-2a_0s\widehat\alpha}\xi^7|\theta|^2dxdt\\
        	&\leq Cs^7\lambda^8\int\limits_0^T\int\limits_{\omega_1}\zeta e^{-2s\alpha-2a_0s\widehat\alpha}\xi^7
        	(-\varphi_{t}+\varphi_{xxxx}+\varphi_{xx}-\overline{u}\varphi_x -g_1)\theta dxdt\\
        	&=C\Biggl( s^7\lambda^8\int\limits_0^T\int\limits_{\omega_1}\zeta (e^{-2s\alpha-2a_0s\widehat\alpha}\xi^7)_t
        	\varphi\theta dxdt\\
        	&+s^7\lambda^8\int\limits_0^T\int\limits_{\omega_1}\zeta e^{-2s\alpha-2a_0s\widehat\alpha}\xi^7
        	(\theta_t+\theta_{xxxx}+\theta_{xx}+(\overline{u}\theta)_x)\varphi dxdt\\
        	&+ s^7\lambda^8\int\limits_0^T\int\limits_{\omega_1}
        	(\zeta e^{-2s\alpha-2a_0s\widehat\alpha}\xi^7)_{xxxx}\varphi\theta dxdt
			+ s^7\lambda^8\int\limits_0^T\int\limits_{\omega_1}
        	(\zeta e^{-2s\alpha-2a_0s\widehat\alpha}\xi^7)_{xxx}\varphi\theta_{x} dxdt\\  
        	&+s^7\lambda^8\int\limits_0^T\int\limits_{\omega_1}
        	(\zeta e^{-2s\alpha-2a_0s\widehat\alpha}\xi^7)_{xx}(\varphi\theta_{xx}+\varphi\theta)dxdt\\
        	&+s^7\lambda^8\int\limits_0^T\int\limits_{\omega_1}
        	(\zeta e^{-2s\alpha-2a_0s\widehat\alpha}\xi^7)_{x}(\varphi\theta_{xxx}+\varphi\theta_{x}
       		+\overline{u}\varphi\theta)dxdt.       	
        	\Biggr)	
        \end{aligned}
	\end{equation*} 
		Now, using the estimates
		\[ |\partial_x^k(\zeta e^{-2s\alpha-2a_0s\widehat\alpha}\xi^7)|\leq Cs^k\lambda^k\xi^{8+\frac{5k}{2}}
		e^{-2s\alpha-2a_0s\widehat\alpha},\quad k=1,\dots,4,\]
		\[|\partial_t(e^{-2s\alpha-2a_0s\widehat\alpha}\xi^7)|\leq CTe^{-2s\alpha-2a_0s\widehat\alpha}\xi^{\frac{21}{2}},\]
		as well as the equation related to $\theta$ (see \eqref{sys.adjoint.coupled}) and the fact that 
		$\overline{u}\in L^\infty(0,T;W^{1,\infty}(0,1))$, the term $J$ can be estimated as follows:
		\begin{equation*}
    	\begin{aligned}
        	J&\leq C\Biggl( s^7\lambda^8\int\limits_0^T\int\limits_{\omega_1}\zeta e^{-2s\alpha-2a_0s\widehat\alpha}
        	\xi^{\frac{21}{2}}|\varphi||\theta| dxdt\\
        	&\quad+s^7\lambda^8\int\limits_0^T\int\limits_{\omega_1}\zeta e^{-2s\alpha-2a_0s\widehat\alpha}\xi^7
        	(g_2 -\ell^{-2}\varphi 1_{\mathcal{O}}+\gamma^{-2}\varphi )\varphi dxdt\\
        	&\quad+ s^{11}\lambda^{12}\int\limits_0^T\int\limits_{\omega_1}
        	e^{-2s\alpha-2a_0s\widehat\alpha}\xi^{18}|\varphi||\theta| dxdt
			+ s^{10}\lambda^{11}\int\limits_0^T\int\limits_{\omega_1}
        	 e^{-2s\alpha-2a_0s\widehat\alpha}\xi^{\frac{31}{2}}|\varphi||\theta_{x}| dxdt\\
        	&\quad+s^{9}\lambda^{10}\int\limits_0^T\int\limits_{\omega_1}
        	 e^{-2s\alpha-2a_0s\widehat\alpha}\xi^{13}|\varphi||\theta_{xx}| dxdt
        	+ s^{8}\lambda^{9}\int\limits_0^T\int\limits_{\omega_1}
        	e^{-2s\alpha-2a_0s\widehat\alpha}\xi^{\frac{21}{2}}|\varphi||\theta_{xxx}| dxdt    	
        	\Biggr),	
        \end{aligned}
		\end{equation*}
		with $C$ depending on $T$ and $\|\overline{u}\|_{L^\infty(0,T;W^{1,\infty}(0,1))}$.
		 
		Taking into account that  $\omega\cap\mathcal{O}=\emptyset$ and applying Young's inequality at each term of the 
		previous inequality, it is easy to deduce the following inequality:
		\begin{equation}\label{laststep.carleman}
    	\begin{aligned}
        	J&\leq \varepsilon I_{1}(e^{-2s\alpha-2a_0s\widehat\alpha},\theta)+
        	C(\varepsilon)s^{15}\lambda^{16}\int\limits_0^T\int\limits_{\omega_1}
        	e^{-2s\breve\alpha-2a_0s\widehat\alpha}(\widehat\xi)^{29}|\varphi|^2dxdt\\
        	&\quad+C(\varepsilon)s^{7}\lambda^{8}\int\limits_0^T\int\limits_{0}^1
        	e^{-2s\alpha-2a_0s\widehat\alpha}\xi^{7}|g_1|^2dxdt
        	+Cs^{7}\lambda^{8}\int\limits_0^T\int\limits_{0}^1
        	e^{-2a_0s\widehat\alpha}|g_2|^2dxdt\\
        	&\quad+Cs^{14}\lambda^{16}\int\limits_0^T\int\limits_{\omega_1}
        	e^{-4s\alpha-2a_0s\widehat\alpha}\xi^{14}|\varphi|^2dxdt,  	
        \end{aligned}
		\end{equation} 
		for every $s\geq C,\,\varepsilon>0,\, \ell>0$ and $\gamma$ large enough.
		
		From the definition of $\breve\alpha,\,\widehat{\alpha}$ and $\widehat{\xi}$ (see \eqref{carleman_weights}),
		the second term  in the right--hand side can estimate both the last term in the right--hand side
		and the first  term in the right--hand side of  \eqref{c.theta.and.varphi}. In fact, the first affirmation holds for every $s\geq1$,
		meanwhile the second one is a consequence of using Lemma \ref{lema.relationship.weights} with 
		$\varepsilon=((m_0-1)/a_{0})-1$ and $M_1=M_2=4/a_{0}$. Therefore, 
		from \eqref{c.theta.and.varphi}  and \eqref{laststep.carleman}, we conclude the proof of Proposition
		\ref{teo.carleman_carleman1}.
	\end{proof}

\subsection{Null controllability of the linearized system}
	In this subsection we will prove the null controllability for the coupled system \eqref{sys.linearized.coupled}
	with a right--hand side with external sources decreasing exponentially to zero when $t$ goes to $T$. In other words, we
	would like to find $h\in L^2(0,T;L^2(\omega))$ such that the solution of	
	\begin{equation}\label{eq.ks.controllineal}
	\begin{cases}
		u_{t}+u_{xxxx}+u_{xx}+\overline{u}u_x+\overline{u}_xu= f_1+
		h1_{\omega}+ (-\ell^{-2}1_{\mathcal{O}}+\gamma^{-2})z & \text{in }Q,\\
		-z_{t}+z_{xxxx}+z_{xx}-(u+\overline{u})z_{x}=f_2+(u-u_{d})1_{\mathcal{O}_{d}} & \text{in }Q,\\
		u(0,t)=u(1,t)=z(0,t)=z(1,t)=0 & \text{on }(0,T),\\
		u_{x}(0,t)=u_{x}(1,t)=z_{x}(0,t)=z_{x}(1,t)=0 & \text{on }(0,T),\\
		u(\cdot,0)=u_{0}(\cdot),\,\, z(\cdot,T)=0& \text{in }(0,1).
	\end{cases}
	\end{equation}
	satisfies
	\begin{equation}\label{cond.zero.linearcase}
		u(\cdot,T)=0\quad \mbox{in } (0,1), 	
	\end{equation}
	
	where the functions $f_1$ and $f_2$ are in appropriate weighted spaces. To this end, let us  first state a Carleman 
    inequality with weight functions not vanishing in $t=0$. Thus, let  $\tilde{\ell}\in C^1([0,T])$ be a positive 
    function in $[0,T)$ such that: 
    \begin{equation*}
    	\tilde{\ell}(t)=T^2/4\quad \forall 
  		t\in [0,T/2]\  \text{ and }\ \tilde{\ell}(t)=t(T-t)\quad \forall  t\in [T/2,T].
  	\end{equation*} 
  	Now, we introduce the following weight functions
	\begin{equation}\label{3.2.carleman.weights}
    \begin{array}{ll}
    	&\beta(x,t) = \dfrac{e^{7\lambda\|\eta\|_{\infty}}-e^{\lambda(2\|\eta\|_{\infty}+\eta(x))}} 
    	{\tilde{\ell}^{2/5}(t)},\quad \tau(x,t)=\dfrac{e^{\lambda(2\|\eta\|_{\infty}+\eta(x))}}{\tilde{\ell}^{2/5}
    	(t)},\\
    	&\widehat\beta(t) = \max\limits_{x\in[0,1]} \beta(x,t),\quad \quad
     	\breve\tau(t) = \min\limits_{x\in[0,1]} \tau(x,t),\\
    	&\breve\beta(t) = \min\limits_{x\in[0,1]} \beta(x,t),\quad 
    	\quad \,\,\,\,\widehat\tau(t) = \max\limits_{x\in[0,1]} \tau(x,t).
    \end{array}
	\end{equation}
	
	\begin{lemma}\label{3.2.lemma.carleman2}
		Let $s$ and $\lambda$ like in Theorem \ref{teo.carleman_carleman1}. Then, there exists a constant $C>0$
    	depending on $s,\lambda,\omega,T$ and  $\|\overline{u}\|_{L^\infty(0,T;W^{1,\infty}(0,1))}$, such that every solution 
    	$(\varphi,\theta)$ of \eqref{sys.adjoint.coupled} satisfies
    	\begin{equation}\label{3.2.ine.carleman2}
    	\begin{array}{ll}
       		\|\varphi(\cdot,0)\|^2_{L^2(0,L)} 
        		+\displaystyle\iint\limits_{Q}e^{-2m_0s\widehat\beta}(\breve\tau)^{7}|\varphi|^2dxdt\\  
        		\hspace{1cm}+\displaystyle\iint\limits_{Q}e^{-2(a_0+1)s\widehat\beta}(\breve\tau)^7|\theta|^2dxdt
        		+\displaystyle\iint\limits_{Q}e^{-4a_0s\widehat\beta}(\breve\tau)^7|\theta|^2dxdt\\
      		\leq C\Biggl(\displaystyle\iint\limits_{Q}e^{-2a_0s\widehat\beta}(\widehat\tau)^{7}|g_1|^2dxdt
        		+\displaystyle\iint\limits_{Q}e^{-2a_0s\widehat\beta}|g_2|^2dxdt
      		+\displaystyle\iint\limits_{\omega\times(0,T)}e^{-2s\breve\beta-2a_0s\widehat\beta}
       		 (\widehat\tau)^{29}|\varphi|^2dxdt\Biggr).
    	\end{array}
    	\end{equation} 
	\end{lemma}
	\begin{proof}
		The proof follows from classical energy estimates and therefore it is omitted. 
		The interested reader might see for instance \cite[Lemma 3.4]{2018-robust-ns} for more details. 
	\end{proof}
	Now, we look for a solution of \eqref{eq.ks.controllineal} in an appropriate weighted functional 
	space. 
	
	Let us define the space $E$ as follows:
	\begin{equation*}
    \begin{array}{ll}
        E:=&\Bigl\{(u,z,h): e^{a_0s\widehat\beta}(\widehat\tau)^{-7/2}u\in L^2(Q), e^{a_0s\widehat\beta}z\in L^2(Q),\\
        &\hspace{5mm}e^{a_0s\widehat\beta+s\breve\beta}(\widehat\tau)^{-29/2}h1_{\omega}\in L^2(Q),\\ 
        &\hspace{5mm}e^{a_0s\widehat\beta}(\widehat\tau)^{-29/2}u\in L^2(0,T;H^2(0,1))\cap L^\infty(0,T;L^2(0,1)),\\
        &\hspace{5mm}e^{a_0s\widehat\beta}(\breve\tau)^{-c_0}z\in L^2(0,T;H^2(0,1))\cap L^\infty(0,T;L^2(0,1)),\,\, c_0\geq \frac{9}{2},\\
        &\hspace{5mm}e^{m_0s\widehat\beta}(\breve\tau)^{-7/2}(u_{t}+u_{xxxx}+u_{xx}+(\overline{u}u)_{x}
        	-h1_{\omega}-(-\ell^{-2}1_{\mathcal{O}}+\gamma^{-2})z)\in L^2(Q),\\
        &\hspace{5mm}e^{2a_0s\widehat\beta}(\breve\tau)^{-7/2}(-z_{t}+z_{xxxx}+z_{xx}-(u+\overline{u})z_{x}
        -(u-u_d)1_{\mathcal{O}_d})\in 
        L^2(Q)	
         \Bigr\}.
    \end{array}
	\end{equation*}

	\begin{proposition}\label{prop.null.control}
 		Assume the hypotheses of Lemma \ref{3.2.lemma.carleman2} and      	
 		\begin{equation}\label{3.3_condition_data}
        		u_0\in L^2(0,1),\, e^{m_0s\widehat\beta}(\breve\tau)^{-7/2}f_1\in L^2(Q),\,\,
        		e^{2a_0s\widehat\beta}(\breve\tau)^{-7/2}f_2\in L^2(Q),     
   		 \end{equation}
   		 \begin{equation}\label{hyp.ud}
   		 	\displaystyle\iint\limits_{\mathcal{O}_d\times(0,T)}\rho^2(t)|u_d|^2 dxdt<+\infty,	
   		 \end{equation}
		where $\rho=\rho(t)$ is a positive function blowing up $t=T$. 
    		Then, there exists a control $h\in L^2(0,T;L^2(\omega))$ such that the associated solution $(u,z,h)$ to
    		\eqref{eq.ks.controllineal} satisfies $(u, z, h)\in E$. 
	\end{proposition}
\begin{proof}[Proof of Proposition \ref{prop.null.control}]
	Let us introduce the following constrained extremal problem:
   	\begin{equation}\label{prop.null.control.extremal.problem}
	 \begin{array}{lll}
	 \inf & \left\{
	\begin{array}{lll}
		&\displaystyle\frac{1}{2}\Bigl(\iint\limits_{Q}e^{2a_0s\widehat\beta}(\widehat\tau)^{-7}|u|^2dxdt
	 +\iint\limits_{Q}e^{2a_0s\widehat\beta}|z|^2dxdt\\
	 &\hspace{1cm}
	 +\displaystyle\iint\limits_{\omega\times(0,T)}e^{2(a_0s\widehat\beta+s\breve\beta)}
	 (\widehat\tau)^{-29}|h|^2dxdt\Bigr)
	\end{array}\right.\\
	 &\mbox{subject to}\,\, h\in L^2(Q),\,\, supp\, h\subset \omega\times (0,T),\,\,\mbox{and} \,\,
	 \eqref{eq.ks.controllineal}.\\
     &\end{array}
	 \end{equation}
	Assume that this problem admits a unique solution $(\widehat u, \widehat z,\widehat h)$. Then, 
	from Lagrange's principle there exists dual variables 
	$(\widehat\varphi, \widehat\theta)$ such that
	\begin{equation}\label{prop.nullcontrol.dualvariables}
	\begin{array}{llll}
	&\widehat u= e^{-2a_0s\widehat\beta}(\widehat\tau)^{7}
	(-\widehat\varphi_{t}+\widehat\varphi_{xxxx}+\widehat\varphi_{xx}-\overline{u}\widehat\varphi_x
      -\widehat\theta1_{\mathcal{O}_d})&\mbox{in}& Q,\\
	&\widehat z=e^{-2a_0s\widehat\beta}
	(\widehat{\theta}_{t}+\widehat{\theta}_{xxxx}+\widehat{\theta}_{xx}+(\overline{u}\widehat{\theta})_x
	-(-\ell^{-2}\chi_{\mathcal{O}}+\gamma^{-2})\widehat\varphi) &\mbox{in}& Q, \\
	&\widehat h=e^{-2(a_0s\widehat\beta+s\breve\beta)}(\widehat\tau)^{29}\widehat\varphi &\mbox{in}& Q,\\
	&\widehat u=\widehat z=0 &\mbox{on}& \{0,1\}\times(0,T).
	\end{array}
	\end{equation}
	Let us now set the space 
	\[P_0:\{(u,z)\in C^4(\overline{Q}): \partial_x^ku(0,t)=\partial_x^ku(1,t)= \partial_x^kz(0,t)=\partial_x^kz(1,t)=0,\quad 
	k=0,1\}.\] 
	as well as the bilinear form $a(\cdot,\cdot)$ over $P_0\times P_0$ defined by:
	\begin{equation*} 
	\begin{array}{ll}
		&
	   \displaystyle\iint\limits_{Q}e^{-2a_0s\widehat\beta}(\widehat\tau)^{7}
		(-\widehat\varphi_{t}+\widehat\varphi_{xxxx}+\widehat\varphi_{xx}-\overline{u}\widehat\varphi_x
      -\widehat\theta1_{\mathcal{O}_d})
	    (-w_{t}+w_{xxxx}+w_{xx}-\overline{u}w_x-z1_{\mathcal{O}_d})\, dxdt\\
	 &+\displaystyle\iint\limits_{Q}e^{-2a_0s\widehat\beta}(\widehat{\theta}_{t}+\widehat{\theta}_{xxxx}+\widehat{\theta}_{xx}+(\overline{u}\widehat{\theta})_x-(-\ell^{-2}\chi_{\mathcal{O}}+\gamma^{-2})\widehat\varphi)	    
		 (z_{t}+z_{xxxx}+z_{xx}+(\overline{u}z)_x)\\
	  &-\displaystyle\iint\limits_{Q}e^{-2a_0s\widehat\beta}(\widehat{\theta}_{t}+\widehat{\theta}_{xxxx}+\widehat{\theta}_{xx}+(\overline{u}\widehat{\theta})_x-(-\ell^{-2}\chi_{\mathcal{O}}+\gamma^{-2})
	   	 (-\ell^{-2}\chi_{\mathcal{O}}+\gamma^{-2}) w)\, dxdt\\
	 &+\displaystyle\iint\limits_{\omega\times(0,T)}e^{-2(a_0s\widehat\beta+s\breve\beta)}(\widehat\tau)^{29}
	 	\widehat\varphi w\,dxdt =:a((\widehat\varphi,\widehat\theta),(w,z)),
	\end{array}
	\end{equation*}
    for every $(w,z)\in P_0$,  and a linear form
    \begin{equation}\label{prop.null.control.def.lineal.form}
        \langle G, (w,z)\rangle:=\displaystyle\iint\limits_{Q}f_1\cdot w\, dxdt+
        \displaystyle\iint\limits_{Q}f_2\cdot z\, dxdt
        +\displaystyle\int\limits_{\Omega}u_0(\cdot)\cdot w(\cdot,0)\, dx. 
    \end{equation}
    Taking into account these definitions, one can see that, if the functions $\widehat u,\widehat z$ and $\widehat h$ solve
    \eqref{prop.null.control.extremal.problem}, we must have for every  $(w,z)$ in $P_0$
    \begin{equation}\label{prop.null.control.identity.bilineal.lineal}
        a((\widehat\varphi,\widehat\theta),(w,z))=\langle G, (w,z)\rangle. 
    \end{equation}
    Note that Carleman inequality \eqref{3.2.ine.carleman2} holds for all $(w,z)\in P_0$. 
    Consequently, 
    \begin{equation}\label{prop.nullcontrol.estimate1.bilinealform}
    \begin{array}{ll}
   		&\|w(\cdot,0)\|^2_{L^2(0,L)} 
        +\displaystyle\iint\limits_{Q}e^{-2m_0s\widehat\beta}(\breve\tau)^{7}|w|^2dxdt\\  
        &\hspace{1cm}+\displaystyle\iint\limits_{Q}e^{-2(a_0+1)s\widehat\beta}(\breve\tau)^7|z|^2dxdt
        +\displaystyle\iint\limits_{Q}e^{-4a_0s\widehat\beta}(\breve\tau)^7|z|^2dxdt\leq C a((w,z),(w,z)),
    \end{array}     
    \end{equation}
    for every $(w,z)\in P_0$.\\
    Therefore, it is easy to prove that $a(\cdot,\cdot):P_0\times P_0\longmapsto\mathbb{R}$  is symmetric, 
    definite positive bilinear form
    on $P_0$, so that, by defining $P$ as the completion of $P_0$ for the norm induced by $a(\cdot,\cdot)$ it implies
    that $a(\cdot,\cdot)$ is well--defined, continuous and again definite positive on $P$. In addition, from 
    Carleman inequality \eqref{3.2.ine.carleman2}, the hypothesis over the functions $f_1$ and $f_2$
    (see \eqref{3.3_condition_data}), and \eqref{prop.nullcontrol.estimate1.bilinealform}, the linear form 
    $(w,z)\longmapsto \langle G, (w,z)\rangle$  is well--defined and continuous on $P$. Indeed, 
    thanks to the relation among $a_, m_0$, see \eqref{constantes}, for every $(w,z)\in P$ we have
    \begin{equation*}
    \begin{aligned}
    	 \langle G, (w,z)\rangle  
        &\leq  \|e^{(a_0+1)s\widehat\beta}(\breve\tau)^{-7/2}f_1\|_{L^2(Q)}
        \|e^{-(a_0+1)s\widehat\beta}(\breve\tau)^{7/2}w\|_{L^2(Q)}\\
        &\hspace{0.5cm}+
        \|e^{m_0s\widehat\beta}(\breve\tau)^{-7/2}f_2\|_{L^2(Q)}\|e^{-m_0s\widehat\beta}(\breve\tau)^{7/2}z\|_{L^2(Q)}
        +\|u_0\|_{L^2(0,1)}\|w(0)\|_{L^2(0,1)}\\
        &\leq  \|e^{m_0s\widehat\beta}(\breve\tau)^{-7/2}f_1\|_{L^2(Q)}
        \|e^{-(a_0+1)s\widehat\beta}(\breve\tau)^{7/2}w\|_{L^2(Q)}\\
        & \hspace{0.5cm}
        +\|e^{2a_0s\widehat\beta}(\breve\tau)^{-7/2}f_2\|_{L^2(Q)}\|e^{-m_0s\widehat\beta}(\breve\tau)^{7/2}z\|_{L^2(Q)}
        +\|u_0\|_{L^2(0,1)}\|w(0)\|_{L^2(0,1)}.
    \end{aligned}
    \end{equation*} 
    Using \eqref{prop.nullcontrol.estimate1.bilinealform} and the density of $P_0$ in $P$, we find
    \begin{equation*}
        \langle G, (w,z)\rangle
        \leq C\Bigl( \|e^{m_0s\widehat\beta}(\breve\tau)^{-7/2}f_1\|_{L^2(Q)}
        +\|e^{2a_0s\widehat\beta}(\breve\tau)^{-7/2}f_2\|_{L^2(Q)}+\|y_0\|_{L^2(0,1)}\Bigr)\|(w,z)\|_{P}.
    \end{equation*}
    Hence, from Lax--Milgram's Lemma, there exists a unique 
    $(\widehat\varphi,\widehat\theta)\in P$ satisfying 
    \begin{equation}\label{prop.nullcontrol.estimate2.identity}
        a((\widehat\varphi,\widehat\theta),(w,z))=\langle G, (w,z)\rangle, \quad \forall  (w,z)\in P.     
    \end{equation}
    Let us set $(\widehat u, \widehat z,\widehat h)$ like in \eqref{prop.nullcontrol.dualvariables} and remark that 
    $(\widehat u,\widehat z,\widehat h)$ verifies
    \begin{equation}\label{aux56789}
    \begin{array}{ll}
      a((\widehat\varphi,\widehat\theta),
      (\widehat\varphi,\widehat\theta))= 
     \displaystyle\iint\limits_{Q}e^{2a_0s\widehat\beta}(\widehat\tau)^{-7}|\widehat u|^2dxdt
      \\
     +\displaystyle\iint\limits_{Q}e^{2a_0s\widehat\beta}|\widehat z|^2dxdt
     +\displaystyle\iint\limits_{\omega\times(0,T)}e^{2(a_0s\widehat\beta+s\breve\beta)}
	 (\widehat\tau)^{-29}|\widehat h|^2dxdt < +\infty.  
    \end{array}
    \end{equation}
    
    Let us prove that $(\widehat u,\widehat z)$ is the weak solution 
    of the coupled system \eqref{eq.ks.controllineal} for $h=\widehat h$. In fact, we introduce 
    the (weak) solution $(\tilde u,\tilde z)$ to the coupled system
    \begin{equation}\label{prop.nullcontrol.system.aux1}
	\begin{cases}
		 \tilde u_{t}+ \tilde u_{xxxx}+ \tilde u_{xx}+\overline{u} \tilde u_x+\overline{u}_x \tilde u= f_1+
		h1_{\omega}+ (-\ell^{-2}1_{\mathcal{O}}+\gamma^{-2}) \tilde z & \text{in }Q,\\
		- \tilde z_{t}+ \tilde z_{xxxx}+ \tilde z_{xx}-( \tilde u+\overline{u}) \tilde z_{x}=f_2
		+( \tilde u-u_{d})1_{\mathcal{O}_{d}} & \text{in }Q,\\
		 \tilde u(0,t)= \tilde u(1,t)= \tilde z(0,t)= \tilde z(1,t)=0 & \text{on }(0,T),\\
		 \tilde u_{x}(0,t)= \tilde u_{x}(1,t)= \tilde z_{x}(0,t)= \tilde z_{x}(1,t)=0 & \text{on }(0,T),\\
		 \tilde u(\cdot,0)= \tilde u_{0}(\cdot),\,\,  \tilde z(\cdot,T)=0& \text{in }(0,1).
	\end{cases}
	\end{equation}
    Clearly, $(\tilde u,\tilde z)$ is the unique solution of \eqref{prop.nullcontrol.system.aux1} defined
    by transposition. This means that, for every $(a,b)\in L^2(Q)^{2}$, 
    \begin{equation}\label{prop.nullcontrol.transposition1}
        \langle (\tilde u,\tilde z), (a,b)\rangle_{L^2(Q)^2}=\langle u_0,\varphi(0)\rangle_{L^2((0,1))}
        +\langle (f_1+\widehat h1_{\omega},f_2), (\varphi,\theta)\rangle_{L^2(Q)^2},
    \end{equation}
    where $(\varphi,\theta)$ is the solution to
    \begin{equation}\label{prop.nullcontrol.system.aux2}
    \left\{
        \begin{array}{lll}
        L^*(\varphi,\theta)=(a,b) &\text{ in } Q,\\
        \varphi(0,t)=\varphi(1,t)=\theta(0,t)=\theta(1,t)=0 & \text{on }(0,T),\\
		\varphi_{x}(0,t)=\varphi_{x}(1,t)=\theta_{x}(0,t)=\theta_{x}(1,t)=0 & \text{on }(0,T),\\
        \varphi(\cdot,T)= \varphi_T(\cdot),\, \theta(\cdot,0)=0& \text{in }(0,1)        
        \end{array}
    \right.
    \end{equation}
    and $L^*$ is the adjoint operator of $L$  given by:
    \[L(\tilde u,\tilde z)=(L_1(\tilde u,\tilde z), L_2(\tilde u,\tilde z)),\]
	with
	\[L_1(\tilde u,\tilde z)=\tilde u_{t}+ \tilde u_{xxxx}+ \tilde u_{xx}+\overline{u} \tilde u_x+\overline{u}_x 
    \tilde u- (-\ell^{-2}1_{\mathcal{O}}+\gamma^{-2})\tilde z\]
    and
    \[L_2(\tilde u,\tilde z)=- \tilde z_{t}+ \tilde z_{xxxx}+ \tilde z_{xx}-( \tilde u+\overline{u}) \tilde z_{x}
	-( \tilde u-u_{d})1_{\mathcal{O}_{d}}.\] 
	
    \noindent From \eqref{prop.nullcontrol.dualvariables} and \eqref{prop.null.control.identity.bilineal.lineal},
    we see that $(\widehat u,\widehat z)$ also satisfies \eqref{prop.nullcontrol.transposition1}. Then 
    $(\widehat u,\widehat z)=(\tilde u,\tilde z)$ is the weak solution to  
    \eqref{prop.nullcontrol.system.aux1}.
    
    Finally, we must see that $(\widehat u,\widehat z,\widehat h)\in E$. Observe that from \eqref{aux56789}, 
    we have that 
	\begin{equation*}
	e^{a_0s\widehat\beta}(\widehat\tau)^{-7/2}\widehat u,\,\, e^{a_0s\beta^*}\widehat z,\,\,\,
   	e^{(a_0s\widehat\beta+s\breve\beta)}(\widehat\tau)^{-29/2}\widehat h 1_{\omega}\in L^2(Q)
   	\end{equation*}
    and by hypothesis \eqref{3.3_condition_data}
	\begin{equation*}
		e^{m_0s\widehat\beta}(\breve\tau)^{-7/2}f_1\in L^2(Q)\quad\mbox{and}\quad 
		e^{2a_0s\widehat\beta}(\breve\tau)^{-7/2}f_2\in L^2(Q).
	\end{equation*}
	Thus, it only remains to check that
	\begin{equation*}
		e^{a_0s\widehat\beta}(\widehat\tau)^{-29/2}\widehat u, \,\, e^{a_0s\widehat\beta}(\breve\tau)^{-c_0}\widehat z
		\in L^2(0,T;H^2(0,1))\cap L^\infty(0,T;L^2(0,1)),
	\end{equation*}	
	where $c_0\geq \frac{9}{2}.$
	\begin{enumerate}
	\item [a)] We define the functions
	    \begin{equation*}
	    \begin{array}{llll}
	    	u^*:=e^{a_0s\widehat\beta}(\widehat\tau)^{-29/2}\widehat u, & & 
	    	z^*:=e^{a_0s\widehat\beta}(\breve\tau)^{-c_0}\widehat z\\
	    \end{array}
	    \end{equation*}
        and
        \begin{equation*}
        \begin{array}{llll}
        f^*_1:=e^{a_0s\widehat\beta}(\widehat\tau)^{-29/2}(f_1+h1_{\omega}), & &
        z^{**}:=e^{a_0s\widehat\beta}(\widehat\tau)^{-29/2}(-\ell^{-2}\chi_{\mathcal{O}}+\gamma^{-2})\widehat z\\
        f^*_2:=e^{a_0s\widehat\beta}(\breve\tau)^{-c_0} f_2, & & 
        u^{**}:=e^{a_0s\widehat\beta}(\breve\tau)^{-c_0}(u-u_d)\chi_{\mathcal{O}_d}. 
        \end{array}
        \end{equation*}
        Then $(u^*,z^*)$ satisfies:
        \begin{equation}\label{prop.nullcontrol.system.aux3}
        \begin{cases}
		 	u^*_{t}+ u^*_{xxxx}+u^*_{xx}+\overline{u}u^*_x+\overline{u}_xu^*= f_1^*+z^{**}
		 	+(e^{a_0s\widehat\beta}(\widehat\tau)^{-29/2})'\widehat u
			& \text{in }Q,\\
			- z^*_{t}+z^*_{xxxx}+ z^*_{xx}-(u^*+\overline{u})z^*_{x}=f_2^*+u^{**}
			 +(e^{a_0s\widehat\beta}(\breve\tau)^{-c_0})'\widehat z & \text{in }Q,\\
			 u^*(0,t)=u^*(1,t)= z^*(0,t)=z^*(1,t)=0 & \text{on }(0,T),\\
			 u^*_{x}(0,t)= u^*_{x}(1,t)= z^*_{x}(0,t)=z^*_{x}(1,t)=0 & \text{on }(0,T),\\
			 u^*(\cdot,0)=e^{a_0s\widehat\beta(0)}(\widehat\tau(0))^{-29/2}u_{0}(\cdot),\,\,
			 z^*(\cdot,T)=0& \text{in }(0,1).
		\end{cases}
    	\end{equation}
    
    \item [b)] Now, we prove that the right--hand side of the main equations in \eqref{prop.nullcontrol.system.aux3}
        is in $L^2(Q)$.
        \begin{itemize}
        \item $|e^{a_0s\widehat\beta}(\widehat\tau)^{-29/2}f_1|\leq 
                \leq Ce^{m_0s\widehat\beta}|\breve\tau|^{-7/2}|f_1|.$
        \item $|e^{a_0s\widehat\beta}(\widehat\tau)^{-29/2}h1_{\omega}|
                \leq Ce^{(a_0s\widehat\beta+s\breve\beta)}|\widehat\tau|^{-29/2}|h|1_{\omega}.$
        \item $|z^{**}|=|e^{a_0s\widehat\beta}(\widehat\tau)^{-29/2}(-\ell^{-2}\chi_{\mathcal{O}}+\gamma^{-2})\widehat z|
        		\leq Ce^{a_0s\widehat\beta}|\widehat{z}|.$
		\item $|(e^{a_0s\widehat\beta}(\widehat\tau)^{-29/2})'\widehat u|
				\leq C e^{a_0s\widehat\beta}|\widehat\tau|^{-7/2}|\widehat u|$.
        \item $|f_2^{*}|=|e^{a_0s\widehat\beta}(\breve\tau)^{-c_0} f_2|\leq 
        		Ce^{(a_0+1)s\widehat\beta}|\breve\tau|^{-c_0}|f_2|$.
        \item $|(e^{a_0s\widehat\beta}(\breve\tau)^{-c_0})'\widehat z|\leq Ce^{a_0s\widehat\beta}|\widehat{z}|$.
        \item Note that $u^{**}\in L^2(Q)$ thanks to the hypothesis \eqref{hyp.ud} and the fact that
        	$c_0\geq \frac{9}{2}$. Indeed,
        \begin{equation*}
        	\begin{array}{ll}
 				|u^{**}|&=|e^{a_0s\widehat\beta}(\breve\tau)^{-c_0}(\widehat u-u_d)\chi_{\mathcal{O}_d}|\\
        		&\leq Ce^{a_0s\widehat\beta}|\widehat{\tau}|^{-9/2}|\widehat{u}|+Ce^{a_0s\widehat\beta}
        		|\breve\tau|^{-c_0}|u_d|.
 			\end{array}
        \end{equation*}
		\end{itemize}
    \end{enumerate}
    Therefore, from  $a), b)$ and taking  $u_0\in L^2(0,1)$, we have 
    $u^*,z^*\in  L^2(0,T;H^2(0,1))\cap L^\infty(0,T;L^2(0,1))$ (see lemma \ref{lema.Nico.Mauro}).
    This concludes the proof of Proposition \ref{prop.null.control}.
\end{proof}

\subsection{Local exact controllability to trajectories}	
	In this subsection we give the proof of Theorem \ref{teo2.RobustStackelbergContKS} through fixed point arguments.
	In order to apply the obtained results in the previous sections we consider the following change of variable.
	Let us set $w=u-\overline{u}$ and $w_d=u_d-\overline{u}$, where $\overline{u}=$ solves \eqref{intro.target.u}. 
	It is easy to verify that $w$ satisfies
	\begin{equation}\label{eq.ks.coupled.nonlinear.lastsystem}
	\begin{cases}
		w_{t}+w_{xxxx}+w_{xx}+(\overline{u}w)_x+w_xw= h1_{\omega}+(-\ell^{-2}1_{\mathcal{O}}+\gamma^{-2})z 
		& \text{in }Q,\\
		-z_{t}+z_{xxxx}+z_{xx}-(w+\overline{u})z_{x}=(w-w_{d})1_{\mathcal{O}_{d}} & \text{in }Q,\\
		w(0,t)=w(1,t)=z(0,t)=z(1,t)=0 & \text{on }(0,T),\\
		w_{x}(0,t)=w_{x}(1,t)=z_{x}(0,t)=z_{x}(1,t)=0 & \text{on }(0,T),\\
		w(\cdot,0)=(u_{0}-\overline{u}_0)(\cdot),\,\, z(\cdot,T)=0& \text{in }(0,1).
	\end{cases}
	\end{equation}
	Observe that these changes reduce our problem to a local null controllability for the solution $w$ of the nonlinear
	problem \eqref{eq.ks.coupled.nonlinear.lastsystem}.i.e., we are looking a control function $h$ and an associated
	solution  $(w,z)$ of \eqref{eq.ks.coupled.nonlinear.lastsystem} such that $w(\cdot,T)=0$ in $(0,1)$. 	
	To this end, we will apply an inverse function theorem of the Lyuternik's type \cite{1982Hamilton},	 
  	which will allow us to complete the proof of theorem \ref{teo2.RobustStackelbergContKS}. More precisely, we will 
  	use the following theorem.
    \begin{teo}\label{teo_inverse_mapping}
		Suppose that $\mathcal{B}_1,\mathcal{B}_2$ are Banach spaces and 
		\begin{equation*}\mathcal{A}:\mathcal{B}_1\to \mathcal{B}_2\end{equation*}
		is a continuously differentiable map. We assume that for $b_1^0\in \mathcal{B}_1, b_2^0\in \mathcal{B}_2$ the
		equality
		\begin{equation}\label{b_1^0}
			\mathcal{A}(b_1^0)=b_2^0
		\end{equation}
		holds and $\mathcal{A}'(b_1^0):\mathcal{B}_1\to \mathcal{B}_2$ is an epimorphism. Then there exists 
		$\delta >0$ such that for any $b_2\in \mathcal{B}_2$ which satisfies the condition 
		\begin{equation*}\|b_2^0-b_2\|_{\mathcal{B}_2}<\delta	\end{equation*}
		there exists a solution $b_1\in \mathcal{B}_1$ of the equation \begin{equation*}\mathcal{A}(b_1)=b_2.\end{equation*}
	\end{teo}
	Before starting the proof of Theorem \ref{teo2.RobustStackelbergContKS}, small data must be considered in 
	our analysis. Thus, we impose that
	\begin{equation}\label{ine.smalldata}
		\|e^{m_0s\widehat\beta}(\breve\tau)^{-7/2}f_1\|_{L^2(Q)}	+\|e^{2a_0s\widehat\beta}(\breve\tau)^{-7/2}f_2\|_{L^2(Q)}
		+\|w(\cdot,0)\|_{L^2(0,1)}+\displaystyle\iint\limits_{\mathcal{O}_d\times(0,T)}\rho^2(t)|w_d|^2 dxdt \leq \delta,
	\end{equation}
	where $\delta$ is a small positive number and $\rho=\rho(t)$ is a positive function blowing up $t=T$.
	\begin{proof}[Proof of Theorem \ref{teo2.RobustStackelbergContKS}]
		We apply Theorem \ref{teo_inverse_mapping} for the spaces $\mathcal{B}_1:=E$ and 
	\begin{equation*}\mathcal{B}_2:=\{(f_1,f_2,w_0)\in X_1\times X_2\times L^2(0,1):
	 f_1,f_2, w_0\,\,
	 \mbox{satisfy} \,\,\,\eqref{ine.smalldata}\},\end{equation*}
	where 
	$X_1:= L^2(e^{m_0s\widehat\beta}(\breve\tau)^{-7/2}(0,T);L^2(0,1))$ and 
	\begin{equation*}X_2:=L^2(e^{2a_0s\widehat\beta}(\breve\tau)^{-7/2}(0,T);L^2(0,1)).\end{equation*}
	We define the operator $\mathcal{A}$ by the formula
	\begin{equation*}
	\begin{array}{l}
		\mathcal{A}(w,z,h) :=\Bigl(w_{t}+w_{xxxx}+w_{xx}+(\overline{u}w)_x+w_xw-h1_{\omega}
		-(-\ell^{-2}1_{\mathcal{O}}+\gamma^{-2})z,\\
		\hspace{2.5cm}-z_{t}+z_{xxxx}+z_{xx}-(w+\overline{u})z_{x}-(w-w_d)1_{\mathcal{O}_d}, w(\cdot,0)\Bigr),
	\end{array}
	\end{equation*}
	for every $(w,z,h)\in \mathcal{B}_1$. 
	
	Let us see that $\mathcal{A}$ is of class $C^1(\mathcal{B}_1, \mathcal{B}_2)$. Indeed, notice that all the
	terms in $\mathcal{A}$ are linear, except for $ww_x$ and $wz_x$. Thus, we only
	have to check that these nonlinear terms are well--defined and depend continuously on the data. Thus, we will
	prove that the bilinear operator 
	\begin{equation*}((w^1,z^1),(w^2,z^2))\longmapsto w^1w_x^2	
	\end{equation*}
	is continuous from $Z\times Z$ to $X_1$, and the bilinear form 
	\begin{equation*}
		((w^1,z^1),(w^2,z^2))\longmapsto w^1z_x^2
	\end{equation*}
	is continuous from $Z\times Z$ to $X_2$, and where 
	\[Z:=\Bigl\{y: e^{a_0s\widehat\beta}(\widehat\tau)^{-c_1}y\in L^2(0,T;H^2(0,1))\cap L^\infty(0,T;L^2(0,1)),\,\, c_1>\frac{29}{2}\Bigr\}.\]
	
	In fact, for any $w^1,w^2\in X_1$ we have
	\begin{equation*}
	\begin{array}{ll}
		\|w^1w^2_x\|_{X_1}&=\|e^{m_0s\widehat\beta}(\breve\tau)^{-7/2}w^1w_x^2\|_{L^2(Q)}\\
		&\leq C \|e^{a_0s\widehat\beta}(\widehat\tau)^{-7/4}w^1 e^{a_0s\widehat\beta}
		(\widehat\tau)^{-7/4}w^2_x\|_{L^2(Q)}\\
		&\leq C\|e^{a_0s\widehat\beta}(\widehat\tau)^{-29/2}w^1e^{a_0s\widehat\beta}(\widehat\tau)^{-29/4}w^2_x\|_{L^2(Q)}\\
		&\leq C\|e^{a_0s\widehat\beta}(\widehat\tau)^{-29/2}w^1\|_{L^\infty(0,T;L^2(0,1))}
			\|e^{a_0s\widehat\beta}(\widehat\tau)^{-29/4}w^2_x\|_{L^\infty(0,T;L^2(0,1))}\\
		&\leq C\|w^1\|_Z\|w^2\|_Z.	
	\end{array}
	\end{equation*}
	On the other hand, for $c_1> \frac{29}{2}$ and any $w^1,z^2\in X_2$ , we have
	\begin{equation*}
	\begin{array}{ll}
		\|w^1z^2_x\|_{X_2}&=\|e^{a_0s\widehat\beta}(\widehat\tau)^{-7/4}w^1 
		e^{a_0s\widehat\beta}z^2\|_{L^2(Q)}\\
		&\leq C\|e^{a_0s\widehat\beta}(\widehat\tau)^{-29/2}w^1e^{a_0s\widehat\beta}(\widehat\tau)^{-c_0}z^2_x\|_{L^2(Q)}
		\\
		&\leq C\|w^1\|_Z\|z^2\|_Z.	
	\end{array}
	\end{equation*}

	Notice that $\mathcal{A}'(0,0,0):\mathcal{B}_1\to \mathcal{B}_2$ is given by
	\[
	(w_{t}+w_{xxxx}+w_{xx}+(\overline{u}w)_x-h1_{\omega}-(-\ell^{-2}1_{\mathcal{O}}+\gamma^{-2})z,
	-z_{t}+z_{xxxx}+z_{xx}-\overline{u}z_{x}-(w-w_d)1_{\mathcal{O}_d}, w(\cdot,0)),\]
	for all\,\,$ (w,z,h)\in \mathcal{B}_1.$
	In virtue of Proposition \ref{prop.null.control}, this functional satisfies 
	$Im (\mathcal{A}'(0,0,0))=\mathcal{B}_2$.
	
	Let $b_1^0=(0,0,0)$ and $b_2^0=(0,0,w_0)$. Then equation \eqref{b_1^0} holds.
	So all necessary conditions to apply Theorem \ref{teo_inverse_mapping} are 
	fulfilled. Therefore there exists a positive number $\delta$ such that, if $(w_0,w_d)$ satisfy the inequality
	\eqref{ine.smalldata}, we can find a control $h\in L^2(0,T;L^2(\omega))$ 
	and an associated solution $(w,z)$ to \eqref{eq.ks.coupled.nonlinear.lastsystem} satisfying $w(\cdot,T)=0$ in
 	$(0,1)$. 
	This finishes the proof of Theorem \ref{teo2.RobustStackelbergContKS}.
\end{proof}

\subsection{Numerical framework} 
	This  section is devoted to present numerical experiments on the RSC problem. 
	which was proved at the above section. 
	In other words, we show approximations to Problem \ref{p3.robustandstackelberg} and
	thereby to Problem \ref{p2.stackelberg} (without disturbance signal, i.e., $\psi\equiv 0$). Our approach given in subsection 
	\ref{section.numerical.method} will be used and completed for tackling these problems. We focus our attention in solving 
	the following extremal problem:
	
	\begin{equation}\label{eq.extremalproblem.robuststackelberg}
	 \begin{aligned}
	 &\inf \displaystyle\frac{1}{2}\iint\limits_{\omega\times(0,T)}|h|^2dxdt,\quad
	 \mbox{subject to}\,\, h\in L^2(Q),\,\, supp\, h\subset \omega\times (0,T),\,\,\mbox{and}\\
     	&
        \begin{cases}
        u_{t}+u_{xxxx}+u_{xx}+uu_{x}=h1_{\omega}+\ell^{-2}1_{\mathcal{O}}z-\gamma^{-2}z &\mbox{ in } (0,1)\times(0,T),\\
        -z_{t}+z_{xxxx}+z_{xx}-uz_{x}=(u-u_d)1_{\mathcal{O}_d}&\mbox{ in } (0,1)\times(0,T),\\
        u(0,t)=u(1,t)=u_{x}(0,t)=u_x(1,t)=0 &\mbox{ on }(0,T),\\
        z(0,t)=z(1,t)=z_{x}(0,t)=z_x(1,t)=0 &\mbox{ on }(0,T),\\
		u(\cdot,0)=u_{0}(\cdot),\,\,  u(\cdot,T)=\overline{u}(\cdot,T),\,\, z(\cdot, T)=0& \mbox{ in }(0,1).
        \end{cases}
        	&\end{aligned}
	 \end{equation}
	 
	 Using optimal control techniques, we consider a  regularization to the functional given in 
	\eqref{eq.extremalproblem.robuststackelberg} as follows:
	\begin{equation}\label{eq:Robustfunctional}
		{\mathcal G}(h)=\frac{\beta}{2}\displaystyle\int\limits_{0}^1|u(x,T)-\overline{u}(x,T)|^{2}dx
		+\frac{1}{2}\displaystyle\iint\limits_{{\mathcal \omega}\times(0,T)}|h|^{2}dxdt,\quad \beta>0.
	\end{equation}
	To optimize  \eqref{eq:Robustfunctional}, a Lagrangian formulation might be developed. Thus, the coupled 
	adjoint system $(\varphi^1,\varphi^2)$ associated to 
	\eqref{eq.extremalproblem.robuststackelberg} is given by
	
	\begin{equation}\label{eq.AdjoindCouplesystem}
        \begin{cases}
		-\varphi^1_{t}+\varphi^1_{xxxx}+\varphi^1_{xx}-u\varphi^1_{x}=-\varphi^21_{\mathcal{O}_{d}}+z_{x}\varphi^2 
		 &\mbox{ in } (0,1)\times(0,T),\\
		\varphi^2_{t}+\varphi^2_{xxxx}+\varphi^2_{xx}+(u\varphi^2)_{x}=\ell^{-2}\varphi^{1}1_{\mathcal{O}}
		-\gamma^{-2}\varphi^1 &\mbox{ in } (0,1)\times(0,T),\\
		\varphi^1(0,t)=\varphi^1(1,t)=\varphi^1_{x}(0,t)=\varphi^1_{x}(1,t)=0  &\mbox{ on }(0,T),\\
		\varphi^2(0,t)=\varphi^2(1,t)=\varphi^2_{x}(0,t)=\varphi^2_{x}(1,t)=0  &\mbox{ on }(0,T),\\
		\varphi^1(x,T)=-\beta(u(\cdot,T)-\overline{u}(\cdot,T)),\,\,\varphi^2(x,0)=0 & \text{ in }(0,1).
        \end{cases}
	\end{equation}
	A simple computation allows us to deduce the following expression:
	\[\frac{\partial {\mathcal G}}{\partial h}(h)=h-\varphi^1(u(h),z(h)).\]
	
	In the algorithm \ref{algorithm2.rsc} we describe the required steps for solving  the problem
	\eqref{eq.extremalproblem.robuststackelberg}. Some remarks on this algorithm are given below.  
	
	\begin{algorithm}[ht]\label{algorithm2.rsc}
	\SetAlgoLined{} 
	\KwIn {Initialize a $h^0$  on $t\in [0,T]$.}
	For $n\geq 0$. 
	
	\textbf{STEP1:} Compute: $u^{n},z^{n}$ solution of the system
		\begin{equation}\label{primal.system.u.z}
		\begin{array}{ll}\left\{
        		\begin{array}{llll}
       		 u_{t}+u_{xxxx}+u_{xx}+uu_{x}=h^{n}1_{\omega}+\ell^{-2}1_{\mathcal{O}}z-\gamma^{-2}z &\text{ in }& (0,1)\times(0,T),\\
        		-z_{t}+z_{xxxx}+z_{xx}-uz_{x}=(u-u_d)1_{\mathcal{O}_d}&\text{ in }& (0,1)\times(0,T),\\
        		u(0,t)=u(1,t)=u_{x}(0,t)=u_x(1,t)=0 &\text{ on }&(0,T),\\
        		z(0,t)=z(1,t)=z_{x}(0,t)=z_x(1,t)=0 &\text{ on }&(0,T),\\
		u(\cdot,0)=u_{0}(\cdot),\,\,  u(\cdot,T)=\overline{u}(\cdot,T),\,\, z(\cdot, T)=0& \text{ in }&(0,1).
        		\end{array}
        		\right.
		\end{array}
		\end{equation}
		
	\textbf{STEP2:} Compute $\varphi^{1,n},\,\varphi^{2,n}$ using the system
		\begin{equation}\label{dual.system.varphi1varphi2}\left\{
       		 \begin{array}{llll}
		-\varphi^{1}_{t}+\varphi^1_{xxxx}+\varphi^1_{xx}-u^n\varphi^1_{x}=-\varphi^21_{\mathcal{O}_{d}}+z^n_{x}\varphi^2 
		 &\text{ in }& (0,1)\times(0,T),\\
		\varphi^2_{t}+\varphi^2_{xxxx}+\varphi^2_{xx}+(u^n\varphi^2)_{x}=\ell^{-2}\varphi^{1}1_{\mathcal{O}}
		-\gamma^{-2}\varphi^1 &\text{ in }& (0,1)\times(0,T),\\
		\varphi^1(0,t)=\varphi^1(1,t)=\varphi^1_{x}(0,t)=\varphi^1_{x}(1,t)=0  &\text{ on }&(0,T),\\
		\varphi^2(0,t)=\varphi^2(1,t)=\varphi^2_{x}(0,t)=\varphi^2_{x}(1,t)=0  &\text{ on }&(0,T),\\
		\varphi^1(x,T)=-\beta(u^n(\cdot,T)-\overline{u}(\cdot,T)),\,\,\varphi^2(x,0)=0 & \text{ in }&(0,1)
        		\end{array}\right.
		\end{equation}
		
	\textbf{STEP3:} Compute 
		\[\frac{\partial {\mathcal G}}{\partial h}(h^n)=h^{n}-\varphi^{1,n}(u(h^n),z(h^n)).\]	
			
	\textbf{STEP4:} Find $\alpha\in \mathbb{R}^{+}$ such that
		\[\min\limits_{\alpha\in \mathbb{R}^{+}}\mathcal{G}\Bigl(h^n-\alpha\frac{\partial {\mathcal G}}{\partial h}(h^n)
		\Bigr).\] 
		
	\textbf{STEP5:} Set 
		\[h^{n+1}=h^n-\alpha\frac{\partial {\mathcal G}}{\partial h}(h^n).\] 
		
	\textbf{STEP6:} If $\|\frac{\partial {\mathcal G}}{\partial h}(h^n)\|_{L^2(Q)}\leq tol$, set $h=h^{n+1}$. Otherwise, return to STEP1.
	
	\caption{Robust Stackelberg controllability algorithm to the problem \eqref{eq.extremalproblem.robuststackelberg}}
	\end{algorithm}
	
	\begin{Obs}\smallskip\
		\begin{itemize}
		\item  We highlight that the algorithm 1 associated to the robust control problem
		 must be used in the STEP1 of algorithm \ref{algorithm2.rsc} to find a numerical solution of \eqref{primal.system.u.z}.
		   
		\item   Respect to the STEP 2, observe that \eqref{dual.system.varphi1varphi2} corresponds to a linear model, 
		whose implementation is carried out with the conjugate gradient method (CGM) for coupled system, which is 
		inspired in the book \cite{book.glowinski-lions}. Due to the linearity of  \eqref{dual.system.varphi1varphi2}, 
		we mention that the CGM shows
		a better convergence than method proposed in \cite{2002iterative-tachim,2000-bewleytemamziane}. 
		However, this analysis is omitted in this paper because it is far away of our main goals.  
		
		\item  On the STEP 4, we have used the nonlinear gradient conjugate method \cite{1998Lucquin}. 
		As mentioned in subsection \ref{section.numerical.method}, a FreeFem algorithm
		 on nonlinear optimization is used for the implementation.
		 \end{itemize}
	\end{Obs}
	
	Now, we present some numerical examples related to the robust Stackelberg controllability problem given in 
	\eqref{eq.extremalproblem.robuststackelberg}. We set the parameter $\beta=10^{-7}$  meanwhile the trajectory
	is the function $\overline{u}(x,t)=0$.  Again,  $\Omega=(-30,30)$. In Figures \ref{fig1.rsc.emptyset},  \ref{fig2.rsc.emptyset}
	we take configurations in which the intersection of the
	 sub--domain $\omega$  (for the leader control) 
	and the sub--domain $\mathcal{O}$ (for the follower control) is the empty set. Additionally, we bring a numeric response 
	to the case $\omega\cap \mathcal{O}\neq \emptyset$, see Figures  \ref{fig1.rsc.nonemptyset},\ref{fig2.rsc.nonemptyset}.
	Recall that in Theorem \ref{teo2.RobustStackelbergContKS},
	the geometrical condition $\omega\cap \mathcal{O}=\emptyset$ is a sufficient hypothesis, and used in section on
	Carleman estimates.
	   
	\begin{figure}[ht!]
	\begin{center}
		\includegraphics[scale=0.179]{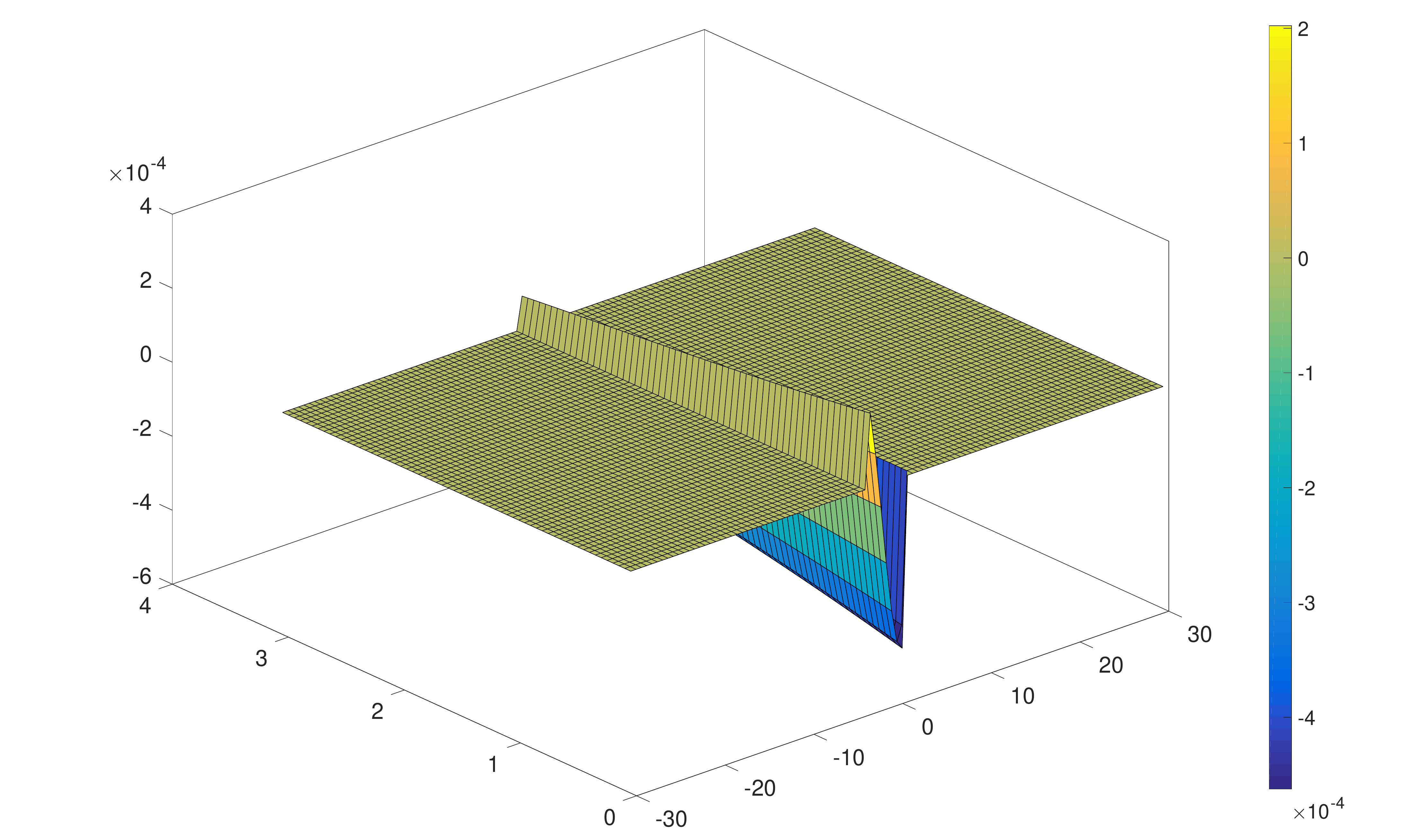} 
		\includegraphics[scale=0.179]{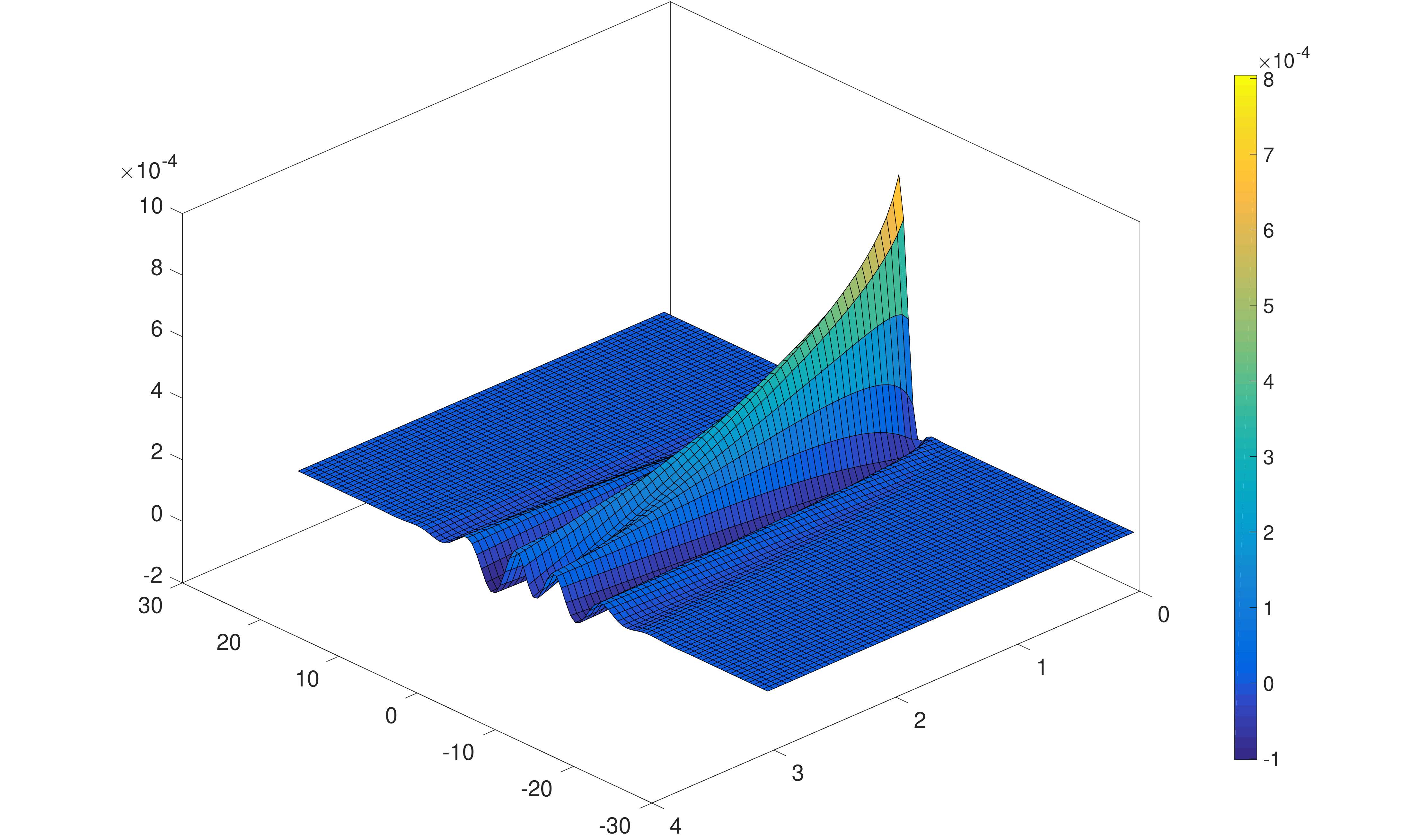}   			
	\end{center}
	\par
	\caption{Lider function (left) and state (right). $T=3s, N=100, \Delta t=2\times 10^{-2}$,  
	$\ell=\gamma=40$. Domains $\omega=(-3,1)$ and $\mathcal{O}=(2,5)$, initial datum $u_0(x)=10^{-3}\exp{(-x^2)}$.}
	\label{fig1.rsc.emptyset}
	\end{figure}	
	
	\begin{figure}[ht!]
	\begin{center}
		\includegraphics[scale=0.17]{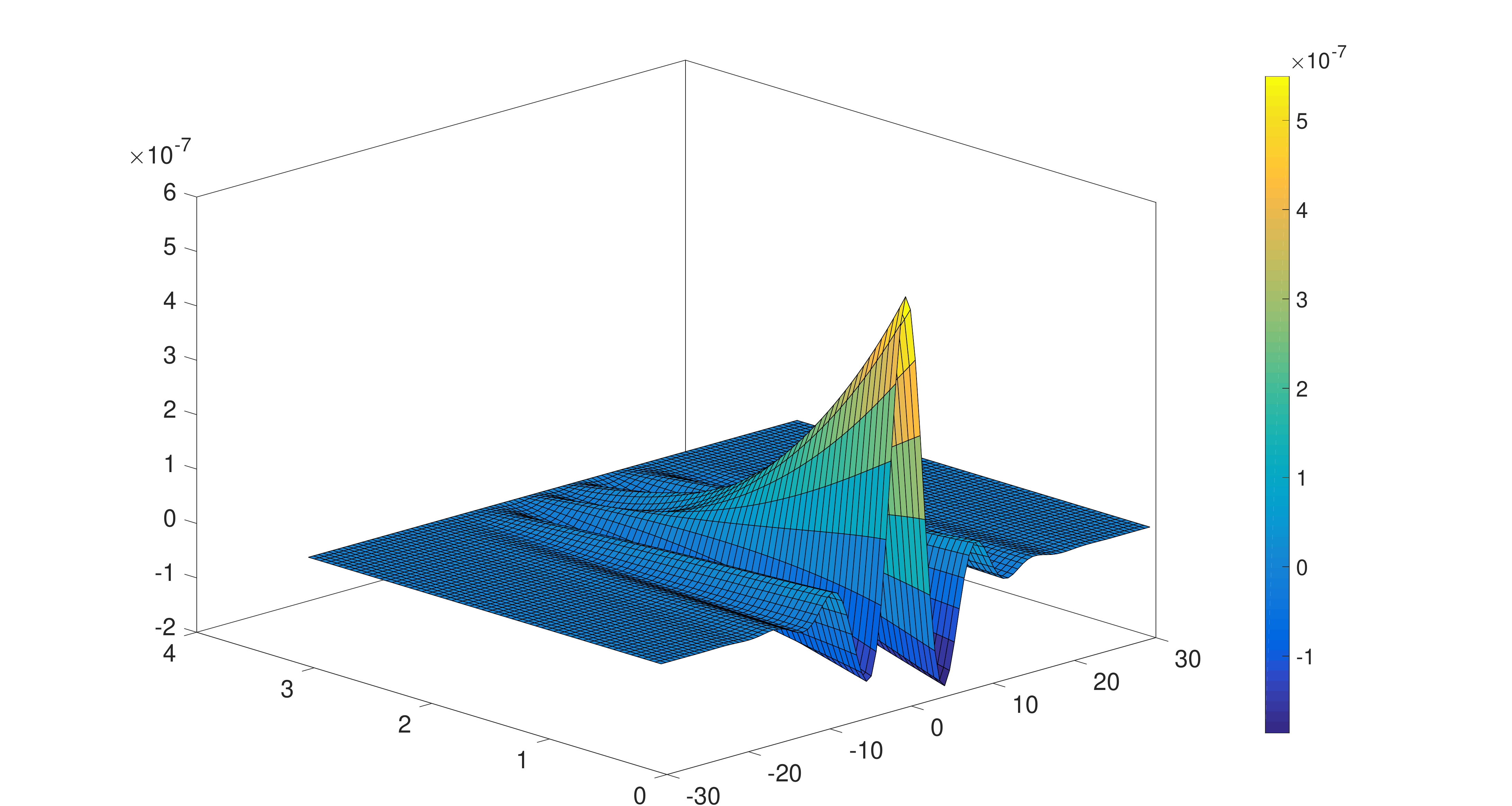} 
		\includegraphics[scale=0.17]{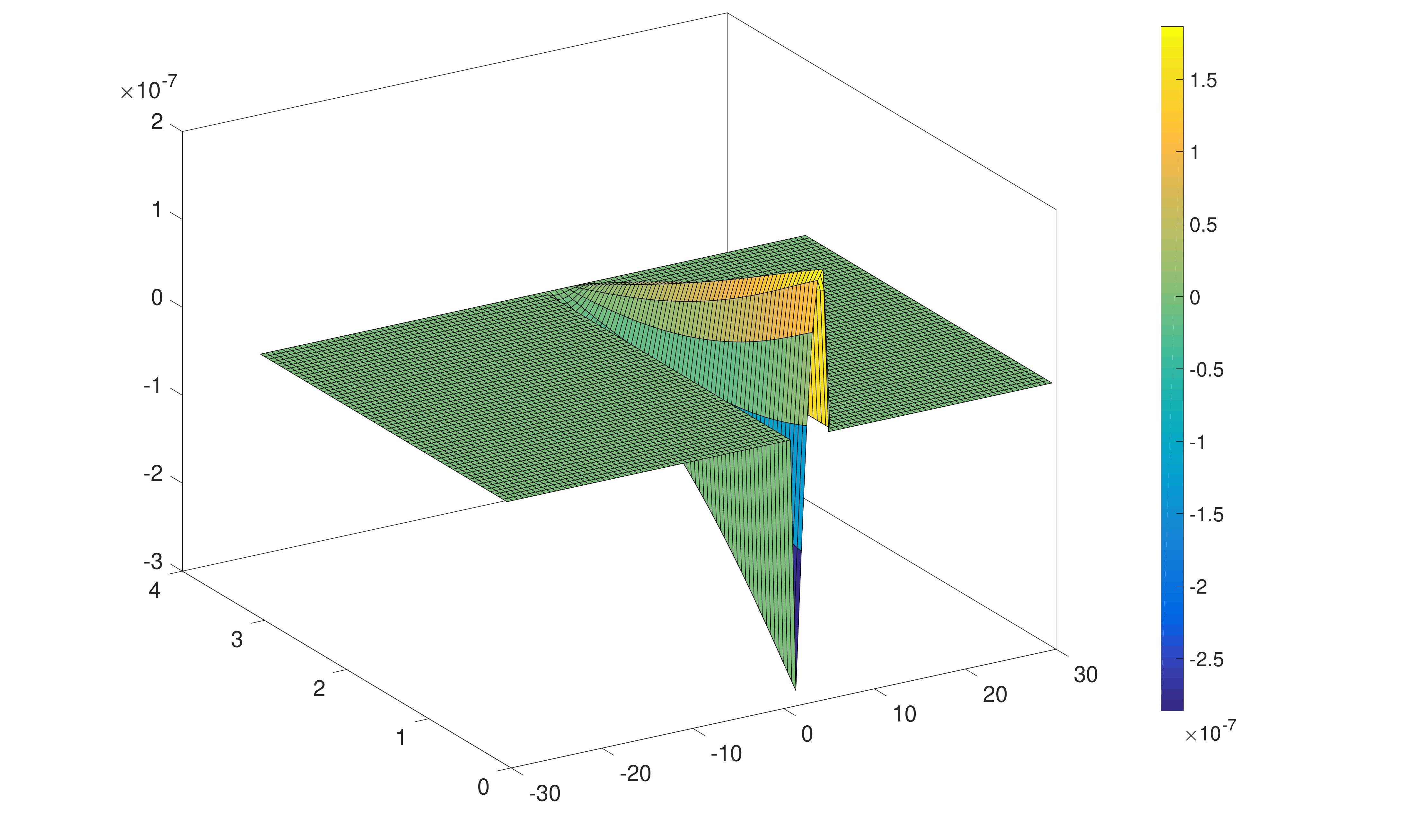}   			
	\end{center}
	\par
	\caption{Disturbance (left) function and follower (right). $T=3s, N=100, \Delta t=2\times 10^{-2}$, 
	$\ell=\gamma=40$. Domains $\omega=(-3,1)$ and $\mathcal{O}=(2,5)$, initial datum $u_0(x)=10^{-3}\exp{(-x^2)}$.}
	\label{fig2.rsc.emptyset}
	\end{figure}	

	\begin{figure}[ht!]
	\begin{center}
		\includegraphics[scale=0.16]{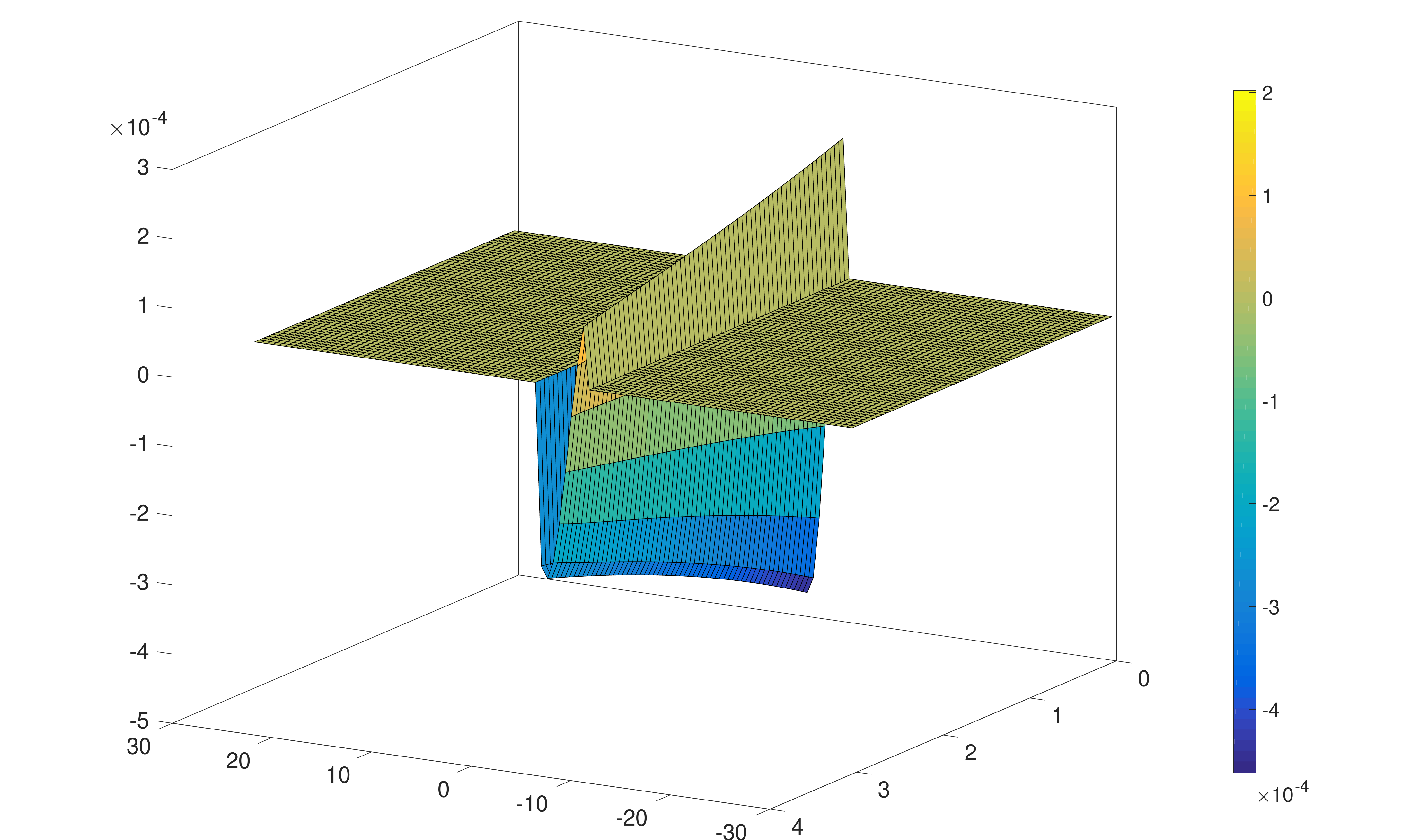} 
		\includegraphics[scale=0.16]{Image/RSC/state-rsc-ex2-eps-converted-to.pdf}   			
	\end{center}
	\par
	\caption{ Lider (left) and state (right).  $T=3s, N=100, \Delta t=2\times 10^{-2}$, 
	$\ell=\gamma=40$. Domains $\omega=(-3,1)$ and $\mathcal{O}=(-1,3)$, initial datum $u_0(x)=10^{-3}\exp{(-x^2)}$.}
	\label{fig1.rsc.nonemptyset}
	\end{figure}	
	
	\begin{figure}[ht!]
	\begin{center}
		\includegraphics[scale=0.16]{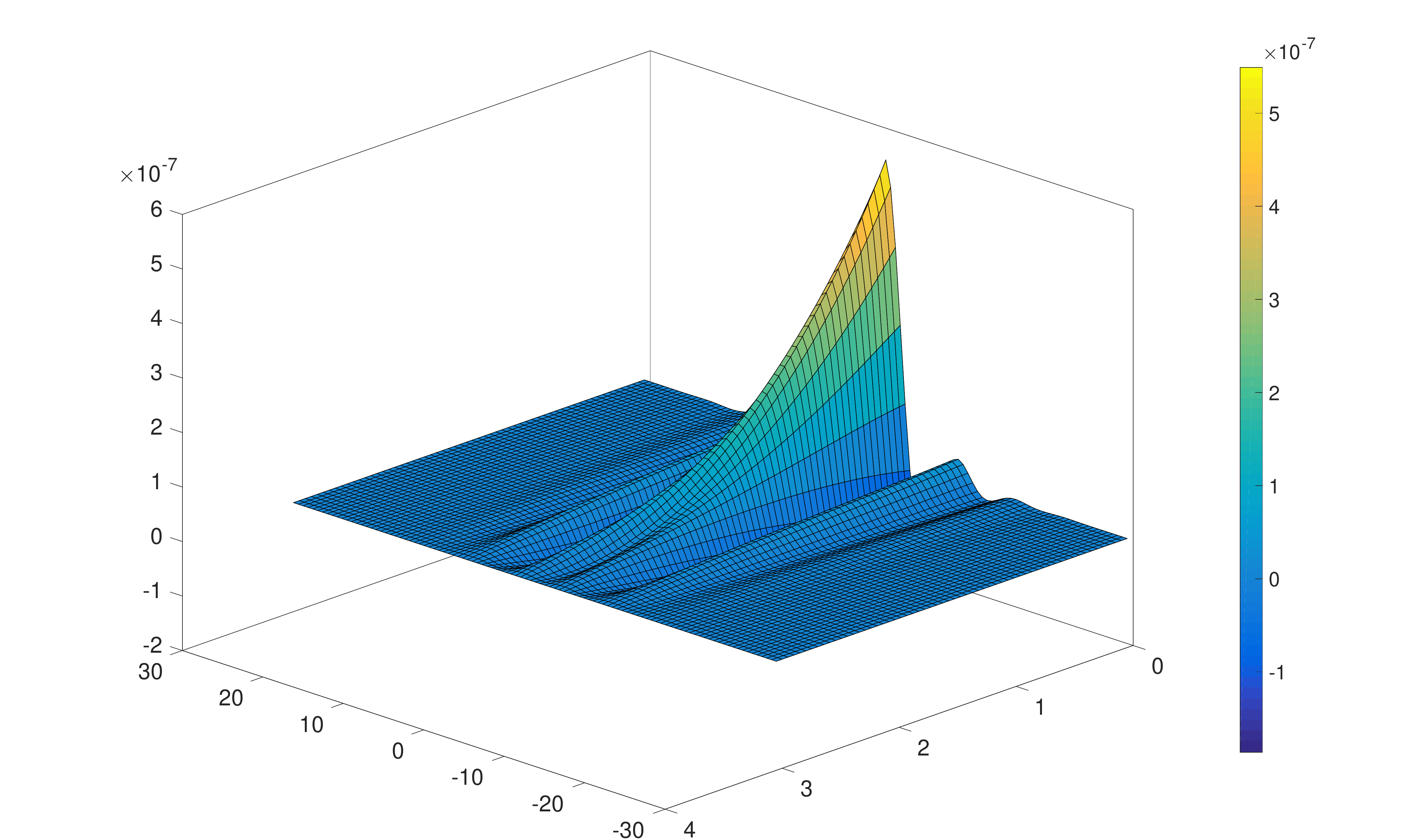} 
		\includegraphics[scale=0.16]{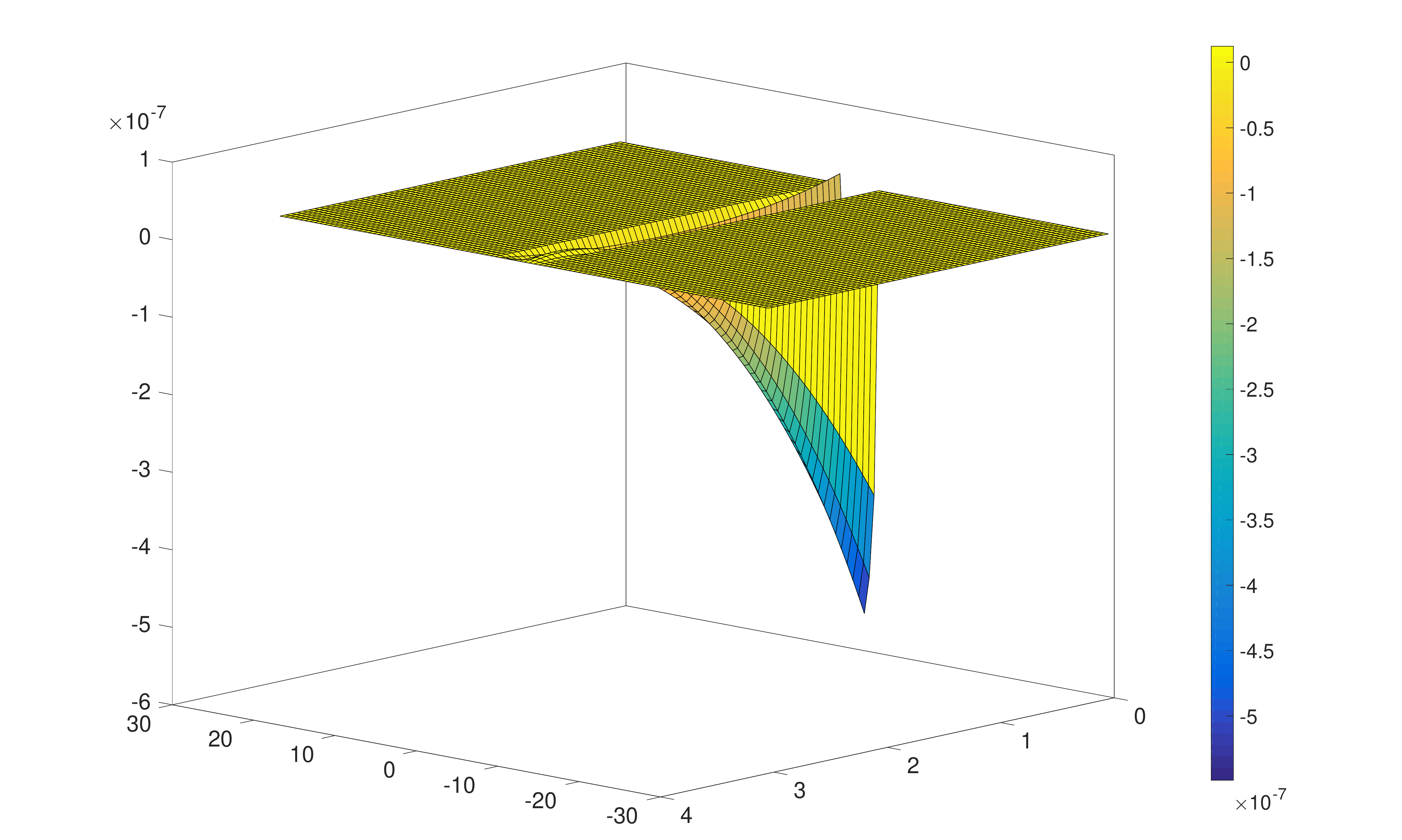}   			
	\end{center}
	\par
	\caption{Disturbance function (left) and follower (right). $T=3s, N=100, \Delta t=2\times 10^{-2}$,  
	$\ell=\gamma=40$. Domains $\omega=(-3,1)$ and $\mathcal{O}=(-1,3)$, initial datum $u_0(x)=10^{-3}\exp{(-x^2)}$.}
	\label{fig2.rsc.nonemptyset}
	\end{figure}
	
\section{Comments and open problems}

\noindent	In this paper, we have considered the robust Stackelberg controllability problem  for the KS equations. However, there are several comments and open questions that are worth mentioning.
	\begin{itemize}
	\item The robustness of a nonlinear KS equation posed in a bounded domain is achieved by using
		optimal control theory that allows us to guarantee the existence, uniqueness and also characterization of a 
		saddle point for the system \eqref{eq.ks}. 
		
		To our knowledge, this paper contains the first numerical description concerning the robustness process for the KS
		equation. Due to the high--order in space (i.e., fourth order derivates), an appropriate change of variable  
		is used to implement low--order
		finite elements, more precisely, $\mathbb{P}_1$--type  Lagrange elements, meanwhile, a 
		$\theta$--scheme/Bashforth method was  used for the
		time discretization. Although this paper  does not present an exhaustive numerical analysis of our method, 
		since it is far way
		of the main goals, several configurations to the  time--space discretization displayed good results for the error
		in the $L^2$--norm and $L^\infty$--norm, see Table  \ref{table.errorKS1}.
		Besides, from the algorithms presented in  \cite{2000-bewleytemamziane,2002iterative-tachim} for the 
		Navier--Stokes system, we proposed new iterative schemes of constructing the 
		ascent and descent directions.  
		
	\item In this paper we present the robust stackelberg controllability (RSC) problem for the KS equation, that is, 
		once we have obtained the robust pair $(\bar{v},\bar\psi)$, we proved the exact controllability to 
		the trajectories for the leader control $h$. A direct consequence is the Stackelberg strategy
		between the leader $h$ and the follower $v$.  
		From a theoretical perspective, the main novelties  are new Carleman inequalities and its relationship with the 
		robustness parameters  $\ell$ and $\gamma$, see Proposition \ref{teo.carleman_carleman1} and Proposition
		\ref{prop.null.control}. 
		
		Numerically, some approximate solutions to the RSC problem
		are presented by implementing the algorithm \ref{algorithm2.rsc}.  
		In addition, by considering the geometrical condition between the leader and the follower, i.e., 
		$\omega\cap \mathcal{O}=\emptyset$, the 
		numerical examples allow us to visualize that such an condition (sufficient condition in Theorem \ref{teo2.RobustStackelbergContKS})
		could be removed in some sense, that means, $\omega\cap \mathcal{O}\neq\emptyset$ could proceed by using perhaps 
		another strategy. 
		
	\item  It would be interesting to study the case of a cooperative game 
	between the leader control $h$ and the follower control $v$, that is, to analyze the case in which 
	\[(\mbox{leader domain } h\,\,\cap\,\,\mbox{follower control } v)\neq\emptyset.\]
	
	\item Another problem consists in the possibility of extending the notion of a robust control to several inputs, for example, 
	instead of a control $v$ and a disturbance $\psi$, to take several control $v_1,\cdots, v_N$ and several 
	disturbance signals $\psi_1,\cdots, \psi_M$, $M, N\in\mathbb{N}$. 	
	
	\item In the same spirit of this paper, the extension of our main results (Theorem \ref{teo1.saddlepoint}
	and Theorem \ref{teo2.RobustStackelbergContKS}) to its model in higher dimensional, that is, biharmonic--type
	equations could be interesting. 
	
	\item Finally, efficient numerical schemes always presenting a challenge to overcome in each problem.
		
	\end{itemize}

\section{Appendix}\label{appendix1}
	In this appendix we mention the well--posedness results we used in this paper for both linearized and 
	nonlinear equations. First, in order to consider external sources with lower regularity in space, we define 
	solution by transposition for the linearized KS equation. Let us define	
	\[\mathcal{Z}:=C([0,T];H_0^2(0,1))\cap L^2(0,T;H^4(0,1))\cap L^\infty(0,T;W^{1,\infty}(0,1)).\]
	Hence, let $\overline{y}_0\in H_0^2(0,1)$ and let $\overline{y}\in \mathcal{Z}$ be a solution of the KS equation 
	\begin{equation}\label{apendix.target.y}
	\begin{cases}
		\overline{y}_{t}+\overline{y}_{xxxx}+\overline{y}_{xx}+\overline{y}\overline{y}_x
		=0 & \text{in }Q,\\
		\overline{y}(0,t)=\overline{y}(1,t)=\overline{y}_{x}(0,t)=\overline{y}_{x}(1,t)=0  & \text{on }(0,T),\\
		\overline{y}(\cdot,0)=\overline{y}_0& \text{in }(0,1).
	\end{cases}
	\end{equation}
	First, we consider the following linearized system:
	\begin{equation}\label{apendix.linearized.y}
	\begin{cases}
		y_{t}+y_{xxxx}+y_{xx}+\overline{y}y_x+\overline{y}_{x}y
		=f & \text{in }Q,\\
		y(0,t)=y(1,t)=y_{x}(0,t)=y_{x}(1,t)=0  & \text{on }(0,T),\\
		y(\cdot,0)=y_0& \text{in }(0,1).
	\end{cases}
	\end{equation}
	
	Now, from \cite[Section 2]{2011cerpamercado} we have the following definition.
	\begin{definition}\label{def.weakweaksol}
	Let $y_0\in H^{-2}_0(0,1)$ and $f\in L^1(0,T;W^{-1,1}(0,1))$. A solution of the system \eqref{apendix.linearized.y}
	is a solution $y\in L^2(Q)$ such that for any $g\in L^2(Q)$,
	\begin{equation}
		\iint\limits_{Q}y(x,t)g(x,t)dxdt=\langle y_0,w(0,\cdot)\rangle_{H^{-2}(0,1),H^2(0,1)}
		+\langle f,w\rangle_{L^1(0,T;W^{-1,1}(0,1)), L^\infty(0,T;W^{1,\infty}(0,1))},	
	\end{equation}
	where $w=w(x,t)\in \mathcal{Z}$ is the solution to
	\begin{equation}\label{apendix.linearized.w}
	\begin{cases}
		-w_{t}+w_{xxxx}+w_{xx}-\overline{y}w_x=g & \text{in }Q,\\
		w(0,t)=w(1,t)=w_{x}(0,t)=w_{x}(1,t)=0  & \text{on }(0,T),\\
		w(\cdot,T)=0& \text{in }(0,1).
	\end{cases}
	\end{equation}
	\end{definition}
	\begin{lemma}\label{apendix.lemma.weak}
		Assume $\overline{y}\in\mathcal{Z}$. Then, for any $y_0\in H^{-2}_0(0,1)$ and $f\in L^1(0,T;W^{-1,1}(0,1))$,
		the linearized system \eqref{apendix.linearized.y} admits a unique solution 
		$y\in C([0,T];H^{-2}(0,1))\cap L^2(0,T;L^2(0,1))$.
	\end{lemma}
	
	\begin{remark}
		Note that both the regularity for the solution $w$ of \eqref{apendix.linearized.w} and an exhaustive
		proof of Lemma \ref{apendix.lemma.weak} can be obtained in an easy way
		from \cite[Proposition 2.1]{2011cerpamercado} and \cite{2001TemamChangbing2}. Due to that, we have omitted 
		those details here. 
	\end{remark}
	
	The following lemma shows regularity results for \eqref{apendix.linearized.y} by considering data 
	$(f,y_0)$ belong to more regular spaces like $L^2(Q)\times L^2(0,1)$ and 
	$L^2(Q)\times H^2_0(0,1)$. We invite to the reader to review \cite{2001TemamChangbing2} and 
	\cite[Appendix A]{2018carreno-santos-stackelberg} for more details.
	\begin{lemma}\label{apendix.lema.strong.linear}
		Assume $\overline{y}\in\mathcal{Z}$. 
	\begin{enumerate}
	\item [a)] For any $y_0\in L^{2}(0,1)$	 and $f\in L^2(Q)$,
		the linearized system \eqref{apendix.linearized.y} admits a unique solution 
		$y\in C([0,T];L^2(0,1))\cap L^2(0,T;H^2(0,1))$ with $y_t\in L^2(0,T;H^{-2}(0,1))$. 
		Moreover, there exists a positive constant 
		$C=C(\|\overline{y}\|_{L^\infty(0,T;W^{1,\infty}(0,1)).})$ such that	
		\begin{equation}
		\|y\|_{C([0,T];L^2(0,1))\cap L^2(0,T;H^2(0,1))}\leq 
		C\Bigl(\|f\|_{L^2(Q)}+\|y_0\|_{L^{2}_0(0,1)}\Bigr).	
		\end{equation}

	\item [b)] 	For  $(y_0,f)\in H_0^{2}(0,1)\times L^2(Q)$,
		the linearized system \eqref{apendix.linearized.y} admits a unique solution 
		$y$ in $C([0,T];H_0^2(0,1))\cap L^2(0,T;H^4(0,1))$. Moreover, 
		\begin{equation}\label{reg.high.linear}
		\|y\|_{C([0,T];H_0^2(0,1))\cap L^2(0,T;H^4(0,1))}\leq 
		C\Bigl(\|f\|_{L^2(Q)}+\|y_0\|_{H^{2}_0(0,1)}\Bigr),	
		\end{equation}
		where $C$ is a positive constant depending on $\|\overline{y}\|_{L^\infty(0,T;W^{1,\infty}(0,1)).}$	
	\end{enumerate}
	\end{lemma}

	Now, we mention a result for coupled fourth--order system. Its proof can be found in 
	\cite[Appendix A]{2018carreno-santos-stackelberg}. Let us consider the system:
	\begin{equation}\label{apendix.coupled.linearized}
	\begin{cases}
		y_{t}+y_{xxxx}+y_{xx}+\overline{y}y_x+\overline{y}_xy= g_1+
		 -\mu^{-2}z & \text{in }Q,\\
		-z_{t}+z_{xxxx}+z_{xx}-(y+\overline{y})z_{x}=g_2& \text{in }Q,\\
		y(0,t)=y(1,t)=z(0,t)=z(1,t)=0 & \text{on }(0,T),\\
		y_{x}(0,t)=y_{x}(1,t)=z_{x}(0,t)=z_{x}(1,t)=0 & \text{on }(0,T),\\
		y(\cdot,0)=y_{0}(\cdot),\,\, z(\cdot,T)=0& \text{in }(0,1).
	\end{cases}
	\end{equation}
	\begin{lemma}\label{lema.Nico.Mauro}	
		Assume that $\overline{y}\in L^\infty(Q)$. Then, there exists $\mu_0>0$ such that for every $\mu\geq \mu_0$,
		any $g_1,g_2\in L^2(Q)$ and any $y_0\in L^2(0,1)$, $(y,z)$ is the unique solution of 
		\eqref{apendix.coupled.linearized} in the space
		 \[(y,z)\in (L^\infty(0,T;L^{2}(0,1))\cap L^2(0,T;H^2(0,1)))^2.\]
	\end{lemma}

	The next step in this appendix corresponds to the nonlinear problem

	\begin{equation}\label{apendix.nonlinearized.y}
	\begin{cases}
		y_{t}+y_{xxxx}+y_{xx}+\overline{y}y_x+\overline{y}_{x}y+yy_x
		=f & \text{in }Q,\\
		y(0,t)=y(1,t)=y_{x}(0,t)=y_{x}(1,t)=0  & \text{on }(0,T),\\
		y(\cdot,0)=y_0& \text{in }(0,1).
	\end{cases}
	\end{equation}	
	
	\begin{lemma}\label{lema.nonlinearregularity}\smallskip\
	\begin{enumerate}
	\item [a)] Assume $\overline{y}\in L^\infty(0,T;W^{1,\infty}(0,1))$. There exists $\delta>0$ such that for any $(f,y_0)\in L^2(Q)\times L^2(0,1)$ satisfying
		\[\|y_0\|_{L^2(0,1)}+\|f\|_{L^2(Q)}\leq \delta \]
		problem \eqref{apendix.nonlinearized.y} has a unique solution in $C([0,T];L^2(0,1))\cap L^2(0,T;H^2(0,1))$.
	\item [b)]	Let $\overline{y}=0$ in \eqref{apendix.nonlinearized.y}. There exists $\delta>0$ such that for any 
		$(f,y_0)\in L^2(Q)\times H_0^2(0,1)$ satisfying
		\[\|y_0\|_{H_0^2(0,1)}+\|f\|_{L^2(Q)}\leq \delta \]
		problem \eqref{apendix.nonlinearized.y} has a unique solution in $C([0,T];H_0^2(0,1))\cap L^2(0,T;H^4(0,1))$.
	\end{enumerate}
	\end{lemma}
	
	\begin{remark}
		Although in \cite[Theorem A.4]{2018carreno-santos-stackelberg} the authors have proved the first part of the
		above result by considering $f\in L^1(0,T;L^2(0,1))$ instead of $f\in L^2(0,T;L^2(0,1))$, their arguments 
		can be easily adapted  for proving this part of lemma \ref{lema.nonlinearregularity}. The second part can be 
		obtained from \cite{2001TemamChangbing2}. For this reason, we have omitted the
		proof of Lemma \ref{lema.nonlinearregularity}. 
	\end{remark}
	\begin{remark}\label{obs.final}
		Observe that, from lemma \ref{lema.nonlinearregularity}, part $b)$, and the fact that $H_0^2(0,1)$ embeds continuously into
		$W^{1,\infty}(0,1)$, it follows that $y\in L^\infty(0,T;W^{1,\infty}(0,1))$.	
	\end{remark}
						

\end{document}